\documentclass[dvips,preprint]{imsart}

\RequirePackage{amsthm,amsmath}
\RequirePackage[OT1]{fontenc}
\RequirePackage[numbers]{natbib}
\RequirePackage[colorlinks,citecolor=blue,urlcolor=blue]{hyperref}

\arxiv{arXiv:1304.4169}
\usepackage{amsmath, amsthm} 
\usepackage{amsfonts,amssymb}
\usepackage{graphicx}
\usepackage{epstopdf}
\usepackage{color}
\usepackage{enumerate}

\newcommand{\PP}{\mathbb{P}}
\newcommand{\EE}{\mathbb{E}}

\newcommand{\QQ}{\mathbb{Q}}
\newcommand{\RR}{\mathbb{R}}

\newcommand{\NN}{\mathbb{N}}
\newcommand{\NNN}{\mathcal{N}}
\newcommand{\B}{\mathcal{B}}

\newcommand{\LL}{\mathcal{L}}
\newcommand{\PPP}{\mathcal{P}}

\newcommand{\e}{\epsilon}
\newcommand{\oo}{\omega}

\newcommand{\al}{\alpha}
\newcommand{\G}{\Gamma}
\newcommand{\D}{\mathcal{D}}
\newcommand{\Y}{\mathcal{Y}}
\newcommand{\F}{\mathcal{F}}
\newcommand{\GG}{\mathcal{G}}
\newcommand{\X}{\mathcal{X}}

\startlocaldefs
\numberwithin{equation}{section}
\theoremstyle{plain}
\newtheorem{thm}{Theorem}[section]
\newtheorem{cor}[thm]{Corollary}
\newtheorem{lem}[thm]{Lemma}
\newtheorem{df}[thm]{Definition}
\newtheorem{prop}[thm]{Property}
\newtheorem{rk}[thm]{Remark}
\endlocaldefs

\begin{document}
\setlength{\parindent}{0mm}
\begin{frontmatter}
\title{Law of Large Numbers for Semi-Markov inhomogeneous Random Evolutions on Banach spaces}
\begin{aug}
\author{\snm{N. Vadori}\thanksref{t2}}
\and
\author{ \snm{A. Swishchuk}\thanksref{t3}}

\thankstext{t2}{Corresponding author: 2500 University Drive NW, Calgary, AB, Canada T2N 1N4, email: nvadori@ucalgary.ca}
\thankstext{t3}{2500 University Drive NW, Calgary, AB, Canada T2N 1N4, email: aswish@ucalgary.ca}
\runauthor{N. Vadori and A. Swishchuk}

\affiliation{University of Calgary\\ Alberta, Canada}
\end{aug}

\begin{abstract}
Using backward propagators, we construct inhomogeneous Random Evolutions on Banach spaces driven by (uniformly ergodic) Semi-Markov processes. After studying some of their properties (measurability, continuity, integral representation), we establish a Law of Large Numbers for such inhomogeneous Random Evolutions, and more precisely their weak convergence - in the Skorohod space $D$ - to an inhomogeneous semigroup. A martingale characterization of these inhomogeneous Random Evolutions is also obtained. Finally, we present applications to inhomogeneous L\'{e}vy Random Evolutions.
\end{abstract}
\end{frontmatter}


\textbf{Subject Classification AMS 2010:} Primary: 60F17, 60F05; Secondary: 60B10, 60G44.\\

\textbf{Notations to be used throughout the paper:}
\begin{itemize}
\item $\NN$, $\NN^*$: non negative integers, positive integers
\item $\RR$, $\RR^*$, $\RR^+$, $\RR^{+*}$:  real numbers, non zero real numbers, non negative real numbers, positive real numbers
\item $L(E,F)$(resp. $\B(E,F)$): the space of linear (resp. bounded linear) operators $E \to F$
\item $Bor(E)$: Borel sigma-algebra on $E$
\item $C^n(E)$ (resp. $C_b^n(E)$, $C_0^n(E)$): continuous (resp. continuous bounded, continuous vanishing at infinity) functions $E \to \RR$ such that the $n^{th}$ derivative is in $C(E)$ (resp. $C_b(E)$, $C_0(E)$)
\item $L^p_{E}(\Omega, \F, \PP)$: quotient space of $\F - Bor(E)$ measurable functions $\Omega \to E$ s.t. $\int_\Omega ||f||_E^pd\PP<\infty$ 
\item $B_E(\Omega, \F)$ (resp. $B^b_E(\Omega, \F)$): the space of $\F - Bor(E)$ measurable (resp. measurable bounded) functions $\Omega \to E$
\item $D(J,E)$ the Skorohod space of RCLL (right-continuous with left limits) functions $J \to E$ $(J \in Bor(\RR))$
\item $\PPP(E)$: family of Borel probability measures on $E$
\item $\LL(X|\PP)$ (or $\LL(X)$ if there is no ambiguity): Law of $X$, i.e. $\PP \circ X^{-1}$
\item $\stackrel{a.e.}{\to}$, $\stackrel{P}{\to}$, $\Rightarrow$: convergence resp. a.e., in probability, in distribution.
\item $t^{\e,s}:=s+\e(t-s)$ (section 4 and Appendix only)
\item $|E|$: cardinal of E
\item $d$: Skorohod metric (see \cite{EK}, chapter 3, equation 5.2)
\end{itemize}

\section{Introduction}

Random Evolutions began to be studied in the 1970's, because of their potential applications to biology, movement of particles, signal processing, quantum physics, finance \& insurance, etc. (see \cite{GH}, \cite{He}, \cite{HP}). As R. Hersh says in \cite{He}: "Random evolutions model a situation in which an evolving system changes its law of motion because of random changes in the environment". These random changes are usually modeled by jump processes, because they aim at modeling the fact that the system moves from some state to another, for example a stock that switches between different volatilities. In 1972, T. Kurtz established a Law of Large numbers for Random Evolutions constructed with homogeneous semigroups (\cite{K}). Following this work, J. Watkins published a very interesting paper (\cite{Wa84}), followed by two more (\cite{Wa85a}, \cite{Wa85b}) in which, using a martingale characterization of the Random Evolution, he established a Central Limit Theorem for Random Evolutions constructed with i.i.d. generators of (homogeneous) semigroups under some technical assumptions, especially on the dual of the Banach space. Later, A. Swishchuk and V. Korolyuk established a Law of Large Numbers and a Central Limit theorem for Random Evolutions driven by (uniformly ergodic) Semi-Markov processes (\cite{S}, \cite{SK} chapter 4, \cite{LS}) in which the switching between the different (homogeneous) semigroups occurs at the jump times of a semi-Markov process. \\

To the best of our knowledge, only homogeneous Random Evolutions have been studied as of today, i.e. Random Evolutions constructed with homogeneous semigroups. In this paper, we consider Random Evolutions constructed with backward inhomogeneous semigroups (\ref{1000}), also called backward propagators or backward evolution systems in the literature (e.g. \cite{P} chapter 5, \cite{GVC} chapter 2) and driven by (uniformly ergodic) Semi-Markov processes. We choose the backward case for practical reasons: for example the Chapman-Kolmogorov equation is backward in time. In this paper we present several new results:

\begin{enumerate}

\item In sections 2 and 3, we study some general properties of inhomogeneous semigroups and construct inhomogeneous Random Evolutions. These sections are mainly related to functional analysis but we need them to prove the main law of large numbers result of section 4, especially the characterization of inhomogeneous semigroups as a unique solution to a well-posed Cauchy problem (\ref{10}), a second order Taylor formula for inhomogeneous semigroups (\ref{3}) and the forward integral representation for inhomogeneous Random Evolutions (\ref{14}).

\item In section 4 we establish our main result (\ref{27}), that can be thought of as a law of large numbers for inhomogeneous Random Evolutions driven by uniformly ergodic Semi-Markov processes: we index the inhomogeneous Random Evolution by a small parameter $\e$ that we use to rescale time so that the Semi-Markov process goes to its unique stationary distribution as $\e \to 0$. We also obtain a martingale characterization of the Random Evolution (\ref{25}). We establish the weak convergence of the Random Evolution in the Skorohod space $D(J,Y)$ to an inhomogeneous semigroup, where $J \subseteq \RR^+$ and $Y$ is a separable Banach space. The proof of the unicity of the limiting distribution is linked to the unicity of the Cauchy problem of section 2 and cannot be done using what has been done before in \cite{SK} or \cite{Wa84} for the following reason: in the latter papers, because of the time-homogeneity of the random evolutions, weak convergence of sequences of martingales of the following type were studied:
\begin{align*}
& X_n(t)-X_n(s)-\int_s^t A X_n(u)du,
\end{align*}
where $A$ is the generator of a semigroup and $\{X_n(\bullet)\}_{n \in \NN}$ a sequence of $D(J,Y)$ valued random variables. Provided the relative compactness of $\{X_n(\bullet)\}_{n \in \NN}$, one could choose a limiting process $X_0(\bullet)$ and prove its unicity in distribution. The problem that we will be facing in our paper is the following: in the (backward) inhomogeneous case, we will have to prove weak convergence of sequences of martingales of the type:
\begin{align*}
& V_n(t)f-V_n(s)f-\int_s^t V_n(u)A(u)fdu,
\end{align*}
where $f \in Y$, $A(t)$ is the generator of an inhomogeneous semigroup and $\{V_n(\bullet)\}_{n \in \NN}$ are stochastic processes with sample paths in $D(J,\B(Y))$, where $\B(Y)$ is the space of bounded linear operators on $Y$. Nevertheless, we will typically only have the relative compactness of $\{V_n(\bullet)f\}_{n \in \NN}$ in $D(J,Y)$, for every $f \in Y$ and not the relative compactness of $\{V_n(\bullet)\}_{n \in \NN}$ in $D(J, \B(Y))$, one of the reasons being that $\B(Y)$ might not be separable (and all the usual techniques for the proof of relative compactness in $D(J, E)$ require $E$ to be separable, see \cite{EK}). Therefore we cannot just pick a limiting process $V_0(\bullet) \in D(J, \B(Y))$ and prove its unicity in distribution, but we will have to construct it ourself using mainly density and the Skorohod representation theorem, handling negligible sets of our probability space with care.

\item Finally an important remark should be made about applications where $Y=C_0(\RR^d)$, the space of continuous functions $\RR^d \to \RR$ vanishing at infinity. To prove the crucial compact containment criterion, it is said in both \cite{Wa84} and \cite{SK} that there exists a compact embedding of a Sobolev space into $C_0(\RR^d)$, which is not true. Therefore we shall see on the specific example of  inhomogeneous L\'{e}vy Random Evolutions (section 5) how we can still prove the compact containment criterion using a characterization of compact sets in $C_0(\RR^d)$ (see \ref{18}), and how this proof can be recycled in the case of other examples.

\end{enumerate}

The paper is organized as follows: in section 2 we present some results on inhomogeneous semigroups, in section 3 we introduce inhomogeneous random evolutions driven by (uniformly ergodic) Semi-Markov processes and some of their properties, in section 4 we prove our main law of large number result of weak convergence of the inhomogeneous random evolution to an inhomogeneous semigroup as well as a martingale characterization of the inhomogeneous Random Evolution, and finally in section 5 we give an application to inhomogeneous L\'{e}vy Random Evolutions.\\

\textbf{Convention throughout the paper:}
Let $(Y,|| \cdot ||)$ be a real separable Banach space. Let $\Y$ be the Borel sigma-algebra generated by the norm topology. Let $Y^*$ the dual space of $Y$. $(Y_1,|| \cdot ||_{Y_1})$ is assumed to be a real separable Banach space which is continuously embedded in $Y$ (this idea was used in \cite{P}, chapter 5), i.e. $Y_1 \subseteq Y$ and $\exists c_1 \in \RR^+$: $||f|| \leq c_1 || f ||_{Y_1}$ $\forall f \in Y_1$. Unless mentioned otherwise, limits are taken in the $Y-$norm. Limits in the $Y_1$ norm will be denoted $Y_1-\lim$. In the following, $J$ will refer either to $\RR^+$ or to $[0,T_\infty]$ for some $T_\infty>0$ and $\Delta_J:=\{(s,t) \in J^2: s \leq t  \}$. Let also, for $s \in J$: $J(s):=\{t \in J: s \leq t  \}$ and $\Delta_J(s):=\{(r,t) \in J^2: s \leq r \leq t  \}$.\\

\section{Inhomogeneous operator semigroups}

This section presents some results on inhomogeneous semigroups. Some of them are similar to what can be found in \cite{P} chapter 5 and \cite{GVC} chapter 2, but to the best of our knowledge, they are new.

\begin{df}
A function $\G: \Delta_J \to \B(Y)$ is called a (backward) inhomogeneous $Y$-semigroup if:\\
i) $\forall t \in J$: $\G(t,t)=I$\\
ii) $\forall (s,r), (r,t) \in \Delta_J$: $\G(s,r)\G(r,t)=\G(s,t)$\\
If in addition, $\forall (s,t) \in \Delta_J$: $\G(s,t)=\G(0,t-s)$, $\G$ is called homogeneous $Y-$semigroup.\\
\end{df}

We now introduce the generator of the inhomogeneous semigroup:
\begin{df}
Let $t \in J$ and:
\begin{align*}
 \D(A_\G(t)):=\left\{ f \in Y: \lim_{\substack{h \downarrow 0\\ t+h \in J}}\frac{(\G(t,t+h)-I)f}{h}=  \lim_{\substack{h \downarrow 0\\ t-h \in J}}\frac{(\G(t-h,t)-I)f}{h} \in Y \right\}
\end{align*} 

Define $\forall t \in J$, $\forall f \in \D(A_\G(t))$:
\begin{align*}
&A_\G(t)f:=\lim_{\substack{h \downarrow 0\\ t+h \in J}}\frac{(\G(t,t+h)-I)f}{h}=  \lim_{\substack{h \downarrow 0\\ t-h \in J}}\frac{(\G(t-h,t)-I)f}{h}
\end{align*} 
Let $\D(A_\G):= \bigcap \limits_{t \in J} \D(A_\G(t))$. Then $A_\G: J \to L(\D(A_\G), Y)$ is called the generator of the inhomogeneous $Y$-semigroup $\G$.\\
\end{df}

The following definitions deal with continuity and boundedness of semigroups:
\begin{df}
An inhomogeneous $Y$-semigroup $\G$ is $\B(Y)-$bounded (resp. $\B(Y)-$ contraction) if $\sup \limits_{(s,t) \in \Delta_J}||\G(s,t)||_{\B(Y)} < \infty$ (resp. $\sup \limits_{(s,t) \in \Delta_J}||\G(s,t)||_{\B(Y)} \leq 1$).
\end{df}

\begin{df}
Let $E_Y \subseteq Y$. An inhomogeneous $Y$-semigroup $\G$ is $E_Y-$strongly continuous if $\forall (s,t) \in \Delta_J$, $\forall f \in E_Y$:
\begin{align*}
\lim_{\substack{(h_1,h_2) \to (0,0)\\ (s+h_1,t+h_2) \in \Delta_J}} ||\G(s+h_1,t+h_2)f-\G(s,t)f||=0
\end{align*} 
\end{df}

\begin{df}
Let $E_{Y_1} \subseteq Y_1$. An inhomogeneous $Y$-semigroup $\G$ is $E_{Y_1}-$super strongly continuous if $\forall (s,t) \in \Delta_J$, $\forall f \in E_{Y_1}$:
\begin{align*}
\G(s,t)Y_1 \subseteq Y_1 \mbox{ and }\lim_{\substack{(h_1,h_2) \to (0,0)\\ (s+h_1,t+h_2) \in \Delta_J}} ||\G(s+h_1,t+h_2)f-\G(s,t)f||_{Y_1}=0
\end{align*} 
\end{df}

\textbf{Remark}: throughout the paper, we will use this terminology that "super strong continuity"  refers to continuity in the $Y_1-$norm and that  "strong continuity"  refers to continuity in the $Y-$norm.\\

\textbf{Terminology}: For the above types of continuity, we use the terminology $t-$continuity (resp. $s-$continuity) for the continuity of the partial application $u \to \G(s,u)f$ (resp. $u \to \G(u,t)f $).\\

The following theorems give conditions under which the semigroup is differentiable in $s$ and $t$.
\begin{thm}
\label{111}
Let $\G$ be an inhomogeneous $Y$-semigroup. Assume that $Y_1 \subseteq \D(A_\G)$ and that $\G$ is $Y_1-$super strongly $s-$continuous, $Y_1-$strongly $t-$continuous. Then:
\begin{align*}
\frac{\partial}{\partial s} \G(s,t)f=-A_\G(s) \G(s,t)f \hspace{5mm} \forall (s,t) \in \Delta_J, \forall f \in Y_1
\end{align*}
\end{thm}
\begin{proof} see Appendix \ref{A1} \end{proof}

\begin{thm}
\label{2}
Let $\G$ an inhomogeneous $Y$-semigroup. Assume $Y_1 \subseteq \D(A_\G)$ and that $\G$ is $Y-$strongly $t$-continuous. Then we have:
\begin{align*}
\frac{\partial}{\partial t} \G(s,t)f= \G(s,t) A_\G(t)f \hspace{7mm} \forall (s,t) \in \Delta_J, \forall f \in Y_1
\end{align*}
\end{thm}
\begin{proof} see Appendix \ref{A1} \end{proof}

In general, for $f \in Y_1$, we will want to use the semigroup integral representation $\G(s,t)f-f=\int_s^t \G(s,u) A_\G(u) fdu$ and therefore we will need that $u \to \G(s,u) A_\G(u)f$ is in $L^1_Y([s,t])$. The following theorem give sufficient conditions for which it is the case, as we will typically have $A_\G(t) \in \B(Y_1,Y)$  $\forall t \in J$.\\

\begin{thm}
\label{7}
Assume that theorem \ref{2} holds. Assume also that $\forall t \in J$, $A_\G(t) \in \B(Y_1,Y)$ and $ \forall (s,t) \in \Delta_J$, $u \to ||A_\G(u)||_{\B(Y_1,Y)} \in L^1_\RR([s,t])$. Then $\forall f \in Y_1$, $(s,t) \in \Delta_J$: 
\begin{align*}
\G(s,t)f-f=\int_s^t \G(s,u) A_\G(u) fdu
\end{align*}
\end{thm}
\begin{proof} see Appendix \ref{A1} \end{proof}

The following definition introduces the concept of regular inhomogeneous semigroup, which basically means that it is differentiable and that its derivative is integrable.
\begin{df}
An inhomogeneous $Y$-semigroup $\G$ is said to be \textbf{regular} if it satisfies theorems \ref{111}, \ref{2} and $\forall (s,t) \in \Delta_J$, $\forall f \in Y_1$, $u \to \G(s,u) A_\G(u) f$ is in $L^1_Y([s,t])$.\\
\end{df}

Now we are ready to characterize the inhomogeneous semigroup as a unique solution of a well-posed Cauchy problem, which we will need for our main result \ref{27}:

\begin{thm}
\label{10}
Let $A_\G$ the generator of a a regular inhomogeneous $Y$-semigroup $\G$ and $s \in J$, $G_s \in \B(Y)$. A solution operator $G: J(s) \to \B(Y)$ to the Cauchy problem:
\begin{align*}
\left\{
\begin{array}{ll}
\frac{d}{dt}G(t)f=G(t)A_\G(t)f \hspace{7mm} \forall t \in J(s), \forall f \in Y_1\\
G(s)=G_s
\end{array}
\right.
\end{align*}

is said to be \textbf{regular} if it is $Y-$strongly continuous and if it satisfies the Cauchy problem above. If $G$ is such a regular solution, then we have $G(t)f=G_s \G(s,t)f$, $\forall t \in J(s)$, $\forall f \in Y_1$.\\
\end{thm}
\begin{proof} see Appendix \ref{A1} \end{proof}

The following Corollary comes straightforwardly from \ref{10} and expresses the fact that equality of generators of regular semigroups implies equality of semigroups.  
\begin{cor}
\label{9}
Assume that $\G_1$ and $\G_2$ are regular inhomogeneous $Y$-semigroups and that $\forall f \in Y_1$, $\forall t \in J$, $A_{\G_1}(t)f=A_{\G_2}(t)f$. Then $\forall f \in Y_1$, $\forall (s,t) \in \Delta_J$: $\G_1(s,t)f=\G_2(s,t)f$. In particular if $Y_1$ is dense in $Y$, then $\G_1$ and $\G_2$ agree on $Y$.
\end{cor}

We conclude this section with a $2^{nd}$ order Taylor formula for inhomogeneous semigroups:
\begin{df}
\label{1}
Let:
\begin{align*}
& \D(A_\G \in Y_1):=\left\{ f \in \D(A_\G) \cap Y_1:  A_\G(t)f \in Y_1\hspace{2mm} \forall t \in J \right\} \\
& \D(A_\G'):= \left\{ f \in \D(A_\G \in Y_1): Y_1-\lim_{\substack{h \to 0\\ t+h \in J}}\frac{A_\G(t+h)f-A_\G(t)f}{h}  \in Y_1  \hspace{2mm} \forall t \in J \right\}
\end{align*}

and for $t \in J$, $f \in \D(A_\G')$:
\begin{align*}
& A_\G'(t)f:=Y_1-\lim_{\substack{h \to 0\\ t+h \in J}} \frac{A_\G(t+h)f-A_\G(t)f}{h}\\
\end{align*} 
\end{df}

\begin{thm}
\label{3}
Let $\G$ a regular inhomogeneous $Y$-semigroup, $(s,t) \in \Delta_J$. Assume that $\forall u \in J$, $A_\G(u) \in \B(Y_1,Y)$ and $u \to ||A_\G(u)||_{\B(Y_1,Y)} \in L^1_\RR([s,t])$. Then we have for $f \in \D(A_\G \in Y_1)$:
\begin{align*}
& \G(s,t)f=f+\int_s^tA_\G(u)f du+\int_s^t \int_s^u \G(s,r) A_\G(r) A_\G(u)f dr du
\end{align*}

Assume in addition that:
\begin{enumerate}[i)]
\item $A_\G$ is $Y_1-$strongly continuous
\item $u \to \G(s,u) A_\G^2(u)f$ and $u \to \int_s^u  \G(s,r) A_\G(r) A'_\G(u)fdr$ $\in L^1_Y([s,t])$
\end{enumerate}

then we have for $f \in \D(A_\G')$:
\begin{align*}
 \G(s,t)f=&f+\int_s^tA_\G(u)f du+\int_s^t (t-u) \G(s,u) A_\G^2(u)fdu\\&+\int_s^t (t-u) \int_s^u  \G(s,r) A_\G(r) A'_\G(u)fdrdu
\end{align*}
\end{thm}
\begin{proof} see Appendix \ref{A1} \end{proof}

Remark how the latter formula coincides with the well-known $2^{nd}$ order Taylor formula for homogeneous semigroups (see e.g. \cite{SK}, proposition 1.2.):
\begin{align*}
& \G(t-s)f=f+(t-s)A_\G f +\int_s^t (t-u) \G(u-s) A_\G^2fdu
\end{align*}


\section{Inhomogeneous Random Evolutions}
As in \cite{JM} (section 3.5), let $(\Omega, \F, \PP)$ a complete probability space, $(X, \X)$ a finite measurable space and $(x_n)_{n \in \NN}$, $(\tau_n)_{n \in \NN^*}$  random variables resp. $\Omega \to X$, $\Omega \to \RR^{+*}$. Let $\tau_0:=0$, $T_n:= \sum_{k=0}^n \tau_k$. The sojourn times $T_n$ form a strictly increasing sequence, for every $\oo$. We say that $(x_n,T_n)_{n \in \NN}$ is a Markov renewal process if there exists a semi-Markov kernel $Q: X \times X \times \RR^{+} \to [0,1]$ (see e.g. \cite{LO}, definition 2.2.) such that $\forall y \in X$, $t \in \RR^+$, $n \in \NN$:

\begin{align*}
\PP[x_{n+1} =y, T_{n+1} -T_n \leq t| x_k,T_k: k \in [|0,n|]]&=\PP[x_{n+1} =y, T_{n+1} -T_n \leq t| x_n]\\&=Q(x_n,y,t) \mbox{ a.e.}
\end{align*}

In the following we let $(x_n,T_n)_{n \in \NN}$ a Markov renewal process and let:
\begin{itemize}
\item $P(x,y):= Q(x, y ,\infty):= \lim_{t \to \infty} Q(x, y ,t)=\PP[x_{n+1} =y | x_{n}=x]$
\item $F(x,t ):= \sum_{y \in X}Q(x,y,t )=\PP[T_{n+1}-T_n \leq t | x_{n}=x]$ \\
\end{itemize}
 
Define the counting process:
\begin{align*}
& N(t)(\omega):=\sup \{n \in \NN: T_n(\omega) \leq t\}=\sup_{n \in \NN^*} \sum_{k=1}^n 1_{\{T_k \leq t\}}(\omega) \mbox{ , for } \omega \in \Omega\mbox{ , } t \in \RR^+ 
\end{align*}
 
By the latter representation, $N(t)$ is $\F-Bor(\overline{\RR^+})$ measurable and represents the number of jumps on $(0,t]$ and can possibly be infinite in the case of a general state space. In the case of a finite state space however, the renewal process is regular, i.e. $\forall t \in \RR^+$, $N(t)< \infty$ a.e.. Further, $\forall \omega \in \Omega$, $t \to N_t(\omega)$ is right continuous, as it is constant on the intervals $[T_n(\omega),T_{n+1}(\omega))$: $n \in \NN$. We define the semi-Markov process $(x(t))_{t \in \RR^+}$ by $x(t)(\omega):=x_{N(t)(\omega)}(\omega)$ on 
\begin{align*}
& \Omega^*:=\bigcap_{t \in \QQ} \{\omega \in \Omega: N(t)(\omega)<\infty\} \hspace{3mm} \mbox{(so that } \PP(\Omega^*)=1). 
\end{align*}

\textbf{Remark: }From now on, we will work on the probability space $(\Omega^*, \F^*,\PP^*):=(\Omega^*, \F|_{\Omega^*},\PP|_{\Omega^*})$, where $\F|_{\Omega^*}:=\{A \in \F: A \subseteq \Omega^*\}$ and $\PP|_{\Omega^*}:=\PP(A)$ $\forall A \in \F|_{\Omega^*}$ ($\F|_{\Omega^*}$ is a sigma-algebra on $\Omega^*$ since $\Omega^*\neq \emptyset \in \F $). We point out that the restrictions to $\Omega^*$ of $\F-$ measurable functions are $\F^*-$ measurable. For sake of clarity, we will not write explicitly that we are working with the restrictions to $\Omega^*$ of the random variables (e.g. $x_n|_{\Omega^*}$, $T_n|_{\Omega^*}$), but we will always be. In order to avoid heavy notations, $(\Omega^*, \F^*,\PP^*)$ will be noted $(\Omega, \F,\PP)$.\\

We define the following random variables on $\Omega$, for $s \leq t \in \RR^+$:
\begin{itemize}
\item the number of jumps on $(s,t]$: $N_s(t):=N(t)-N(s)$
\item the jump times on $(s,\infty)$: $T_n(s):=T_{N(s)+n}$ for $n \in \NN^*$, and $T_0(s):=s$.
\item the states visited by the process on $[s,\infty)$: $x_n(s):=x(T_n(s))$, for $n \in \NN$.
\end{itemize}
\vspace{2mm}

We will assume that $(x_n,T_n)_{n \in \NN}$ is an ergodic Markov Renewal Process, which means that it is irreducible aperiodic, and the imbedded Markov Chain $(x_n)_{n \in \NN}$ is aperiodic (see \cite{JM}, section 3.7). We will also assume that $(x_n)_{n \in \NN}$ is a uniformly ergodic Markov Chain, namely that there exists a probability measure $\rho:=\{\rho_x\}_{x \in X}$ on $X$ such that (see \cite{M}):
 \begin{align*}
 & \lim_{n \to \infty} || P^n -\Pi ||_{\B(B^b_\RR(X))}=0
 \end{align*}

where for $f \in B^b_\RR(X)$ (since $X$ finite, $B^b_\RR(X)=B_\RR(X)$), $x \in X$:
  \begin{align*}
  & \Pi f(x):= \sum_{y \in X} \rho_y f(y) \hspace{3mm} \mbox{ (constant of $x$)}\\
 & Pf(x):= \sum_{y \in X} P(x,y) f(y)=\EE[f(x_{n+1})|x_{n}=x]   
  \end{align*}

and we have that $P^nf(x)=\EE[f(x_{n})|x_{0}=x]$. From the standard theory of semi-Markov processes (see e.g. \cite{JM} propositions 3.16.2 and 3.9.1) we get that:
 \begin{align*}
 & \lim_{t \to \infty} \frac{1}{t} \EE[N(t)]=\frac{1}{M_\rho}\\
 & \lim_{t \to \infty} \frac{1}{t} N(t)=\frac{1}{M_\rho} \mbox{ a.e.}
  \end{align*}
  With:
  \begin{align*}
 & M_\rho:=\Pi m_1
 & m_n(x):=\int_{\RR^+} t^n  F(x,dt) \hspace{3mm} 
 \end{align*}

Now we are ready to introduce inhomogeneous random evolutions:\\

Consider a family of inhomogeneous $Y-$semigroups $(\G_x)_{x \in X}$, with respective generators $(A_x)_{x \in X}$, satisfying:
\begin{align*}
\forall s \in J: &(r,t,x,f) \to \G_x(r \wedge t, r \vee t)f \mbox{ is } Bor(J(s))\otimes Bor(J(s))\otimes \X \otimes \Y-\Y \mbox{ measurable }\\
\end{align*}

as well as a family $(D(x,y))_{(x,y) \in X \times X} \subseteq \B(Y)$ of $\B(Y)-$contractions, satisfying:
\begin{align*}
& (x,y,f) \to D(x,y)f \mbox{ is } \X \otimes \X \otimes \Y-\Y \mbox{ measurable }\\
\end{align*}

\begin{df}
\label{1000}
The function $V: \Delta_J \times  \Omega \to \B(Y)$ defined pathwise by:
\begin{align*}
& V(s,t)(\omega)= \left[\prod \limits_{k=1}^{N_s(t)} \Gamma_{x_{k-1}(s)} \left(T_{k-1}(s), T_{k}(s) \right)  D(x_{k-1}(s),x_k(s)) \right] \Gamma_{x(t)} \left(T_{N_s(t)}(s),t \right)
\end{align*}
is called a $(\G,D,x)-$inhomogeneous $Y$-random evolution, or simply an inhomogeneous $Y$-random evolution. $V$ is said to be continuous if $D(x,y)=I$, $\forall (x,y) \in X \times X$. $V$ is said to be regular (resp. $\B(Y)-$contraction) if $(\G_x)_{x \in X}$ are regular (resp. $\B(Y)-$contraction).
\end{df}

\textbf{Remark:} We use as conventions that $\prod \limits_{k=1}^{0} :=I$ and $\prod \limits_{k=1}^{n} A_k:=A_1...A_{n-1}A_n$, that is, the product operator applies the product on the right.\\

\textbf{Remark:} if $N_s(t)>0$, then $x_{N_s(t)}(s)=x(T_{N_s(t)}(s))=x(T_{N(t)})=x(t)$. If $N_s(t)=0$, then $x(s)=x(t)$ and $x_{N_s(t)}(s)=x_0(s)=x(T_0(s))=x(s)=x(t)$. Therefore in all cases $x_{N_s(t)}(s)=x(t)$.\\

The following proposition deals with the measurability of the inhomogeneous Random evolution:
\begin{prop}
\label{11}
For $s \in J$, $f \in Y$, the stochastic process $(V(s,t)(\omega)f)_{(\omega,t) \in \Omega \times J(s)}$ is adapted to the (augmented) filtration:
\begin{align*}
&\F_t(s):=\sigma \left[x_{n \wedge N_s(t)}(s), T_{n \wedge N_s(t)}(s) : n \in \NN \right] \vee \sigma(\PP-\mbox{null sets})
\end{align*}
\end{prop}

\begin{proof}
Let $E \in \Y$, $(s,t) \in \Delta_J$, $f \in Y$. We have:
\begin{align*}
& V(s,t)f^{-1}(E)=\bigcup_{n \in \NN} \{V(s,t)f \in E\} \cap \{ N_s(t)=n\}
\end{align*}

Denoting the $\F_t(s)-Bor(\RR^{+})$ measurable (by construction) function $h_k:=T_{(k+1) \wedge N_s(t)}(s)-T_{k \wedge N_s(t)}(s)$, remark that $N_s(t)(\omega)= \sup_{m \in \NN} \sum_{k=0}^m 1_{h_k^{-1}(\RR^{+*})}(\omega)$ and is therefore $\F_t(s)-Bor(\RR^{+})$ measurable. Therefore $\{ N_s(t)=n\} \in \F_t(s)$. Let:
\begin{align*}
& \Omega_n:=\{N_s(t)=n\}
& M:=\left \{n \in \NN:  \Omega_n \neq \emptyset \right \}\\
\end{align*}

$M \neq \emptyset$ since $\Omega=\bigcup_{n \in \NN} \Omega_n$, and for $n \in M$, let the sigma-algebra $\F_{n}:=\F_t(s)|_{\Omega_n}:=\{A \in \F_t(s): A \subseteq \Omega_n\}$ ($\F_n$ is a sigma-algebra on $\Omega_n$ since $\Omega_n \neq \emptyset \in \F_t(s)$). Now consider the map $V(s,t)f(n): (\Omega_n, \F_{n}) \to (Y, \Y)$:
\begin{align*}
& V_n(s,t)f:=\left[ \prod \limits_{k=1}^{n} \Gamma_{x_{k-1}(s)} \left(T_{k-1}(s), T_{k}(s) \right) D(x_{k-1}(s),x_{k}(s)) \right] \Gamma_{x_n(s)} \left(T_{n}(s),t \right)f
\end{align*}

We have:
\begin{align*}
V(s,t)f^{-1}(E)&=\bigcup_{n \in \NN} \{V_n(s,t)f \in E\} \cap \Omega_n=\bigcup_{n \in \NN} \{\omega \in \Omega_n: V_n(s,t)f \in E\}\\
&=\bigcup_{n \in \NN} V_n(s,t)f^{-1}(E)
\end{align*}

Therefore it remains to show that $V_n(s,t)f^{-1}(E) \in \F_{n}$, since $\F_{n} \subseteq \F_t(s)$. \\

First let $n>0$. Notice that $V_n(s,t)f= \psi \circ \beta_{n}  \circ \alpha_n ... \circ \beta_1 \circ \alpha_1 \circ \phi$, where:

\begin{align*}
\phi: \hspace{2mm}&\Omega_n \to  J(s) \times X \times \Omega_n \to Y \times  \Omega_n\\
&\omega \to (T_{n}(s)(\omega),x_{n}(s)(\omega), \omega) \to  (\G_{x_{n}(s)(\omega)}(T_{n}(s)(\omega),t)f, \omega)
\end{align*}

The previous mapping holding since $T_{k}(s)(\omega) \in [s,t]$ $\forall \omega \in \Omega_n$, $k \in [|1,n|]$. $\phi$ is measurable iff each one of the coordinate mappings are. The canonical projections are trivially measurable. Let $A \in Bor(J(s))$, $B \in \X$. We have:
\begin{align*}
& \{\omega \in \Omega_n: T_{n}(s) \in A\}=\Omega_n \cap T_{n \wedge N_s(t)}(s)^{-1}(A) \in \F_{n}\\
& \{\omega \in \Omega_n: x_{n}(s) \in B\}=\Omega_n \cap x_{n \wedge N_s(t)}(s)^{-1}(B)\in \F_{n}\\
\end{align*}

Now, by measurability assumption, we have for $B \in \Y$:
\begin{align*}
& \{(t_{n},y_{n}) \in J(s) \times X:  \G_{y_{n}}(t \wedge t_{n}, t \vee t_{n})f\in B\}=C \in Bor(J(s))\otimes \X \mbox{ and therefore: }\\
& \{(t_{n},y_{n},\omega) \in J(s) \times X \times \Omega_n :  \G_{y_{n}}(t \wedge t_{n}, t \vee t_{n})f\in B\}=C \times \Omega_n \in Bor(J(s))\otimes \X \otimes \F_{n}\\
\end{align*}

Therefore $\phi$ is $\F_{n}- \Y \otimes \F_{n}$ measurable. Define for $i \in [1,n]$:
\begin{align*}
\alpha_{i}: \hspace{2mm} & Y \times  \Omega_n \to X \times X \times Y \times \Omega_n \to Y \times  \Omega_n\\
& (g, \omega) \to (x_{n-i}(s)(\omega), x_{n-i+1}(s)(\omega), g, \omega) \to (D(x_{n-i}(s)(\omega), x_{n-i+1}(s)(\omega))g, \omega)\\
\end{align*}

Again, the canonical projections are trivially measurable. We have for $p \in [|0,n|]$:
\begin{align*}
& \{\omega \in \Omega_n: x_{p}(s) \in B\}=\Omega_n \cap x_{p \wedge N_s(t)}(s)^{-1}(B):=C \in \F_{n} \\
& \mbox{Therefore } \{(g,\omega) \in Y \times \Omega_n: x_{p}(s) \in B\}=Y \times C \in \Y \otimes \F_{n}\\ 
\end{align*}

Now, by measurability assumption, $\forall B \in \Y$, $\exists C \in \X \otimes \X \otimes \Y$:
\begin{align*}
& \{(y_{n-i},y_{n-i+1},g, \omega) \in X \times X \times Y \times \Omega_n:  D(y_{n-i}, y_{n-i+1})g \in B \}=C \times \Omega_n \in \X \otimes \X \otimes \Y \otimes \F_{n}
\end{align*}

, which proves the measurability of $\alpha_i$. Then we define for $i \in [1,n]$:
\begin{align*}
\beta_{i}: \hspace{2mm} & Y \times  \Omega_n \to J(s) \times J(s) \times X \times Y \times \Omega_n \to Y \times  \Omega_n\\
& (g, \omega) \to (T_{n-i}(s)(\omega), T_{n-i+1}(s)(\omega), x_{n-i}(s)(\omega), g, \omega) \to (\G_{x_{n-i}(s)(\omega)}(T_{n-i}(s)(\omega),T_{n-i+1}(s)(\omega))g, \omega)\\
\end{align*}

By measurability assumption, $\forall B \in \Y$, $\exists C \in Bor(J(s))\otimes Bor(J(s))\otimes \X \otimes \Y$:
\begin{align*}
& \{(t_{n-i},t_{n-i+1}, y_{n-i},g, \omega) \in J(s) \times J(s) \times X \times Y  \times \Omega_n: \G_{y_{n-i}}(t_{n-i} \wedge t_{n-i+1}, t_{n-i} \vee t_{n-i+1})g  \in B \}\\
&=C \times \Omega_n \in Bor(J(s))\otimes Bor(J(s))\otimes \X \otimes \Y \otimes \F_{n},
\end{align*}
which proves the measurability of $\beta_i$.\\

Finally, define the canonical projection:
\begin{align*}
\psi: \hspace{2mm} & Y \times  \Omega_n \to Y\\
& (g,\omega) \to g
\end{align*}

which proves the measurability of $V_n(s,t)f$.\\

For $n=0$, we have $V(s,t)f(n)=\Gamma_{x(s)} \left(s,t \right)f$ and the proof is similar.\\
\end{proof}

The following propositions show that the random evolution has right-continuous paths and that it satisfies some integral representation, which will be used in section 4 to prove relative compactness in the Skorohod space $D(J(s),Y)$.
\begin{prop}
\label{12}
Let $V$ an inhomogeneous $Y$-random evolution and $(s,t) \in \Delta_J$, $\omega \in \Omega$. Then $V(\bullet,\bullet)(\omega)$ is an inhomogeneous $Y-$semigroup. Further, if $V$ is regular, then $u \to V(s,u)(\omega)$ is $Y-$strongly RCLL on $J(s)$, i.e. $\forall f \in Y$, $u \to V(s,u)(\omega)f \in D(J(s),(Y,|| \cdot ||))$. More precisely, we have for $f \in Y$:
\begin{align*}
& V(s,u^-)f=V(s,u)f \mbox{ if } u \notin \{T_{n}(s): n \in \NN\}\\
&V(s, T_{n+1}(s))f= V(s, T_{n+1}(s)^-) D(x_{n}(s),x_{n+1}(s))f \hspace{3mm}  \forall n \in \NN
\end{align*}
where we denote $V(s,t^-)f:=\lim_{u \uparrow t} V(s,u)f$.
\end{prop}
\begin{proof} see Appendix \ref{A1} \end{proof}
In particular we observe that if $D=I$, then $u \to V(s,u)(\omega)$ is in fact $Y-$strongly continuous on $J(s)$.

\begin{prop}
\label{14}
Let $V$ a regular inhomogeneous $Y$-random evolution and $(s,t) \in \Delta_J$, $f \in Y_1$. Then $V$ satisfies on $\Omega$:
\begin{align*}
& V(s,t)f=f+\int_s^t V(s,u)A_{x(u)}(u)f du+\sum \limits_{k=1}^{N_s(t)} V(s,T_k(s)^-) [D(x_{k-1}(s),x_k(s))-I]f
\end{align*}
\end{prop}
\begin{proof} see Appendix \ref{A1} \end{proof}


 \section{Law of Large numbers for the inhomogeneous Random Evolution}.\\

 \textbf{Notation}: in the following we denote for $\e \in \RR^+$, $(s,t) \in \RR^{+2}$: $t^{\e,s}:=s+\e(t-s)$.\\
 
 In the same way we introduced inhomogeneous $Y-$random evolutions, we consider a family $(D^\e(x,y))_{(x,y) \in X \times X, \e \in (0,1]}$ of $\B(Y)-$contractions, satisfying $\forall \e \in (0,1]$:
\begin{align*}
& (x,y,f) \to D^\e(x,y)f \mbox{ is } \X \otimes \X \otimes \Y-\Y \mbox{ measurable }\\
\end{align*}
and let $D^0(x,y):=I$. We define:
\begin{align*}
 \D(D_1):= \bigcap_{\substack{\e \in [0,1]\\(x,y) \in X \times X }} \left\{ f \in Y: \lim_{\substack{h \to 0\\ \e+h \in [0,1]}}  \frac{D^{\e+h}(x,y) f-D^\e(x,y)f}{h} \in Y \right\}
\end{align*} 

and $\forall f \in  \D(D_1)$:
\begin{align*}
&D^\e_1(x,y)f:=\lim_{\substack{h \to 0\\ \e+h \in [0,1]}} \frac{D^{\e+h}(x,y) f-D^\e(x,y)f}{h}
\end{align*}  

In the same way we introduce $\D(D_2)$, corresponding to the $2^{nd}$ derivative.\\

We also let:
\begin{align*}
 \D(D^0_1 \in Y_1):=\left\{ f \in \D(D_1) \cap Y_1:  D^0_1(x,y)f \in Y_1\hspace{2mm} \forall (x,y) \in X \times X \right\}\\
\end{align*}

We define the space:
\begin{align*}
&\widehat{\D}:=\bigcap_{x \in X} \D(A'_x) \cap \D(D_2) \cap \D(D^0_1 \in Y_1), 
\end{align*}

where $\D(A'_x)$ was defined in \ref{1}. These notations will hold until the end of the paper. In this section we assume the following set of assumptions that we call \textbf{(A0)}:\\

\textit{Assumptions on the regularity of operators}:
\begin{enumerate}[i)]
\item $Y_1 \subseteq \D(D_1)$
\item $(\G_x)_{x \in X}$ are regular
\item $A_x$ is $Y_1-$strongly continuous, $\forall x \in X$\\
\end{enumerate}

\textit{Assumptions on the semi-Markov process}:
\begin{enumerate}[i)]
\item $\exists \bar{\tau}>0$ such that $Q(x,y,(\bar{\tau},\infty))=0$ $\forall (x,y) \in X \times X$ (uniformly bounded sojourn increments)
\end{enumerate}

\textbf{Remark}: the latter implies in particular that all the moments $m_n$ are well-defined. \\

\textit{Assumptions on the boundedness of operators}:
\begin{enumerate}[i)]
\item $(\G_x)_{x \in X}$ and $(D^\e(x,y))_{(x,y) \in X \times X, \e \in (0,1]}$ are $\B(Y)-$contractions, with $D^0:=I$.
\item $A_x(t) \in \B(Y_1,Y)$ $\forall t \in J$, $\forall x \in X$ and $\sup_{u \in J}||A_x(u)||_{\B(Y_1,Y)}<\infty$
\item $\sup_{\substack{u \in J \\x \in X}}||A'_x(u)f||<\infty$, $\sup_{\substack{u \in J \\x \in X}} ||A_x(u)f||_{Y_1}<\infty$, $\forall f \in \bigcap_{x \in X}\D(A'_x)$
\item $\sup_{\substack{\e \in [0,1] \\ (x,y) \in X^2}}||D_1^\e(x,y) f||<\infty$, $\forall f \in \D(D_1)$
\item $\sup_{\substack{\e \in [0,1] \\ (x,y) \in X \times X}}||D_2^\e(x,y) f||<\infty$, $\forall f \in \D(D_2)$
\end{enumerate}

\vspace{5mm}

As said in introduction: similarly to what has been done in \cite{SK} or \cite{Wa84} we index the inhomogeneous Random Evolution by a small parameter $\e$ that we use to rescale time so that the Semi-Markov process goes to its unique stationary distribution:
\begin{df}
Let $V$ an inhomogeneous $Y-$random evolution. We define (pathwise on $\Omega$) the inhomogeneous $Y-$random evolution in the averaging scheme $V_\e$ for $\e \in (0,1]$, $(s,t) \in \Delta_J$ by:
\begin{align*}
& V_{\e}(s,t):=\left[ \prod \limits_{k=1}^{N_s \left(t^{\frac{1}{\e},s} \right)} \Gamma_{x_{k-1}(s)} \left(T_{k-1}^{\e,s}(s), T_{k}^{\e,s}(s) \right) D^\e(x_{k-1}(s),x_k(s))\right] \Gamma_{x\left( t^{\frac{1}{\e},s}\right)} \left(T_{N_s\left(t^{\frac{1}{\e},s}\right)}^{\e,s}(s),t \right)\\
\end{align*}
\end{df}
 
 Remark: we notice that $V_\e$ is well-defined since on $\Omega$:
 \begin{align*}
 &T_{N_s\left(t^{\frac{1}{\e},s}\right)}^{\e,s}(s)=s+\e \left( T_{N_s\left(t^{\frac{1}{\e},s}\right)}(s)-s\right) \leq s+\e \left( t^{\frac{1}{\e},s}-s\right) =t
\end{align*}

and that it coincides with $V$ for $\e=1$, i.e. $V_1(s,t)=V(s,t)$.\\

Our goal is to prove, as in \cite{SK}, that for each $f$ in some suitable subset of $Y$, $\{V_\e(s,\bullet)f\}$ - seen as a family of elements of $D(J(s),Y)$ - converges weakly to some continuous limiting process $V_0(s,\bullet)f$ to be determined. To this end, we will first prove that $\{V_\e(s,\bullet)f\}$ is relatively compact with a.e. continuous weak limit points. This is equivalent to the notion of $C-$tightness in \cite{JS} (VI.3) because $\PPP(D(J(s),Y))$ topologized with the Prohorov metric is a separable and complete metric space ($Y$ being a separable Banach space), which implies that relative compactness and tightness are equivalent in $\PPP(D(J(s),Y))$ (by Prohorov's theorem). Then we will identify the limiting process $V_0$.\\

We first need some classical elements that can be found in \cite{SK} (1.4) and \cite{EK} (sections 3.8 to 3.11). In particular the Skorohod space $D(J(s),Y)$ will always be topologized with the Skorohod metric $d$ (see \cite{EK}, chapter 3, equation 5.2).

\begin{df}
Let $(\nu_n)_{n \in \NN}$ a sequence of probability measures on a metric space $(S, d)$. We say that $\nu_n$ converges weakly to $\nu$, and write $\nu_n \Rightarrow \nu$ iff $\forall f \in C_b(S)$:
 \begin{align*}
 \lim_{n \to \infty} \int_S fd \nu_n =\int_S fd \nu
 \end{align*}
\end{df}

\begin{df}
Let $\{\nu_\e\}$ a family of probability measures on a metric space $(S, d)$. $\{\nu_\e\}$ is said to be relatively compact iff for any sequence $(\nu_n)_{n \in \NN} \subseteq \{\nu_\e\}$, there exists a weakly converging subsequence. 
\end{df}

\begin{df} 
Let $s \in J$, $\{X_\e\}$ a family of stochastic processes with sample paths in $D(J(s),Y)$. We say that $\{X_\e\}$ is relatively compact iff $\{\LL(X_\e)\}$ is (in the metric space $\PPP(D(J(s),Y))$ endowed with the Prohorov metric). We write that $X_\e \Rightarrow X$ iff $\LL(X_\e) \Rightarrow \LL(X)$. We say that $\{X_\e\}$ is C-relatively compact iff it is relatively compact and if ever $X_\e \Rightarrow X$, then $X$ has a.e. continuous sample paths.\\

If $E_Y \subseteq Y$, we say that $\{V_\e\}$ is $E_Y-$relatively compact (resp. $E_Y-$C-relatively compact) iff $\{V_\e(s,\bullet)f\}$ is $\forall f \in E_Y$, $\forall s \in J$.\\
\end{df}
 
 \begin{df}
Let $s \in J$, $\{X_\e\}$ a family of stochastic processes with sample paths in $D(J(s),Y)$. We say that $\{X_\e\}$ satisfies the compact containment criterion if $\forall \Delta \in (0,1]$, $\forall T \in J(s)$, $\exists K \subseteq Y$ compact set such that:
\begin{align*} 
& \liminf_{\e \to 0} \PP[X_\e(t) \in K \hspace{3mm} \forall t \in [s,T]] \geq 1-\Delta
\end{align*}

We say that $\{V_\e\}$ satisfies the compact containment criterion in $E_Y \subseteq Y$ if $\forall f \in E_Y$, $\forall s \in J$, $\{V_\e(s,\bullet)f\}$ satisfies it. 
\end{df}

\begin{thm}
\label{15}
Let $s \in J$, $\{X_\e\}$ a family of stochastic processes with sample paths in $D(J(s),Y)$. $\{X_\e\}$ is C-relatively compact iff it is relatively compact and $J_s(X_\e) \Rightarrow 0$, where:
 \begin{align*} 
& J_s(X_\e):=\int_s^\infty e^{-u} ( J_s(X_\e,u)\wedge 1)du\\
& J_s(X_\e,u):= \sup_{t \in [s,u]} ||X_\e(t)-X_\e(t^-)||
\end{align*} 
\end{thm}
 
\begin{thm}
\label{16}
Let $s \in J$, $\{X_\e\}$ a family of stochastic processes with sample paths in $D(J(s),Y)$. $\{X_\e\}$ is relatively compact iff:
\begin{enumerate}[i)]
\item $\{X_\e\}$ satisfies the compact containment criterion
\item $\forall T \in J(s)$, $\exists r>0$ and a family $\{C_s(\e,\eta): (\e,\eta) \in (0,1] \times (0,1)\}$ of nonnegative random variables such that $\forall (\e,\eta) \in (0,1] \times (0,1)$, $\forall h \in [0,\eta]$, $\forall t \in [s,T]$:
\begin{align*}
&  \EE \left[ \left.||X_\e(t+h)-X_\e(t) ||^r \right| \GG^{\e,s}_t \right] \leq \EE[C_{s}(\e, \eta) | \GG^{\e,s}_t] \\
& \lim_{\eta \to 0} \limsup_{\e \to 0} \EE[C_{s}(\e, \eta)]=0
\end{align*}
where $\GG^{\e,s}_t:=\sigma \left[ X_\e(u): u \in [s,t]\right]$
\end{enumerate}
\end{thm}

 \subsection{The compact containment criterion}
  
The compact containment criterion can be in practice quite hard to prove, because we need either a compact embedding of a Banach space into $Y$, or a characterization of compact sets in $Y$. An important remark should be made about applications where $Y=C_0(\RR^d)$, the space of continuous functions $\RR^d \to \RR$ vanishing at infinity. To prove this compact containment criterion, it is said in both \cite{Wa84} and \cite{SK} that there exists a compact embedding of a Sobolev space into $C_0(\RR^d)$, which is not true. Here we suggest a method to prove it (\ref{18}), that will be applied in section 5 to the case of inhomogeneous L\'{e}vy Random Evolutions.\\

If we have a compact embedding then the proof of the compact containment is easy, as mentioned here:  
\begin{prop}
\label{17}
Assume that there exists a Banach space $(Z,|||\cdot|||)$ compactly embedded in $Y$, and that $(\G_x)_{x \in X}$, $(D^\e(x,y))_{\e \in (0,1], (x,y) \in X \times X}$ are $\B(Z)-$contractions. Then $\{V_\e\}$ satisfies the compact containment criterion in $Z$.
\end{prop}

\begin{proof} 
Let $f \in Z$, $(s,t) \in \Delta_J$, $c:=|||f|||$ and $K:=cl(Y)-S_c(Z)$, the $Y-$closure of the $Z-$closed ball of radius $c$. $K$ is compact because of the compact embedding of $Z$ into $Y$. Let $\e \in (0,1]$. Because $(\G_x)_{x \in X}$, $(D^\e(x,y))_{(x,y) \in X \times X}$ are $\B(Z)-$contractions, $V_\e$ is a $\B(Z)-$contraction and we have $\forall \omega \in \Omega$: $|||V_\e(s,t)(\omega)f|||\leq |||f||| = c$. Therefore $V_\e(s,t)(\omega)f \in S_c(Z) \subseteq K$ and so $\PP[V_\e(s,t)f \in K \hspace{3mm} \forall t \in [s,T] ]=\PP(\Omega)=1 \geq 1-\Delta$.\\
\end{proof} 
 
 For example, we can consider the Rellich-Kondrachov compactness theorem: if $U \subseteq \RR^d$ is an open, bounded Lipschitz domain, then the Sobolev space $W^{1,p}(U)$ is compactly embedded in $L^q(U)$, where $p \in [1,d)$ and $q \in [1,\frac{dp}{d-p})$.\\
 
For the space $C_0(\RR^d)$, there is no well-known such compact embedding, therefore we have to proceed differently.
\begin{prop}
\label{18}
Let $Y:=C_0(\RR^d)$, $E_Y \subseteq Y$. Assume that $\forall \Delta \in (0,1]$, $\forall (s,T) \in \Delta_J$, $\forall \e \in (0,1]$, $\forall f \in E_Y$, $\exists A_\e \subseteq \Omega: \PP(A_\e) \geq 1-\Delta$ and the family $\{V_\e(s,t)(\omega)f: t \in[s,T], \e \in (0,1], \omega \in A_\e \}$ converge uniformly to $0$ at infinity, is equicontinuous and uniformly bounded. Then $\{V_\e\}$ satisfies the compact containment criterion in $E_Y$.
\end{prop}
 
 \textbf{Remark:} we say that a family of functions $\{f_\al\}$ converge uniformly to 0 at infinity if:
 \begin{align*}
 & \forall \e>0, \exists \delta>0: |z|>\delta \Rightarrow \sup_\al |f_\al(z)| < \e
 \end{align*}
 
\begin{proof} 
Let $f \in E_Y$, $K$ the $Y-$closure of the set: 
\begin{align*}
K_1:=\{V_\e(s,t)(\omega)f: \e \in (0,1], \omega \in A_\e, t \in [s,T]\}
\end{align*}
$K_1$ is a family of elements of $Y$ that are equicontinuous, uniformly bounded and that converge uniformly to $0$ at infinity by assumption. Therefore it is well-known, using the Arzela-Ascoli theorem on the Alexandroff compactification of $\RR^d$, that $K_1$ is relatively compact in $Y$ and therefore that $K$ is compact in $Y$. And we have $\forall \e \in (0,1]$:
 \begin{align*} 
& \PP[V_\e(s,t)f \in K \hspace{3mm} \forall t \in [s,T] ] \geq \PP[\omega \in A_\e: V_\e(s,t)f \in K \hspace{3mm} \forall t \in [s,T]] = \PP(A_\e) \geq 1-\Delta\\
\end{align*} 
\end{proof} 
   
\textbf{Remark:} Because $ \lim_{t \to \infty} \frac{1}{t} N(t)=\frac{1}{M_\rho} \mbox{ a.e.}$, this set $A_\e$ will typically be $A_\e:=\left \{\e N_s \left(T^{\frac{1}{\e},s} \right) \leq n_{00} \right \}$ for some well chosen constant $n_{00}$ (see Appendix \ref{A2}).

 \subsection{Relative compactness of the inhomogeneous Random Evolution in $D(J(s),Y)$}
  
In the following we will make the following assumption:
\begin{align*} 
& \textbf{(A1)} \hspace{3mm} \{V_\e\} \mbox{ satisfies the compact containment criterion in $Y_1$}  \\
\end{align*}

\begin{lem}
\label{19}
Let $(s,t) \in \Delta_J$, $f \in Y_1$. Under \textbf{(A0)}, $V_\e$ satisfies on $\Omega$:
\begin{align*}
& V_\e(s,t)f=f+\int_s^t V_\e(s,u)A_{x \left(u^{\frac{1}{\e},s} \right)}(u)f du+\sum \limits_{k=1}^{N_s \left(t^{\frac{1}{\e},s} \right)} V_\e(s,T_k^{\e,s}(s)^-) [D^\e(x_{k-1}(s),x_k(s))-I]f
\end{align*}
\end{lem}

\begin{proof} 
Same proof as \ref{14}, except that the induction is made on the intervals $[T_k^{\e,s}(s),T_{k+1}^{\e,s}(s)) \cap J(s)$, instead of $[T_k(s),T_{k+1}(s)) \cap J(s)$.\\
\end{proof} 

 \begin{lem}
\label{20}
Assume \textbf{(A0)} and \textbf{(A1)}. Then $\{V_\e\}$ is $Y_1-$relatively compact.
\end{lem}
 
\begin{proof} 
We want to prove \ref{16}. Using \ref{19} we have for $h \in [0, \eta]$:
\begin{align*}
& ||V_\e(s,t+h)f-V_\e(s,t)f||\\
& \leq  \left| \left| \int_t^{t + h} V_\e(s,u)A_{x \left(u^{\frac{1}{\e},s} \right)}(u)f du + \sum \limits_{k=N_s \left(t^{\frac{1}{\e},s} \right)}^{N_s \left((t+h)^{\frac{1}{\e},s} \right)} V_\e(s,T_k^{\e,s}(s)^-) [D^\e(x_{k-1}(s),x_k(s))-I]f \right| \right|\\
& \leq  \eta M_1+  \e \sum \limits_{k=N_s \left(t^{\frac{1}{\e},s} \right)}^{N_s \left((t+\eta)^{\frac{1}{\e},s} \right)} \frac{1}{\e}||D^\e(x_{k-1}(s),x_k(s))f-f|| \\
& \leq \eta M_1 + \e M_2 \left[N \left((t+\eta)^{\frac{1}{\e},s} \right)-N\left(t^{\frac{1}{\e},s} \right) +1 \right]
\end{align*}

where $M_1:= \sup_{x , u } ||A_x (u)||_{\B(Y_1,Y)} ||f||_{Y_1} $ and $M_2:=\sup_{\e,x,y} ||D^\e_1(x,y)f||$ (by \textbf{(A0)}). Notice $||V_\e(s,T_k^{\e,s}(s)^-)||_{\B(Y)} \leq 1$ because both $\G$ and $D^\e$ are $\B(Y)-$contractions. \\

For $\e \in (0,1]$ we have:
\begin{align*}
&   \e \left[N \left((t+\eta)^{\frac{1}{\e},s} \right)-N\left(t^{\frac{1}{\e},s} \right) +1 \right]\\
& \leq   \e \sup_{t \in [s,s+\eta]} \left[N \left((t+\eta)^{\frac{1}{\e},s} \right)-N\left(t^{\frac{1}{\e},s} \right) \right] +  \e \sup_{t \in [s+\eta,T]} \left[N \left((t+\eta)^{\frac{1}{\e},s} \right)-N\left(t^{\frac{1}{\e},s} \right) \right] + \e \\
& \leq  \e N \left((s+2\eta)^{\frac{1}{\e},s} \right)  +  \e \sup_{t \in [s+\eta,T]} \left[N \left((t+\eta)^{\frac{1}{\e},s} \right)-N\left(t^{\frac{1}{\e},s} \right) \right] + \e
\end{align*}

Note that the supremums in the previous expression are a.e. finite as they are a.e. bounded by $N \left((T+1)^{\frac{1}{\e},s} \right)$. Now let:
\begin{align*}
C_s(\e, \eta):=\eta M_1+ M_2 \e N \left((s+2\eta)^{\frac{1}{\e},s} \right)  + M_2 \e \sup_{t \in [s+\eta,T]} \left[N \left((t+\eta)^{\frac{1}{\e},s} \right)-N\left(t^{\frac{1}{\e},s} \right) \right] +M_2 \e
\end{align*}

We have to show that $ \lim_{\eta \to 0} \lim_{\e \to 0} \EE[C_s(\e, \eta)]=0$. We have:
\begin{align*}
 \lim_{\eta \to 0} \lim_{\e \to 0}\eta M_1+M_2  \e \EE \left[N \left((s+2\eta)^{\frac{1}{\e},s} \right) \right]+M_2 \e=  \lim_{\eta \to 0} \eta M_1+M_2 \frac{2\eta}{M_\rho} =0
\end{align*}

Let $\{\e_n\}$ any sequence that goes to 0, and denote: 
\begin{align*}
Z_n:=\e_n \sup_{t \in [s+\eta,T]} \left[N \left((t+\eta)^{\frac{1}{\e_n},s} \right)-N\left(t^{\frac{1}{\e_n},s} \right) \right]
\end{align*}
We first want to show that $\{Z_n\}$ is uniformly integrable. By \cite{EK}, it is sufficient to show that $\sup_n \EE(Z_n^2)<\infty$. We have that $\EE(Z_n^2) \leq \e_n^2 \EE \left[ N^2 \left((T+1)^{\frac{1}{\e_n},s} \right) \right]$. But by \cite{H}, we get that for an ergodic Markov Renewal process (which is our framework):
\begin{align*}
& \lim_{t \to \infty} \frac{\EE(N^2(t))}{t^2} < \infty
\end{align*}

And therefore $\{Z_n\}$ is uniformly integrable.\\

Then we show that $Z_n \stackrel{a.e.}{\to} Z:=\frac{\eta}{M_\rho}$. Let:
\begin{align*}
& \Omega^*:=\left \{\lim_{\e \to 0} \e N \left((s+1)^{\frac{1}{\e},s} \right) = \frac{1}{M_\rho} \right \},
\end{align*}

so that $\PP(\Omega^*)=1$. Let $\oo \in \Omega^*$ and $\delta>0$. There exists some constant $r_2(\oo,\delta)>0$ such that for $\e<r_2$:
\begin{align*}
& \left| \e N \left((s+1)^{\frac{1}{\e},s} \right) - \frac{1}{M_\rho}\right| < \frac{\delta}{T+\eta}
\end{align*}

and if $t \in [s+\eta, T+\eta]$:
\begin{align*}
& \left| (t-s) \e N \left((s+1)^{\frac{1}{\e},s} \right) - \frac{t-s}{M_\rho}\right| < \frac{\delta (t-s)}{T+\eta} \leq \delta
\end{align*}

Let $\e<\eta r_2$ (recall $\eta>0$)  and $\e_2:=\frac{\e}{t-s}$. Then $\e_2 < \frac{\eta r_2}{\eta}=r_2$, and therefore:
\begin{align*}
& \left| (t-s) \e_2 N \left((s+1)^{\frac{1}{\e_2},s} \right) - \frac{t-s}{M_\rho}\right| <  \delta\\
& \Rightarrow \left| \e N \left(t^{\frac{1}{\e},s} \right) - \frac{t-s}{M_\rho}\right| <  \delta
\end{align*}

And therefore for $\e<\eta r_2$ and $t \in [s+\eta, T]$:
\begin{align*}
& \left| \e N \left((t+\eta)^{\frac{1}{\e},s} \right)-\e N \left(t^{\frac{1}{\e},s} \right) - \frac{\eta}{M_\rho}\right| < 2 \delta\\
& \Rightarrow \left| \sup_{t \in [s+\eta, T] } \left[\e N \left((t+\eta)^{\frac{1}{\e},s} \right)-\e N \left(t^{\frac{1}{\e},s} \right) \right] - \frac{\eta}{M_\rho}\right| \leq 2 \delta < 3 \delta
\end{align*}

We have proved that $Z_n \stackrel{a.e.}{\to} Z$. By uniform integrability of $\{Z_n\}$, we get that $\lim_{n \to \infty} \EE(Z_n)=\EE(Z)$ and therefore since the sequence $\{\e_n\}$ is arbitrary:
\begin{align*}
\lim_{\e \to 0} \e \EE \left[ \sup_{t \in [s+\eta,T]} \left[N \left((t+\eta)^{\frac{1}{\e},s} \right)-N\left(t^{\frac{1}{\e},s} \right) \right] \right] = \frac{\eta}{M_\rho}
\end{align*}

 \end{proof}

 \begin{lem}
\label{21}
Assume  \textbf{(A0)}, \textbf{(A1)}. Then $\{V_\e\}$  is $Y_1-$C-relatively compact.
\end{lem}
 
\begin{proof} 
We want to prove \ref{15}. By \ref{20} it is relatively compact. By \ref{15} it is sufficient to show that $J_s(V_\e(s, \bullet)f)  \stackrel{a.e.}{\to} 0$. Let $\delta>0$ and choose $T>0$ such that $e^{-T} \leq \frac{\delta}{2}$. For $u \in [s,T]$ we have:
\begin{align*} 
J_s(V_\e(s,\bullet)f,u) &\leq \sup_{t \in [s,T]} || V_\e(s,t)f-V_\e(s,t^-)f|| \\(\mbox{using }\ref{19})&= \max_{k \in \left[ \left| 1,N_s \left(T^{\frac{1}{\e},s} \right) \right| \right]} \left| \left|  V_\e \left(s, T_{k}^{\e,s}(s) \right)f-V_\e \left(s, T_{k}^{\e,s-}(s) \right)f\right| \right|\\
(\mbox{using }\ref{19})&=\max_{k \in \left[ \left| 1,N_s \left(T^{\frac{1}{\e},s} \right) \right| \right]} ||V_\e \left(s, T_{k}^{\e,s-}(s) \right)(D^\e(x_{k-1}(s),x_k(s))f-f)||\\
& \leq \max_{k \in \left[ \left| 1,N_s \left(T^{\frac{1}{\e},s} \right) \right| \right]} ||D^\e(x_{k-1}(s),x_k(s))f-f||\\
& \leq \max_{(x,y) \in X \times X} ||D^\e(x,y)f-f||
\end{align*} 

Since $Y_1 \subseteq \D(D_1)$, $\exists r>0: \e<r \Rightarrow \max_{(x,y) \in X \times X} ||D^\e(x,y)f-f|| <\frac{\delta}{2}$, and therefore for $\e<r$, $\oo \in \Omega$:
\begin{align*} 
& J_s(V_\e(s,\bullet)f)=\int_s^T e^{-u} ( J_s(V_\e(s,\bullet)f,u)\wedge 1)du + \int_T^\infty e^{-u} ( J_s(V_\e(s,\bullet)f,u)\wedge 1)du\\
& < \frac{\delta}{2} (e^{-s}-e^{-T})+e^{-T} \leq \delta
\end{align*} 

\end{proof} 

 \subsection{Weak convergence in $D(J(s),Y)$ of the inhomogeneous Random Evolution}

In this section, we prove our main result \ref{27}. As in \cite{Wa84} and \cite{SK}, we start by finding a martingale characterization of $V_\e(s,\bullet)f$ (\ref{24}, \ref{25}). We then prove that this martingale converges weakly to zero (\ref{26}). Nevertheless, the proof of the main result \ref{27} will differ completely from what has been done in the latter references, for the reasons mentioned in introduction.\\

We first start by the following definition that involves the operator $(P-I)^{-1}$ of the uniformly ergodic Markov Chain $(x_n)$.
\begin{df}
\label{112}
Assume  \textbf{(A0)}. For $f \in Y_1$, $x \in X$, $t \in J$, let $f^\e(x,t):=f+\e f_1(x,t)$, where $f_1$ is the unique solution of the equation (see \cite{M}, proposition 4):
\begin{align*}
& (P-I)f_1(\bullet,t)(x)=M_\rho [\widehat{A}(t)-a(x,t)]f\\
& a(x,t):=\frac{1}{M_\rho} \left(m_1(x) A_x(t)+PD_1^0(x,\bullet)(x) \right)\\
& \widehat{A}(t):=\Pi  a(\bullet,t)
\end{align*}

namely, $f_1(x,t)=M_\rho (P-I)^{-1}(\widehat{A}(t)f-a(\bullet,t)f)(x)$.
\end{df}

\begin{rk}
\label{113}
Recall that $P, \Pi, M_\rho$ have been defined at the beginning of section 3. The existence of $f_1$ is guaranteed because $\Pi[\widehat{A}(t)-a(\bullet,t)]f=0$ by definition of $\widehat{A}$ (see \cite{M}, proposition 4). In fact, in \cite{M}, the operators $\Pi$ and $P$ are defined on $B^b_\RR(X)$ but the results hold true if we work on $B^b_{E}(X)$, where $E$ is any Banach space such that $[\widehat{A}(t)-a(x,t)]f \in E$ (e.g. $E=Y_1$ if $f \in \widehat{\D}$, $E=Y$ if $f \in Y_1$). To see that, first observe that $P$ and $\Pi$ can be defined the same way on $B^b_{E}(X)$ as they were on $B^b_{\RR}(X)$. Then take $\ell \in E^*$ such that $||\ell|| = 1$ and $g \in B^b_{E}(X)$ such that $||g||_{B^b_{E}(X)}=\max_x ||g(x)||_{E}=1$. We therefore have that: $||\ell \circ g||_{B^b_{\RR}(X)} \leq 1$, and since we have the uniform ergodicity on $B^b_{\RR}(X)$, we have that :
\begin{align*}
& \sup_ {\substack{ || \ell|| = 1\\ ||g||_{B^b_{E}(X)}=1\\ x \in X }} | P^n(\ell \circ g)(x) -\Pi (\ell \circ g)(x) | \leq || P^n -\Pi ||_{\B(B^b_{\RR}(X))} \to 0
\end{align*}
By linearity of $\ell, P, \Pi$ (and because $P$ and $\Pi$ can be defined the same way on $B^b_{E}(X)$ as they were on $B^b_{\RR}(X)$) we get that $| P^n(\ell \circ g)(x) -\Pi (\ell \circ g)(x) | = |\ell (P^n g(x)-\Pi g(x))|$. But because $||P^n g(x)-\Pi g(x)||_{E}=\sup_{|| \ell|| = 1} |\ell (P^n g(x)-\Pi g(x))|$ and that this supremum is attained (see e.g. \cite{C}, section III.6), then:
\begin{align*}
&\sup_ {\substack{ || \ell|| = 1\\ ||g||_{B^b_{E}(X)}=1\\ x \in X }} |\ell (P^n g(x)-\Pi g(x))| = \sup_ {\substack{ ||g||_{B^b_{E}(X)}=1\\ x \in X }} || P^n g(x)-\Pi g(x) ||_{E}\\
&=\sup_ {\substack{ ||g||_{B^b_{E}(X)}=1}} || P^n g-\Pi g ||_{B^b_{E}(X)}=|| P^n -\Pi ||_{\B(B^b_{E}(X))}
\end{align*}
and so we also have $|| P^n -\Pi ||_{\B(B^b_{E}(X))} \to 0$, i.e. the uniform ergodicity in $B^b_{E}(X)$. Now, according to the proofs of theorems 3.4, 3.5 chapter VI of \cite{Re} and because the Markov chain $(x_n)$ is aperiodic, $|| P^n -\Pi ||_{\B(B^b_{E}(X))} \to 0$ is the only thing we need to prove that $P-I$ is invertible on: 
\begin{align*}
B^\Pi_{E}(X):=\{f \in B^b_{E}(X): \Pi f=0\},
\end{align*}
the space $E$ plays no role. Further, $(P-I)^{-1} \in \B(B^\Pi_{E}(X))$ by the bounded inverse theorem.
\end{rk}

\begin{lem}
\label{24}
Assume \textbf{(A0)}. Define recursively for $\e \in (0,1]$, $s \in J$:
\begin{align*}
& V_0^{\e}(s):=I\\
& V_{n+1}^{\e}(s):=V_{n}^{\e}(s)\Gamma_{x_{n}(s)} \left(T_{n}^{\e,s}(s) , T_{n+1}^{\e,s}(s)   \right) D^\e(x_{n}(s),x_{n+1}(s))
\end{align*}
i.e. $V_{n}^{\e}(s)=V_\e(s,T_{n}^{\e,s}(s))$; and for $f \in Y_1$:
\begin{align*}
&M_n^{\e}(s)f:=V_n^{\e}(s)f^\e(x_n(s),T_{n}^{\e,s}(s)  )-f^\e(x(s),s)\\&-\sum_{k=0}^{n-1} \EE[V_{k+1}^{\e}(s)f^\e(x_{k+1}(s),T_{k+1}^{\e,s}(s)  )-V_{k}^{\e}(s)f^\e(x_k(s),T_{k}^{\e,s}(s) )| \F_k(s)]\\
\end{align*}

so that $(M_n^{\e}(s)f)_{n \in \NN}$ is a $\F_n(s)-$martingale by construction. Let for $t \in J(s)$:
\begin{align*}
&\widetilde{M}_t^{\e}(s)f:=M^{\e}_{N_s \left(t^{\frac{1}{\e},s} \right)+1}(s)f\\
& \widetilde{\F}_t^\e(s):=\F_{N_s \left(t^{\frac{1}{\e},s} \right)+1}(s)
\end{align*}
 
 where $\F_n(s):=\sigma \left[x_{k}(s), T_{k}(s) : k \leq n \right] \vee \sigma(\PP-\mbox{null sets})$ and $\F_{N_s \left(t^{\frac{1}{\e},s} \right)+1}(s)$ is defined the usual way (provided we have shown that $N_s \left(t^{\frac{1}{\e},s} \right)+1$ is a $\F_n(s)$-stopping time $\forall t \in J(s)$). Then $\forall \ell \in Y^*$, $\forall s \in J$, $\forall \e \in (0,1]$, $\forall f \in Y_1$, $(\ell(\widetilde{M}_t^{\e}(s)f),\widetilde{\F}^\e_t(s))_{t \in J(s)}$ is a real-valued martingale.
\end{lem}

\textbf{Remark:} The expectations in \ref{24} are taken in the usual Bochner sense and that will be the case throughout all the paper, unless mentioned otherwise.\\

\begin{proof} 
By construction $(\ell(M_n^{\e}(s)f), \F_n(s))_{n \in \NN}$ is a martingale. Let $\theta(t):=N_s \left(t^{\frac{1}{\e},s} \right)+1$. $\forall t \in J(s)$, $\theta(t)$ is a $\F_n(s)-$stopping time, because:
\begin{align*}
&\left\{\theta(t) = n \right\} =\left\{N_s \left(t^{\frac{1}{\e},s} \right) = n-1\right\} =\left\{T_{n-1}(s) \leq t^{\frac{1}{\e},s} \right\} \cap \left\{ T_{n}(s) > t^{\frac{1}{\e},s}\right\} \in \F_n(s)
\end{align*}

Let $t_1 \leq t_2 \in J(s)$. We have that $(\ell(M_{\theta(t_2) \wedge n}^{\e}(s)f), \F_n(s))_{n \in \NN}$ is a martingale. Assume we have shown that it is uniformly integrable, then we can apply the optional sampling theorem for uniformly integrable martingales to the stopping times $\theta(t_1) \leq \theta(t_2)$ a.e and get:
\begin{align*}
& \EE[ \ell(M^{\e}_{\theta(t_2) \wedge \theta(t_2)}(s)f )| \F_{\theta(t_1)}(s)] = \ell(M^{\e}_{\theta(t_1) \wedge \theta(t_2)}(s)f) \mbox{ a.e.}\\& \Rightarrow \EE[ \ell(M^{\e}_{\theta(t_2)}(s)f) | \F_{\theta(t_1)}(s)] = \ell(M^{\e}_{\theta(t_1)}(s)f) \mbox{ a.e.} \\
& \Rightarrow \EE[ \ell( \widetilde{M}_{t_2}^{\e}(s)f )|\widetilde{\F}^\e_{t_1}(s)] =\ell(\widetilde{M}_{t_1}^{\e}(s)f) \mbox{ a.e.},
\end{align*}
which shows that $(\ell(\widetilde{M}_t^{\e}(s)f),\widetilde{\F}^\e_t(s))_{t \in J(s)}$ is a martingale. Now to show the uniform integrability, by \cite{EK} it is sufficient to show that $\sup_n \EE(||M_{\theta(t_2) \wedge n}^{\e}(s)f||^2)<\infty$. But:
\begin{align*}
& ||M_{\theta(t_2) \wedge n}^{\e}(s)f|| \leq 2||f||+||f_1||+2(||f||+||f_1||) (\theta(t_2) \wedge n) \leq 2(||f||+||f_1||) (1+\theta(t_2))
\end{align*}

where $||f_1||:=\sup_{x,u}||f_1(x,u)||$ ($||f_1||<\infty$ since by \textbf{(A0)}, $\sup_{u}||A_x(u)||_{\B(Y_1,Y)}<\infty$ and $(P-I)^{-1} \in \B(B^\Pi_{Y}(X))$ by Remark \ref{113}). Since $\EE(\theta(t_2)^2)<\infty$ by \cite{H}, we are done.
 \end{proof} 
 \vspace{5mm}

 \begin{rk}
 \label{22}
In the following we will make use of the fact that can be found in \cite{B} (theorem 3.1) that for sequences $(X_n)$, $(Y_n)$ of random variables with value in a separable metric space with metric $d'$, if $X_n \Rightarrow X$ and $d'(X_n, Y_n) \Rightarrow 0$, then $Y_n \Rightarrow X$. In our case we will typically have $d'(X_n, Y_n) \stackrel{a.e.}{\to} 0$, and to show it we will use the remark in \cite{EK} after lemma 5.1, chapter 3 that implies that if $(X_n)$, $(Y_n)$ take value in $D(J(s),Y)$ and if $\forall T \in \QQ^{+} \cap J(s)$:
\begin{align*}
 & \lim_{n \to \infty} \sup_{t \in [s,T]} ||X_n(\oo)(t)-Y_n(\oo)(t)||=0 \mbox{ a.e.}
\end{align*}
then $d(X_n, Y_n) \stackrel{a.e.}{\to} 0$ (where as mentioned before, $d$ is the Skorohod metric).
 \vspace{5mm}
 \end{rk}
   
\begin{lem}
\label{25}
Assume \textbf{(A0)} and let $n_\e(s,t):=1+\lfloor \frac{t-s}{\e M_\rho} \rfloor$ and $t^\e_k(s,t):=s+k \frac{t-s}{n_\e(s,t)}$. For $f \in \widehat{\D}$ and $s \in J$, $\widetilde{M}_\bullet^{\e}(s)f$ has the asymptotic representation:
\begin{align*}
&\widetilde{M}_\bullet^{\e}(s)f=V_\e(s,\bullet)f-f-\e M_\rho \sum_{k=1}^{n_\e(s,\bullet)} V_\e(s,t^\e_k(s,\bullet)) \widehat{A}  \left( t^\e_k(s,\bullet)  \right)f + \circ(1) \mbox{ a.e.},
\end{align*}
where $\circ(\e^p)$ is defined by the following property:
\begin{align*} 
\forall T \in \QQ^{+} \cap J(s), \lim_{\e \to 0} \e^{-p} \sup_{t \in [s,T]} ||\circ(\e^p)||=0 \mbox{ a.e.},
\end{align*}
so that the remark \ref{22} on the a.e. convergence in the Skorohod space will be satisfied.
\end{lem}

 \begin{proof}

  For sake of clarity let:
 \begin{align*}
 &f^\e_k:=f^\e(x_{k}(s),T_{k}^{\e,s}(s)  )
 &f_{1,k}:=f_1(x_{k}(s),T_{k}^{\e,s}(s)  )
  \end{align*}
 
  First we have that:
  \begin{align*}
 & V_{N_s \left(t^{\frac{1}{\e},s} \right)+1}^{\e}(s)f^\e_{N_s \left(t^{\frac{1}{\e},s} \right)+1}=V_{N_s \left(t^{\frac{1}{\e},s} \right)+1}^{\e}(s)f + \circ(1)\\
  \end{align*}
  because $||V_{N_s \left(t^{\frac{1}{\e},s} \right)+1}^{\e}(s)f_{1,N_s \left(t^{\frac{1}{\e},s} \right)+1} || \leq ||f_1||$.\\
  
  Where $||f_1||$ is defined similarly as in the proof of \ref{24}. Now we have:
  \begin{align*}
 & V_{k+1}^{\e}(s)f^\e_{k+1}-V_{k}^{\e}(s)f^\e_k= V_{k}^{\e}(s)(f^\e_{k+1}-f^\e_{k})+ V_{k+1}^{\e}(s)f^\e_{k+1}-V_{k}^{\e}(s)f^\e_{k+1} 
 \end{align*}
 
 and:
  \begin{align*}
 & \EE[V_{k}^{\e}(s)(f^\e_{k+1}-f^\e_{k}) | \F_k(s)]= \e V_{k}^{\e}(s) \EE[(f_{1,k+1}-f_{1, k}) | \F_k(s)],\\
  \end{align*}
 
 as $V_{k}^{\e}(s)$ is $\F_k(s)-Bor(\B(Y))$ measurable. Now, we know that every discrete time Markov process with stationary transition kernel (i.e. that does not depend on $n$) has the strong Markov property, so the Markov process $(x_n,T_n)$ has it. For $k \geq 1$, the times $N(s)+k$ are $\F_n(0)-$stopping times. Therefore for $k \geq 1$:
 \begin{align*} 
 & \EE[(f_{1,k+1}-f_{1, k}) | \F_k(s)] =\EE[(f_{1,k+1}-f_{1, k}) | T_k(s), x_k(s)]\\
 \end{align*}
  and again using the strong Markov property as well as the stationarity of the Markov Renewal process:
 \begin{align*}
&\EE[(f_{1,k+1}-f_{1, k}) | T_k(s)=t_k, x_k(s)=x]\\
=&\sum_{y \in X} \int_0^\infty f_1 \left(y, t_k^{\e,s} + \e u  \right) Q(x,y,du)-f_1 \left( x , t_k^{\e,s}  \right)\\
 \end{align*}

Let $f'_1(x,t)=M_\rho (P-I)^{-1}[\widehat{A}'(t)f-a'(\bullet,t)f](x)$, and $a'(x,t)=\frac{1}{M_\rho} m_1(x) A'_x(t)$, $\widehat{A}'(t)=\Pi a'(\bullet,t)$ which exists because $(P-I)^{-1} \in \B(B^\Pi_{Y_1}(X))$ and $f \in \cap_{x \in X} \D(A'_x)$. Using the Fundamental theorem of Calculus for the Bochner integral ($v \to f'_1(y,v) \in L^1_Y([a,b])$ $\forall [a,b]$ since $||f'_1||<\infty$ by \textbf{(A0)}: indeed, $\sup_{x,t} ||A'_x(t)f||<\infty$):
 \begin{align*}
&\EE[(f_{1,k+1}-f_{1, k}) | T_k(s)=t_k, x_k(s)=x]\\
=& \sum_{y \in X} \int_0^\infty \left[ f_1 \left(y, t_k^{\e,s}   \right) + \int_{t_k^{\e,s}  }^{t_k^{\e,s}+\e u  } f'_1(y,v) dv \right] Q(x,y,du)-f_1 \left( x , t_k^{\e,s}   \right)\\
&= (P-I) f_1 \left( \bullet , t_k^{\e,s}   \right)(x) + \sum_{y \in X} \int_0^\infty \left[ \int_{0}^{\e u} f'_1(y,t_k^{\e,s}+ v  ) dv \right] Q(x,y,du)\\
&=(P-I) f_1 \left( \bullet , t_k^{\e,s}   \right)(x) + \circ(1)\\
  \end{align*}
 
 because as mentioned before: $||f'_1||<\infty$ by \textbf{(A0)} and:
 \begin{align*}
 & \sum_{y \in X} \int_0^\infty \left[ \int_{0}^{\e u} f'_1(y,t_k^{\e,s}+v  ) dv \right] Q(x,y,du) \leq \e  ||f_1'|| m_1(x)\\
 \end{align*}
   
 All put together we have for $k \geq 1$, using the definition of $f_1$:
 \begin{align*}
 &  \EE[V_{k}^{\e}(s)(f^\e_{k+1}-f^\e_{k}) | \F_k(s)]= \e V_{k}^{\e}(s)(P-I) f_1 \left( \bullet , T_k^{\e,s}(s)   \right)(x_k(s)) + \circ(\e)\\
 &=\e M_\rho V_{k}^{\e}(s) \left[ \widehat{A}(T_k^{\e,s}(s)  )-a(x_k(s),T_k^{\e,s}(s)  )\right]f+ \circ(\e)\\
 \end{align*}
 
 and the first term:
 \begin{align*}
 &  \EE[V_{0}^{\e}(s)(f^\e_{1}-f^\e_{0}) | \F_k(s)]=\circ(1) \hspace{3mm} (\leq 2 \e ||f_1||)\\
 \end{align*}
 
  Now we have to compute the terms corresponding to $V_{k+1}^{\e}(s)f^\e_{k+1}-V_{k}^{\e}(s)f^\e_{k+1} $. We will show that the term corresponding to $k=0$ is $\circ(1)$ and that for $k \geq 1$:
  \begin{align*}
  & \EE[V_{k+1}^{\e}(s)f^\e_{k+1}-V_{k}^{\e}(s)f^\e_{k+1} | \F_k(s)]=\e M_\rho V_{k}^{\e}(s)a(x_k(s),T_k^{\e,s}(s)  )f+ \circ(\e)\\
  \end{align*}
  
  To conclude, we have to show that $\sum_{k=1}^{N_s \left(t^{\frac{1}{\e},s} \right)} \circ(\e)=\circ(1)$. But $ \left| \left| \sum_{k=1}^{N_s \left(t^{\frac{1}{\e},s} \right)} \circ(\e) \right| \right| \leq \frac{||\circ(\e)||}{\e} \e N_s \left(t^{\frac{1}{\e},s} \right)$. The fact that $\sup_{t \in [0,T]}\e N_s \left(t^{\frac{1}{\e},s} \right) = \e N_s \left(T^{\frac{1}{\e},s} \right)\stackrel{a.e.}{\to} \frac{T-s}{M_\rho}$ concludes the proof.\\
  
 For the term $k=0$, we have using \textbf{(A0)}, the definition of $V_{k}^{\e}$ and \ref{7}:
 \begin{align*}
  & V_{1}^{\e}(s)f^\e_{1}-V_{0}^{\e}(s)f^\e_{1}=V_{1}^{\e}(s)f-f+\circ(1)\\
  &=\Gamma_{x(s)} \left(s, T_{1}^{\e,s}(s)   \right) D^\e(x(s),x_{1}(s))f-f+\circ(1)\\
  &=\Gamma_{x(s)} \left(s, T_{1}^{\e,s}(s)   \right) (D^\e(x(s),x_{1}(s))f-f)+\Gamma_{x(s)} \left(s, T_{1}^{\e,s}(s)   \right)f -f+\circ(1)\\
  & \Rightarrow \EE[ V_{1}^{\e}(s)f^\e_{1}-V_{0}^{\e}(s)f^\e_{1} | \F_0(s)] \leq \max_{x,y} ||D^\e(x,y)f-f|| \\&+ \int_0^\infty \e u   \sup_{x , t } ||A_x(t)||_{\B(Y_1,Y)}||f||_{Y_1}F(x,du)\\
 &\leq \e \left(\sup_{\e,x,y} ||D_1^\e(x,y)f|| + ||f||_{Y_1} \sup_{x , t } m_1(x) ||A_x(t)||_{\B(Y_1,Y)} \right)=\circ(1) 
   \end{align*}

  Now we have for $k \geq 1$:
  \begin{align*}
  & V_{k+1}^{\e}(s)-V_{k}^{\e}(s)=V_{k}^{\e}(s) \left[ \Gamma_{x_{k}(s)} \left(T_{k}^{\e,s}(s)  , T_{k+1}^{\e,s}(s)   \right) D^\e(x_{k}(s),x_{k+1}(s))-I\right]\\
  \end{align*}
  
  And because by \textbf{(A0)} we have $\sup_{\e,x,y}||D_1^\e(x,y) g||<\infty$ for $g \in Y_1$, we get using \ref{7}:
 \begin{align*}
 & D^\e(x_{k}(s),x_{k+1}(s))g=g+\int_0^\e D_1^u(x_{k}(s),x_{k+1}(s))g du\\
 & \mbox{and } \G_{x_{k}(s)} \left(T_{k}^{\e,s}(s)  , T_{k+1}^{\e,s}(s)   \right)g =g+\int_{T_{k}^{\e,s}(s)  }^{T_{k+1}^{\e,s}(s)  } \G_{x_{k}(s)}(T_{k}^{\e,s}(s)  ,u)A_{x_{k}(s)}(u)gdu \\
 \end{align*} 
  
  Because $f \in \widehat{\D}$, we get that $\widehat{A}(t)f \in Y_1$, $a(x,t)f \in Y_1$ $\forall t, x$. Since $(P-I)^{-1} \in \B(B^\Pi_{Y_1}(X))$ (by remark \ref{113}), we get that  $f_{1,k+1} \in Y_1$ and therefore using the previous representations and \textbf{(A0)}:
   \begin{align*}
  & \Gamma_{x_{k}(s)} \left(T_{k}^{\e,s}(s)  , T_{k+1}^{\e,s}(s)   \right) D^\e(x_{k}(s),x_{k+1}(s)) f_{1,k+1}= \Gamma_{x_{k}(s)} \left(T_{k}^{\e,s}(s)  , T_{k+1}^{\e,s}(s)   \right) f_{1,k+1} + \circ(1) \\
  &= f_{1,k+1} +\int_{T_{k}^{\e,s}(s)  }^{T_{k+1}^{\e,s}(s)  } \G_{x_{k}(s)}(T_{k}^{\e,s}(s)  ,u)A_{x_{k}(s)}(u) f_{1,k+1} du+ \circ(1)\\
   \end{align*}
  
  Therefore taking the conditional expectation we get:
   \begin{align*}
   & \EE[ V_{k+1}^{\e}(s)f_{1,k+1}-V_{k}^{\e}(s)f_{1,k+1}| \F_k(s)]\\
   &= V_{k}^{\e}(s) \sum_{y \in X} \int_0^\infty \left[ \int_{0}^{\e u}  \G_{x_k(s)}(T_{k}^{\e,s}(s)  ,T_{k}^{\e,s}(s)+v  )A_{x_k(s)}(T_{k}^{\e,s}(s)+v  ) f_{1,k+1} dv \right] Q(x_k(s),y,du)+\circ(1)\\
   &= \circ(1) \hspace{3mm} (\leq \e C m_1(x_k(s)) \mbox{ for some constant $C$ by \textbf{(A0)}})\\
    \end{align*}

  and so:
  \begin{align*} 
   \EE[ V_{k+1}^{\e}(s)f^\e_{k+1}-V_{k}^{\e}(s)f^\e_{k+1}| \F_k(s)]=\EE[ V_{k+1}^{\e}(s)f-V_{k}^{\e}(s)f| \F_k(s)]+ \circ(\e)\\
  \end{align*} 
  
  Now because $f \in \widehat{\D}$ and by \textbf{(A0)} (which ensures that the integral below exists):
   \begin{align*}
 & D^\e(x_{k}(s),x_{k+1}(s))f=f+\e D_1^0(x_{k}(s),x_{k+1}(s))f+\int_0^\e (\e-u) D_2^u(x_{k}(s),x_{k+1}(s))f du\\
  \end{align*}
  
  And so using boundedness of $D_2^\e$ (again \textbf{(A0)}):
  \begin{align*}
  &\Gamma_{x_{k}(s)} \left(T_{k}^{\e,s}(s)  , T_{k+1}^{\e,s}(s)   \right) D^\e(x_{k}(s),x_{k+1}(s))f\\
  &=\Gamma_{x_{k}(s)} \left(T_{k}^{\e,s}(s)  , T_{k+1}^{\e,s}(s)   \right) f+\e \Gamma_{x_{k}(s)} \left(T_{k}^{\e,s}(s)  , T_{k+1}^{\e,s}(s)   \right) D_1^0(x_{k}(s),x_{k+1}(s))f + \circ(\e)\\
  \end{align*}
  
  The first term has the representation by (\ref{3}):
   \begin{align*}
   & \Gamma_{x_{k}(s)} \left(T_{k}^{\e,s}(s)  , T_{k+1}^{\e,s}(s)   \right) f=f+\int_{T_{k}^{\e,s}(s)  }^{T_{k+1}^{\e,s}(s)  } A_{x_{k}(s)}(u)fdu\\ 
  &+\int_{T_{k}^{\e,s}(s)  }^{T_{k+1}^{\e,s}(s)  } \int_{T_{k}^{\e,s}(s)  }^u \G_{x_{k}(s)}(T_{k}^{\e,s}(s)  ,r) A_{x_{k}(s)}(r) A_{x_{k}(s)}(u)f dr du\\
   \end{align*}
   
   Taking the conditional expectation and using the fact that by \textbf{(A0)}: $\sup_{u,x}||A_x(u)f||_{Y_1}<\infty$, $\sup_{u,x}||A_x(u)||_{\B(Y_1,Y)}<\infty$ and $m_2(x_k(s))<\infty$, we have:
   \begin{align*}
   & \EE \left[ \left. \int_{T_{k}^{\e,s}(s)  }^{T_{k+1}^{\e,s}(s)  } \int_{T_{k}^{\e,s}(s)  }^u \G_{x_{k}(s)}(T_{k}^{\e,s}(s)  ,r) A_{x_{k}(s)}(r) A_{x_{k}(s)}(u)f dr du \right|  \F_k(s) \right]= \circ(\e)\\
  \end{align*}
  
  The second term has the representation, because $f \in \widehat{\D}$ (which ensures that $D_1^0(x,y)f \in Y_1$) and using \ref{7}:
  \begin{align*}
  & \e \Gamma_{x_{k}(s)} \left(T_{k}^{\e,s}(s)  , T_{k+1}^{\e,s}(s)   \right) D_1^0(x_{k}(s),x_{k+1}(s))f\\
  &= \e D_1^0(x_{k}(s),x_{k+1}(s))f+\e \int_{T_{k}^{\e,s}(s)  }^{T_{k+1}^{\e,s}(s)  } \Gamma_{x_{k}(s)} \left(T_{k}^{\e,s}(s)  , u \right)A_{x_{k}(s)}(u)D_1^0(x_{k}(s),x_{k+1}(s))fdu\\
  &=\e D_1^0(x_{k}(s),x_{k+1}(s))f+\circ(\e)\\
  \end{align*}
  
  And so all together we have:
  \begin{align*}
  & \EE[ \left. V_{k+1}^{\e}(s)f-V_{k}^{\e}(s)f| \F_k(s)]= V_{k}^{\e}(s) \EE \left[ \int_{T_{k}^{\e,s}(s)  }^{T_{k+1}^{\e,s}(s)  } A_{x_{k}(s)}(u)fdu+ \e D_1^0(x_{k}(s),x_{k+1}(s))f \right| \F_k(s)  \right]+ \circ(\e)\\
  \end{align*}
  
  We have by the strong Markov property and the stationarity of the Markov Renewal process:
   \begin{align*}
  & \EE \left[ D_1^0(x_{k}(s),x_{k+1}(s))f \left| \right. \F_k(s) \right]=PD_1^0(x_{k}(s),\bullet)(x_{k}(s))f\\
   \end{align*}
   
   and:
   \begin{align*}
  & \EE \left[ \left. \int_{T_{k}^{\e,s}(s)  }^{T_{k+1}^{\e,s}(s)  } A_{x_{k}(s)}(u)fdu \right|  x_k(s)=x, T_{k}(s)=t_k \right]\\
  &= \sum_{y \in X} \int_0^{\infty} \left[ \int_{t_{k}^{\e,s}  }^{t_{k}^{\e,s}+\e u  } A_{x}(v)fdv \right] Q(x,y,du)\\
  &=\sum_{y \in X} \int_0^{\infty} \left[ \e u A_{x}(t_{k}^{\e,s})f+\int_{0}^{\e u} (\e u-v) A'_{x}(t_{k}^{\e,s}+v)f dv \right] Q(x,y,du)\\
  &=\e m_1(x) A_{x}(t_{k}^{\e,s})f + \circ(\e) \mbox{ as } \sup_{u,x} ||A'_{x}(u)f|| < \infty \mbox{ and } m_2(x)<\infty.\\
   \end{align*} 
   
   So finally we get:
   \begin{align*}
  & \EE[V_{k+1}^{\e}(s)f^\e_{k+1}-V_{k}^{\e}(s)f^\e_{k+1} | \F_k(s)]=\e M_\rho V_{k}^{\e}(s)a(x_k(s),T_k^{\e,s}(s)  )f+ \circ(\e)\\
  \end{align*}
  
  And therefore:
\begin{align*}
\widetilde{M}_t^{\e}(s)f=V_{N_s \left(t^{\frac{1}{\e},s} \right)+1}^{\e}(s)f-f-\e M_\rho \sum_{k=1}^{N_s \left(t^{\frac{1}{\e},s} \right)} V_\e(s,T_{k}^{\e,s}(s)) \widehat{A}  \left( T_{k}^{\e,s}(s)  \right)f + \circ(1)
 \end{align*}

 Now, let $\theta(t):=N_s \left(t^{\frac{1}{\e},s} \right)+1$. Using \textbf{(A0)} (in particular uniform boundedness of sojourn times): 
 \begin{align*}
& V_{\theta(t)}^{\e}(s)f=V_\e(s,t) \G_{x_{\theta(t)-1}(s)}(t,T^{\e,s}_{\theta(t)}(s)) D^\e(x_{\theta(t)-1}(s),x_{\theta(t)}(s))f\\
& \Rightarrow ||V_{\theta(t)}^{\e}(s)f-V_\e(s,t)f|| \leq ||\G_{x_{\theta(t)-1}(s)}(t,T^{\e,s}_{\theta(t)}(s)) D^\e(x_{\theta(t)-1}(s),x_{\theta(t)}(s))f-f||\\
& \leq  || D^\e(x_{\theta(t)-1}(s),x_{\theta(t)}(s))f-f|| + ||\G_{x_{\theta(t)-1}(s)}(t,T^{\e,s}_{\theta(t)}(s))f-f||\\
& \leq \e \sup_{\e,x,y} ||D^\e_1 (x,y)f|| + (T^{\e,s}_{\theta(t)}(s)-t) \sup_{x,t} ||A_x(t)||_{\B(Y_1,Y)}||f||_{Y_1}\\
& \leq \e \sup_{\e,x,y} ||D^\e_1 (x,y)f|| + (T^{\e,s}_{\theta(t)}(s)-T^{\e,s}_{\theta(t)-1}(s)) \sup_{x,t} ||A_x(t)||_{\B(Y_1,Y)}||f||_{Y_1}\\
& \leq \e \sup_{\e,x,y} ||D^\e_1 (x,y)f|| + \e \bar{\tau} \sup_{x,t} ||A_x(t)||_{\B(Y_1,Y)}||f||_{Y_1}
 \end{align*}

And therefore:
\begin{align*}
V_{N_s \left(t^{\frac{1}{\e},s} \right)+1}^{\e}(s)f=V_\e(s,t)f + \circ(1)\\
 \end{align*}
 
Now let's prove the final step, i.e.:

\begin{align*} 
 &  \e \sum_{k=1}^{N_s \left(t^{\frac{1}{\e},s} \right)} V_\e(s,T_{k}^{\e,s}(s)) \widehat{A}  \left( T_{k}^{\e,s}(s)  \right)f=\e \sum_{k=1}^{n_\e(s,t)} V_\e(s,t^\e_k(s,t)) \widehat{A}  \left( t^\e_k(s,t)  \right)f + \circ(1)
  \end{align*}

We have:
\begin{align*} 
 &   \left| \left| \e \sum_{k=1}^{N_s \left(t^{\frac{1}{\e},s} \right)} V_\e(s,T_{k}^{\e,s}(s)) \widehat{A}  \left( T_{k}^{\e,s}(s)  \right)f-\e \sum_{k=1}^{n_\e(s,t)} V_\e(s,t^\e_k(s,t)) \widehat{A}  \left( t^\e_k(s,t)  \right)f \right| \right| \\
 \leq & \underbrace{ \e \sum_{k=1}^{N_s \left(t^{\frac{1}{\e},s} \right)} || V_\e(s,T_{k}^{\e,s}(s)) \widehat{A}  \left( T_{k}^{\e,s}(s)  \right)f-V_\e(s,t^\e_k(s,t)) \widehat{A}  \left( t^\e_k(s,t)  \right)f || }_{(i)}\\
 +& \underbrace{\e \left| \left| \sum_{k=1}^{N_s \left(t^{\frac{1}{\e},s} \right)} V_\e(s,t^\e_k(s,t)) \widehat{A}  \left( t^\e_k(s,t)  \right)f - \sum_{k=1}^{n_\e(s,t)} V_\e(s,t^\e_k(s,t)) \widehat{A}  \left( t^\e_k(s,t)  \right)f \right| \right|}_{(ii)}
   \end{align*}

 For the second term $(ii)$ we have:
 \begin{align*}
 & \e \left| \left| \sum_{k=1}^{N_s \left(t^{\frac{1}{\e},s} \right)} V_\e(s,t^\e_k(s,t)) \widehat{A}  \left( t^\e_k(s,t) \right)f-\sum_{k=1}^{n_\e(s,t)} V_\e(s,t^\e_k(s,t)) \widehat{A}  \left( t^\e_k(s,t) \right)f \right| \right|\\ &\leq  \e C \left| N_s \left(t^{\frac{1}{\e},s} \right)-n_\e(s,t) \right| =  C  \left| \e N_s \left(t^{\frac{1}{\e},s} \right)- \frac{t-s}{ M_\rho} \right| + \circ(1)\\
 \end{align*}
 
 where $C:=|X| \sup_{t ,x} \left[ m_1(x) ||A_x(t)||_{\B(Y_1,Y)}||f||_{Y_1} + ||PD_1^0(x,\bullet)(x)|| \right]$. Now, let $\{\eta_n\} \subseteq \RR^{+*}$ any sequence such that $\eta_n \to 0$ :
 \begin{align*}
 & \sup_{t \in [s,T]} \left| \e N_s \left(t^{\frac{1}{\e},s} \right)- \frac{t-s}{ M_\rho} \right| \leq \sup_{t \in [s,s + \eta_n]} \left| \e N_s \left(t^{\frac{1}{\e},s} \right)- \frac{t-s}{ M_\rho} \right| + \sup_{t \in [s+\eta_n,T]} \left| \e N_s \left(t^{\frac{1}{\e},s} \right)- \frac{t-s}{ M_\rho} \right|
 \end{align*}
 
 We have:
  \begin{align*}
 & \sup_{t \in [s,s + \eta_n]} \left| \e N_s \left(t^{\frac{1}{\e},s} \right)- \frac{t-s}{ M_\rho} \right| \leq \e N_s \left((s+\eta_n)^{\frac{1}{\e},s} \right)+\frac{ \eta_n}{M_\rho} 
 \end{align*}
 
 and since $\e N_s \left((s+\eta_n)^{\frac{1}{\e},s} \right) \stackrel{a.e.}{\to} \frac{ \eta_n}{M_\rho}$ we have:
 \begin{align*}
 & \limsup_{\e \to 0} \sup_{t \in [s,s + \eta_n]} \left| \e N_s \left(t^{\frac{1}{\e},s} \right)- \frac{t-s}{ M_\rho} \right| \leq \frac{2 \eta_n}{M_\rho} \mbox{ on some subset } \Omega_n: \PP(\Omega_n)=1
 \end{align*}
 
 Now, in the proof of \ref{20} we showed that:
  \begin{align*}
 & \lim_{\e \to 0} \sup_{t \in [s+\eta_n,T]} \left| \e N_s \left(t^{\frac{1}{\e},s} \right)- \frac{t-s}{ M_\rho} \right| =0 \mbox{ on some subset } \Omega'_n: \PP(\Omega'_n)=1
  \end{align*}
 
 So finally we have on $\Omega^*:=\bigcap_n \Omega_n \cap \Omega'_n$:
 \begin{align*}
 & \limsup_{\e \to 0} \sup_{t \in [s,T]} \left| \e N_s \left(t^{\frac{1}{\e},s} \right)- \frac{t-s}{ M_\rho} \right| \leq \frac{2 \eta_n}{M_\rho} \hspace{3mm} \forall n
 \end{align*}
 
 Taking the limit as $\eta_n \to 0$ we get that :
  \begin{align*}
 & \lim_{\e \to 0} \sup_{t \in [s,T]} \left| \e N_s \left(t^{\frac{1}{\e},s} \right)- \frac{t-s}{ M_\rho} \right| =0 \mbox{ on } \Omega^*
 \end{align*}
 
 which shows that $(ii)=\circ(1)$. \\
 
 For the first term (i), we begin by the same trick:
  \begin{align*}
 & \sup_{t \in [s,T]} (i) \leq  \sup_{t \in [s,s+\eta_n]} (i) + \sup_{t \in [s+\eta_n,T]} (i)
  \end{align*}
 
 and on $\Omega^*$, $\forall n$:
  \begin{align*}
  & \sup_{t \in [s,s+\eta_n]} (i) \leq C \e N_s \left((s+\eta_n)^{\frac{1}{\e},s} \right) \\
 \Rightarrow& \limsup_{\e \to 0} \sup_{t \in [s,s+\eta_n]} (i) \leq C \frac{\eta_n}{M_\rho}  
  \end{align*}
  
Let $\oo \in \Omega^*$. Let $\delta>0$. Because $f \in \widehat{\D}$, $t \to \widehat{A}(t)f$ is continuous on $[s,T]$, and therefore uniformly continuous on it, therefore $\exists r=r(\delta)$, $r<\delta$ such that:
  \begin{align*}  
 |u_1-u_2|<r \Rightarrow || \widehat{A}  (u_1)f- \widehat{A}  (u_2)f || < \delta 
  \end{align*}  

Take $a:=\{a_j\}_{j=0..n_a}$ any partition of $[s,t]$ ($a_0=s$, $a_{n_a}=t$) such that $||a|| < c_0 r$, for some constant $c_0$ to be chosen. We have:
\begin{align*} 
& \sum_{k=1}^{N_s \left(t^{\frac{1}{\e},s} \right)} || V_\e(s,T_{k}^{\e,s}(s)) \widehat{A}  \left( T_{k}^{\e,s}(s)  \right)f-V_\e(s,t^\e_k(s,t)) \widehat{A}  \left( t^\e_k(s,t)  \right)f ||\\
= & \sum_{j=0}^{n_a-1} \sum_{k=N_s \left(a_j^{\frac{1}{\e},s} \right)+1}^{N_s \left(a_{j+1}^{\frac{1}{\e},s} \right)} || V_\e(s,T_{k}^{\e,s}(s)) \widehat{A}  \left( T_{k}^{\e,s}(s)  \right)f-V_\e(s,t^\e_k(s,t)) \widehat{A}  \left( t^\e_k(s,t)  \right)f ||
\end{align*} 

If $k \in \left[N_s \left(a_j^{\frac{1}{\e},s} \right)+1,N_s \left(a_{j+1}^{\frac{1}{\e},s} \right) \right]$, we have $T_{k}^{\e,s}(s) \in [a_j,a_{j+1}]$. Also, by definition of $\Omega^*$, there exists $r_1(\delta)>0$ such that for $\e < r_1$ we have $ \sup_{t \in [s,T]} \left| \e N_s \left(t^{\frac{1}{\e},s} \right)- \frac{t-s}{ M_\rho} \right| < c_1 r$ (for some $c_1$ to be chosen) so that:
\begin{align*} 
& t^\e_k(s,t) \leq s+ \e N_s \left(a_{j+1}^{\frac{1}{\e},s} \right) \frac{t-s}{\e n_\e(s,t)} \leq s+ \frac{t-s}{\e n_\e(s,t)} \left( c_1r+\frac{a_{j+1}-s}{M_\rho} \right) \\ &\leq s+M_\rho \left(  c_1r+\frac{a_{j+1}-s}{M_\rho} \right)=a_{j+1}+ M_\rho c_1 r
\end{align*} 

We also have:
\begin{align*} 
& t^\e_k(s,t) \geq s+ \e (N_s \left(a_{j}^{\frac{1}{\e},s} \right)+1) \frac{t-s}{\e n_\e(s,t)}
\end{align*} 

 Since $t \geq s+\eta_n$, we have $\frac{t-s}{\e M_\rho} \geq  \frac{\eta_n}{\e M_\rho} $ $\forall t \in [s+\eta_n,T]$. Therefore we can find a $r_2(n, \delta)$ such that for $\e<r_2$: $\frac{\lambda_t^\e}{1+\lfloor \lambda_t^\e \rfloor} \geq 1-c_2r$ $\forall t \in [s+\eta_n,T]$, where $\lambda_t^\e:=\frac{t-s}{\e M_\rho}$ and $c_2$ a constant to be chosen. So for $\e < r_1 \wedge r_2$:
 \begin{align*} 
& t^\e_k(s,t)  \geq s+M_\rho (1-c_2r) (\frac{a_{j}-s}{M_\rho}- c_1r) \geq c_2 r s+(1-c_2r)a_j - M_\rho c_1 r
\end{align*} 
 
 so that, because $||a|| < c_0 r$ and choosing $c_0=\frac{1}{4}$, $c_1=\frac{1}{8M_\rho}$, $c_2=\frac{1}{4T}$:
 \begin{align*} 
& |T_{k}^{\e,s}(s)-t^\e_k(s,t)| \leq  c_0r +2M_\rho c_1 r+c_2rT =\frac{3}{4}r<r
\end{align*}

Now we have:
 \begin{align*}  
 & || V_\e(s,T_{k}^{\e,s}(s)) \widehat{A}  \left( T_{k}^{\e,s}(s)  \right)f-V_\e(s,t^\e_k(s,t)) \widehat{A}  \left( t^\e_k(s,t)  \right)f ||\\
 \leq & || \widehat{A}  \left( T_{k}^{\e,s}(s)  \right)f- \widehat{A}  \left( t^\e_k(s,t)  \right)f || + || (V_\e(s,T_{k}^{\e,s}(s)) -V_\e(s,t^\e_k(s,t))) \widehat{A}  \left( t^\e_k(s,t)  \right)f ||
 \end{align*} 
 
 Let $\e < r_1 \wedge r_2  $. For the first term we have immediately:
  \begin{align*}  
 |T_{k}^{\e,s}(s)-t^\e_k(s,t)|<r \Rightarrow || \widehat{A}  \left( T_{k}^{\e,s}(s)  \right)f- \widehat{A}  \left( t^\e_k(s,t)  \right)f || < \delta \\
  \end{align*}

  For the other term, using \ref{19} as we did in the proof of \ref{20} we get that:
 \begin{align*}   
 &  || (V_\e(s,T_{k}^{\e,s}(s)) -V_\e(s,t^\e_k(s,t))) \widehat{A}  \left( t^\e_k(s,t)  \right)f || \\
 & \leq r C_1+ C_2 \e N \left((s+3r)^{\frac{1}{\e},s} \right)  + C_2 \e \sup_{t \in [s+2r,T-r]} \left[N \left((t+r)^{\frac{1}{\e},s} \right)-N\left((t-r)^{\frac{1}{\e},s} \right) \right] \\
 \end{align*}  
  
  By definition of $\Omega^*$:
   \begin{align*} 
  & \e N \left((s+3r)^{\frac{1}{\e},s} \right) \leq \frac{3r}{M_\rho} +c_1 r\\
    \end{align*} 
  
  and exactly as in the proof of \ref{20}:
   \begin{align*} 
  & \e \sup_{t \in [s+2r,T-r]} \left[N \left((t+r)^{\frac{1}{\e},s} \right)-N\left((t-r)^{\frac{1}{\e},s} \right) \right] < 2c_1 r + \frac{2r}{M_\rho}\\
   \end{align*} 
  
  So that all together:
   \begin{align*}   
 &  || (V_\e(s,T_{k}^{\e,s}(s)) -V_\e(s,t^\e_k(s,t))) \widehat{A}  \left( t^\e_k(s,t)  \right)f || \leq C_0 r < C_0 \delta \mbox{ for some } C_0 \in \RR^{+}\\
 \end{align*}  
  
  Finally we get that for $\e<r_1\wedge r_2 $ and on $\Omega^*$
 \begin{align*}    
  & \sup_{t \in [s+\eta_n,T]} (i)  \leq (1+C_0) \e N_s \left(T^{\frac{1}{\e},s} \right) \delta \leq C_3 \delta \\
 \end{align*}  
 
 This means that on $\Omega^*$: $\lim_{\e \to 0} \sup_{t \in [s+\eta_n,T]} (i)=0$ $\forall n$ and therefore:
  \begin{align*}
 & \limsup_{\e \to 0} \sup_{t \in [s,T]} (i) \leq \frac{\eta_n}{M_\rho}  \hspace{3mm} \forall n \\
 \end{align*}
 
 Taking the limit as $\eta_n \to 0$ we get that $\limsup_{\e \to 0} \sup_{t \in [s,T]} (i)=0$ on $\Omega^*$, which concludes the proof.
 \end{proof}

\vspace{4mm}

\begin{lem}
\label{26}
Assume \textbf{(A0)} and let $f \in \widehat{\D}$, $s \in J$ and assume that for some sequence $\e_n \to 0$, $\widetilde{M}_\bullet^{\e_n}(s)f \Rightarrow M^0_\bullet(s,f)$ for some $ D(J(s),Y)-$valued random variable $M^0_\bullet(s,f)$. Then $\widetilde{M}^0_\bullet(s)f =0$ a.e..
\end{lem}

 \begin{proof} 
Let $\ell \in Y^*$. Then $\ell(\widetilde{M}_\bullet^{\e_n}(s)f )$ is a real-valued martingale by \ref{24} and using problem 13, section 3.11 of \cite{EK} together with the continuous mapping theorem, we get that $\ell(\widetilde{M}_\bullet^{\e_n}(s)f ) \Rightarrow \ell(M_\bullet^0(s,f))$. The fact that the process $\ell(M_\bullet^0(s,f))$ is a local martingale with respect to its natural filtration is a direct application of a result that can be found in \cite{JS} (corollary 1.19, chapter IX). We have to show that the jumps of $\ell(\widetilde{M}_\bullet^{\e_n}(s)f )$ are uniformly bounded. We have:
\begin{align*}
&M_{k+1}^{\e_n}(s)f-M_k^{\e_n}(s)f=\Delta V^{\e_n}_{k+1}(s) f_{k+1}^{\e_n}-\EE[\Delta V^{\e_n}_{k+1}(s) f_{k+1}^{\e_n} | \F_k(s)] \\
& \mbox{where } \Delta V^{\e_n}_{k+1}(s) f_{k+1}^{\e_n} :=V^{\e_n}_{k+1}(s) f_{k+1}^{\e_n}-V^{\e_n}_{k}(s) f_{k}^{\e_n}
\end{align*}

and where $f_{k+1}^{\e_n}$ has been defined in the proof of \ref{25}. Then we have:
\begin{align*}  
 \left|\Delta \ell(\widetilde{M}_t^{\e_n}(s)f ) \right|= \left| \ell \left(\Delta \widetilde{M}_t^{\e_n}(s)f \right)\right|& \leq 2 ||\ell|| \, \sup_{k \in \NN} || \Delta V^{\e_n}_{k+1}(s) f_{k+1}^{\e_n}||\\
& \leq 4 ||\ell||(||f||+||f_1||)
 \end{align*}

Now, we observe that $\ell(\widetilde{M}_\bullet^{\e_n}(s)f )$ is square-integrable since using its definition we get immediately that for some $C_1, C_2 \in \RR^+$:
\begin{align*}
 & |\ell(\widetilde{M}_t^{\e_n}(s)f )| \leq C_1  N_s \left(t^{\frac{1}{\e_n},s} \right) + C_2
 \end{align*}

and that $\EE \left[ N_s \left(t^{\frac{1}{\e_n},s} \right)^2\right] < \infty$ by \cite{H}. Now according to theorem 5.1 in \cite{W}, if we prove that $\forall t$: $\left< \ell(\widetilde{M}_\bullet^{\e_n}(s)f )\right>_t \Rightarrow 0$, then we get that $\ell(M_\bullet^0(s,f))$ is equal to the zero process, up to indistinguishability. In particular, it yields that $\forall t \in J(s)$, $\forall \ell \in Y^*$: $\ell(M_t^0(s,f))=0$ a.e.. Now, by \cite{LT} (chapter 2), we know that the dual space of every separable Banach space has a countable total subset, more precisely there exists a countable subset $S \subseteq Y^*$ such that $\forall g \in Y$:
 \begin{align*}
 & (\ell(g)=0 \hspace{3mm} \forall \ell \in S) \Rightarrow g=0
 \end{align*}
 since $\PP[\ell(M_t^0(s,f))=0, \, \forall \ell \in S]=1$, we get $M_t^0(s,f)=0$ a.e., i.e. $M_\bullet^0(s,f)$ is a modification of the zero process. Since both process have a.e. right-continuous paths, they are in fact indistinguishable (see \cite{KS}). And so $M_\bullet^0(s,f)=0$ a.e..\\
 
Now it remains to show that $\forall t$, the quadratic variation $\left< \ell(\widetilde{M}_\bullet^{\e_n}(s)f )\right>_t \Rightarrow 0$. Using the definition of $\widetilde{M}_\bullet^{\e_n}(s)f $ in \ref{24}, we get that:
\begin{align*}
& \left< \ell(\widetilde{M}_\bullet^{\e_n}(s)f )\right>_t= \sum_{k=0}^{N_s \left(t^{\frac{1}{\e_n},s} \right)} \EE \left[ \left. \ell^2(M_{k+1}^{\e_n}(s)f-M_k^{\e_n}(s)f) \right| \F_k(s) \right] 
\end{align*}

and:
\begin{align*}
&M_{k+1}^{\e_n}(s)f-M_k^{\e_n}(s)f=\Delta V^{\e_n}_{k+1}(s) f_{k+1}^{\e_n}-\EE[\Delta V^{\e_n}_{k+1}(s) f_{k+1}^{\e_n} | \F_k(s)] \\
& \mbox{where } \Delta V^{\e_n}_{k+1}(s) f_{k+1}^{\e_n} :=V^{\e_n}_{k+1}(s) f_{k+1}^{\e_n}-V^{\e_n}_{k}(s) f_{k}^{\e_n}
\end{align*}

In the proof of \ref{25} we proved that if $f \in \widehat{D}$, then:
\begin{align*}
& \left| \left|  \Delta V^{\e_n}_{k+1}(s) f_{k+1}^{\e_n} \right| \right| \leq C \e_n
\end{align*}

and therefore that:
\begin{align*}
& \left< \ell(\widetilde{M}_\bullet^{\e_n}(s)f )\right>_t \leq 4 C^2 ||\ell||^2 \e_n^2 N_s \left(t^{\frac{1}{\e_n},s} \right) \stackrel{a.e.}{\to}0
\end{align*}

because $ \e_n N_s \left(t^{\frac{1}{\e_n},s} \right) \stackrel{a.e.}{\to}\frac{t-s}{M_\rho}$.
\end{proof} 
 
To prove our next theorem, we need the following 2 assumptions, usually fulfilled in practice (see section 5):
\begin{align*} 
& \textbf{(A2)} \hspace{3mm} \forall s \in J, \,\{V_\e(s,\bullet) \widehat{A}(\bullet)\} \mbox{ satisfies the compact containment criterion in $\widehat{\D}$}  \\
& \textbf{(A3)} \hspace{3mm} \widehat{A} \mbox{ is the generator of a regular inhomogeneous $Y-$semigroup $\widehat{\G}$}\\
\end{align*} 

Let's make some comments about these assumptions. About \textbf{(A3)}, we first notice by $\ref{9}$, that if we assume $Y_1$ to be dense in $Y$ as it will be the case, then $\widehat{\G}$ is unique. \textbf{(A3)} will in general be satisfied in practice because from the expression of $\widehat{A}$ (see \ref{112}), we can write $\widehat{\G}$ explicitly as some average of the semigroups $\{\G_x\}_{x \in X}$ (see section 5).\\ 

\textbf{(A2)} yields that $\{V_\e(s,\bullet) \widehat{A}(\bullet)\}$ is $\widehat{\D}-$C-relatively compact. Indeed, in theorems \ref{15} and \ref{16}, the 2 other conditions are proved exactly the same way as in the proofs of theorems \ref{20} and \ref{21}, using the continuity of $\widehat{A}$ on $\widehat{\D}$ (recall $\bigcap_{x \in X} \D(A'_x) \subseteq \widehat{\D}$).\\

The proof of the compact containment criterion is linked to the nature of the Banach space $Y$, as we saw in \ref{17}, \ref{18}. The idea to see how it can be proved here is the following:\\
 
 For fixed $t_0 \in J(s)$ and $f \in \widehat{\D}$, we know that $\widehat{A}(t_0)f \in Y_1$, and since $\{V_\e\}$ satisfies the compact containment criterion in $Y_1$ (\textbf{(A1)}), $\exists K=K(t_0,T,\Delta,f) \subseteq Y$ compact set such that:
\begin{align*} 
& \liminf_{\e \to 0} \PP[V_\e(s,t)\widehat{A}(t_0)f  \in K \hspace{3mm} \forall t \in [s,T]] \geq 1-\Delta
\end{align*} 

What we want to prove is that, in some way, we can "put the $ \forall t_0$ inside the probability", namely $\exists K=K(T,\Delta,f) \subseteq Y$ compact set such that:
\begin{align*} 
& \liminf_{\e \to 0} \PP[V_\e(s,t)\widehat{A}(t)f  \in K \hspace{3mm} \forall t \in [s,T] ] \geq 1-\Delta
\end{align*} 

In practice it is easy to prove, because the dependence in time of the generators $A_x(t)$ doesn't cause any problems when dealing with compacity in $Y$. For example, it is straightforward that it is satisfied in the case of \ref{17}, using some simple boundedness conditions. As for the case of \ref{18}, the reason why it is satisfied is that  usually (see section 5): the set $A_\e$ doesn't depend on $f$ (it is only linked with the Markov Renewal process, i.e. the only source of randomness) and the family of functions:
\begin{align*} 
& \{V_\e(s,t)(\omega)\widehat{A}(t)f: t \in[s,T], \e \in (0,1], \omega \in A_\e \} 
\end{align*} 

converge uniformly to $0$ at infinity, are equicontinuous and uniformly bounded. The latter will be true because typically we'll have that the following family of functions converge uniformly to $0$ at infinity, are equicontinuous and uniformly bounded:
\begin{align*} 
& \{\widehat{A}(t)f: t \in[s,T] \} 
\end{align*} 

That is, the time-dependence of $\widehat{A}$ doesn't affect the 3 previous features.\\
 
 \begin{thm}
\label{27}
Assume that $\widehat{\D}$ contains a countable family that is dense in both $Y_1$ and $Y$. Under assumptions  \textbf{(A0)}, \textbf{(A1)}, \textbf{(A2)}, \textbf{(A3)}, we have that $\widehat{\G}$ is a $\B(Y)-$contraction regular inhomogeneous $Y$-semigroup. Further, for every countable family $\{f_k\} \subseteq Y$ and $s \in J$ we have the weak convergence in the Skorohod topology $D(J(s),Y^\infty)$:
 \begin{align*}
 & (V_{\e}(s,\bullet)f_k: k \in \NN) \Rightarrow (\widehat{\G}(s, \bullet)f_k: k \in \NN)
 \end{align*}
\end{thm}
 
 \textbf{Remark:} typically, $Y=C_0^{n_1}(\RR^d)$, $Y_1=C_0^{n_2}(\RR^d)$, and the countable family is chosen in $C_0^{n_3}(\RR^d)$, for $n_1 \leq n_2 \leq n_3$.\\
 
 \begin{proof}
 The proof can be splitted in the following steps:
 \begin{itemize}
 \item Take $\{g_k\}_{k \in \NN^*} \subseteq \widehat{\D}$ a countable family that is both dense in $Y_1$ and $Y$ and because marginal tightness implies countable tightness, get the tightness of $(V_{\e_n}(s, \bullet)\widehat{A}(\bullet)g_k, V_{\e_n}(s, \bullet) g_k : k \geq 1)$ using \textbf{(A1)}, \textbf{(A2)}. By tightness, take one weakly converging sequence $\e_n$ and the goal is to show that the limit is unique in distribution.
 \item Carry the problem to a new probability space $(\Omega', \F', \PP')$ where convergence holds a.e., by the Skorohod representation theorem.
 \item On a subset $\Omega'_* \subseteq \Omega'$ such that $\PP'(\Omega'_*)=1$ and using density (and completeness), construct a stochastic process $V'_{0}(s, \bullet)$ with sample paths in $D(J(s),\B(Y))$ such that on $\Omega'_*$ we have the convergence in the Skorohod topology (where the subscript $'$ denotes random variables on $\Omega'$): 
 \begin{align*}
 & V'_{\e_n}(s, \bullet,g_k,\widehat{A}) \to V'_{0}(s, \bullet)\widehat{A}(\bullet)g_k, \hspace{3mm} \forall k \\
 & V'_{\e_n}(s, \bullet)g_k \to V'_{0}(s, \bullet)g_k
  \end{align*}
  where $\forall n$:
  \begin{align*}
 (V_{\e_n}(s, \bullet)\widehat{A}(\bullet)g_k, V_{\e_n}(s, \bullet) g_k : k \geq 1) \stackrel{d}{=} (V'_{\e_n}(s, \bullet,g_k,\widehat{A}),V'_{\e_n}(s, \bullet, g_k)  : k \geq 1) 
   \end{align*} 
   \item Show the convergence of the Riemann sum of \ref{25} to an integral and use \ref{26} to show the convergence of the martingale of \ref{25} to the zero process on $\Omega'_*$, so that $V'_{0}$ will satisfy the following integral equation on $\Omega'_*$:
 \begin{align*}
 & V'_{0}(s,\bullet)f=f+\int_s^\bullet V'_0(s,u) \widehat{A}(u)fdu, \hspace{3mm} \forall f \in Y_1
 \end{align*}
 \item By \textbf{(A3)} and unicity of the Cauchy problem \ref{10}, show that $\forall f \in Y$, $\forall s \in J$, $V'_{0}(s,\bullet)f$ must be the constant random variable $\widehat{\G}(s,\bullet)f$ in $C(J(s),Y)$.\\
 \end{itemize}

 Let $\{g_k\}_{k \in \NN^*} \subseteq \widehat{\D}$ a countable family that is both dense in $Y_1$ and $Y$. As we mentioned it, \textbf{(A2)} is enough to prove that $\{V_\e(s,\bullet) \widehat{A}(\bullet)\}$ is $\widehat{\D}-$C-relatively compact. Indeed, in theorems \ref{15} and \ref{16}, the 2 other conditions are proved exactly the same way as in the proofs of theorems \ref{20} and \ref{21}, using the continuity of $\widehat{A}$ on $\widehat{\D}$ (recall $\bigcap_{x \in X} \D(A'_x) \subseteq \widehat{\D}$). By the latter and \ref{21}, the family $\{V_\e(s, \bullet)\widehat{A}(\bullet)g_k, V_\e(s, \bullet) g_k : k \geq 1\}$ is C-relatively compact in $D(J(s),Y)^\infty$, and in fact in $D(J(s),Y^\infty)$ since the limit points are continuous. Take a converging sequence $\e_n$:
 \begin{align*}
 & (V_{\e_n}(s, \bullet)\widehat{A}(\bullet)g_k, V_{\e_n}(s, \bullet) g_k : k \geq 1) \Rightarrow (\alpha(s,\bullet,g_k), v_{0}(s, \bullet,g_k) : k \geq 1)
 \end{align*}
 
 By Skorohod representation theorem, we can consider this convergence to be almost sure, i.e. there exists a probability space $(\Omega', \F', \PP')$ and random variables with the same distributions as the previous ones (denoted by the subscript $'$), such that:
  \begin{align*}
 & (V'_{\e_n}(s, \bullet,g_k,\widehat{A}),V'_{\e_n}(s, \bullet, g_k)  : k \geq 1) \stackrel{a.e.}{\to} (\alpha'(s,\bullet,g_k),v'_{0}(s, \bullet,g_k) : k \geq 1)
 \end{align*}
 
 Let $g \in Y$. By density, there exists a sequence  $(g^0_k)_{k \in \NN^*} \subseteq \{g_k\}_{k \in \NN^*}$ so that $g^0_k \to g$. Let:
 \begin{align*}
 & B(r):= \{(x,y) \in D(J(s),Y) \times D(J(s),Y): d(x,y) \leq r \}=d^{-1} ([0,r])
 \end{align*}
 
 As the preimage of a closed set under a continuous function, $B(r)$ is closed in $D(J(s),Y) \times D(J(s),Y)$. Note that  because $D(J(s),Y) $ is separable, the Borel sigma-algebras $Bor(D(J(s),Y) \times D(J(s),Y))$ and $Bor(D(J(s),Y)) \otimes Bor(D(J(s),Y))$ agree, so we shall not worry about that.\\
 
 Since for every $k_1$, $k_2$, $n$ the pairs $\{V'_{\e_n}(s, \bullet, g^0_{k_1}), V'_{\e_n}(s, \bullet, g^0_{k_2}) \}$ and $\{V_{\e_n}(s, \bullet) g^0_{k_1}, V_{\e_n}(s, \bullet) g^0_{k_2} \}$ have the same distributions, we get that:
  \begin{align*}
  & \PP' \left[ d(V'_{\e_n}(s, \bullet, g^0_{k_1}), V'_{\e_n}(s, \bullet, g^0_{k_2})) \leq ||g^0_{k_1}-g^0_{k_2}||  \right] \\
  =& \PP' \left[ (V'_{\e_n}(s, \bullet, g^0_{k_1}), V'_{\e_n}(s, \bullet, g^0_{k_2})) \in B(||g^0_{k_1}-g^0_{k_2}||)  \right] \\
  =& \PP \left[ (V_{\e_n}(s, \bullet) g^0_{k_1}, V_{\e_n}(s, \bullet) g^0_{k_2}) \in B(||g^0_{k_1}-g^0_{k_2}||)  \right] \\
  =&  \PP \left[ d(V_{\e_n}(s, \bullet) g^0_{k_1}, V_{\e_n}(s, \bullet) g^0_{k_2}) \leq ||g^0_{k_1}-g^0_{k_2}|| \right] =1
  \end{align*}
 
 Let the subset $ \Omega'^* \subseteq  \Omega'$:
 \begin{align*}
& \Omega'^*:= \bigcap_{k_1,k_2,n} \left \{ d(V'_{\e_n}(s, \bullet, g_{k_1}), V'_{\e_n}(s, \bullet ,g_{k_2})) \leq ||g_{k_1}-g_{k_2}|| \right \}\\& \bigcap \left \{ \lim_{n \to \infty} (V'_{\e_n}(s, \bullet,g_k,\widehat{A}),V'_{\e_n}(s, \bullet, g_k)  : k \geq 1) = (\alpha'(s,\bullet,g_k),v'_{0}(s, \bullet,g_k) : k \geq 1) \right\}
\end{align*}
 
so that $\PP'( \Omega'^*)=1$. On $\Omega'^*$, the sequence $(V'_{\e_n}(s, \bullet, g^0_{k}))_{k \in \NN}$ is Cauchy in $D(J(s),Y)$ which is complete, therefore it converges to some $V'_{\e_n}(s, \bullet, g)$ as $k \to \infty$. To see that $V'_{\e_n}(s, \bullet, g)$ has the same distribution as $V_{\e_n}(s, \bullet) g$, we observe that $V_{\e_n}(s, \bullet) g^0_{k} \stackrel{a.e.}{\to}V_{\e_n}(s, \bullet) g$ (by contraction property of $V_{\e_n}$) and we just invoke the unicity of the limit in distribution, together with the fact that $\forall k$: $V'_{\e_n}(s, \bullet, g^0_{k})$ and $V_{\e_n}(s, \bullet) g^0_{k}$ have the same distributions. Note that all the $\{V'_{\e_n}(s, \bullet, g) : g \in Y, n \in \NN \}$ are defined on the common subset $\Omega'^*$.\\

We have on $\Omega'^*$ that $\forall n, k_1, k_2$: $d(V'_{\e_n}(s, \bullet, g^0_{k_1}), V'_{\e_n}(s, \bullet, g^0_{k_2})) \leq ||g^0_{k_1}-g^0_{k_2}||$. Since $d$ is continuous, we may take the limit as $n \to \infty$ and obtain that: 
 \begin{align*}
&d(v'_{0}(s, \bullet, g^0_{k_1}), v'_{0}(s, \bullet, g^0_{k_2})) \leq ||g^0_{k_1}-g^0_{k_2}||, 
 \end{align*}
which by completeness shows the convergence of the sequence $(v'_{0}(s, \bullet, g^0_{k}))_{k \in \NN}$ to some $v'_{0}(s, \bullet, g)$ (which belongs to $C(J(s),Y)$ as a limit in the Skorohod metric of elements of $C(J(s),Y)$). Now to see that the latter is the pointwise limit of $V'_{\e_n}(s, \bullet, g)$ on $\Omega'^*$, we observe that:

\begin{align*}
 d(v'_{0}(s, \bullet,g),V'_{\e_n}(s, \bullet,g))  &\leq  d(v'_{0}(s, \bullet,g),v'_{0}(s, \bullet,g^0_k)) + d(v'_{0}(s, \bullet,g^0_k),V'_{\e_n}(s, \bullet,g^0_k)) \\&+ d(V'_{\e_n}(s, \bullet,g^0_k),V'_{\e_n}(s, \bullet,g))
\end{align*}

Now, by continuity of $d$ , $|| \cdot ||$ and taking the limit as $p \to \infty$, we have on $\Omega'^*$:
\begin{align*}
& d(V'_{\e_n}(s, \bullet,g^0_k),V'_{\e_n}(s, \bullet,g^0_p)) \leq ||g^0_k-g^0_p|| \Rightarrow d(V'_{\e_n}(s, \bullet,g^0_k),V'_{\e_n}(s, \bullet,g)) \leq ||g^0_k-g||
\end{align*}

Therefore first choose $k$ such that the 1st and 3rd terms are small, then choose $n$ such that the 2nd term is small. \\

Now that $v'_{0}(s, \bullet, g)$ is well defined on $\Omega'^*$ for every $g \in Y$, we want to show that on some subset $\Omega'_0 \subseteq \Omega'^*$ such that $\PP'(\Omega'_0)=1$, we have $v'_{0}(s, \bullet, h+\lambda g)=v'_{0}(s, \bullet, h)+\lambda v'_{0}(s, \bullet,  g)$ and $||v'_{0}(s, t, g)|| \leq ||g||$, $\forall h,g \in Y$, $\forall \lambda \in \RR$, $\forall t \in J(s)$, namely that for every $\oo \in \Omega'_0$ and $t \in J(s)$, $v'_{0}(s, t, \bullet)(\oo)$ is a $\B(Y)-$contraction. After having proved the latter, we will adopt the notation $V'_0(s,t)(\oo):=v'_{0}(s, t, \bullet)(\oo)$ to emphasize this fact.\\

Let $h,g \in Y$, $\lambda \in \RR$. There exists sequences $(h^0_k)_{k \in \NN^*} \, , (g^0_k)_{k \in \NN^*} \subseteq \{g_k\}_{k \in \NN^*}$ so that $g^0_k \to g$, $h^0_k \to h$ and $(\lambda_k)_{k \in \NN^*}  \subseteq \QQ$ so that $\lambda_k \to \lambda$. \\

First observe the following equality in distribution:
\begin{align*}
& (V_{\e_n}(s, \bullet) (h^0_{k}+\lambda_k g^0_{k}), V_{\e_n}(s, \bullet) h^0_{k}, \lambda_k V_{\e_n}(s, \bullet)  g^0_{k}) \stackrel{d}{=} (V'_{\e_n}(s, \bullet,h^0_{k}+\lambda_k g^0_{k}), V'_{\e_n}(s, \bullet, h^0_{k}), \lambda_k V'_{\e_n}(s, \bullet, g^0_{k}))
\end{align*}

This is because $\forall k$, there exists a sequence $(\beta^k_p)_{p \in \NN^*} \subseteq \{g_p\}_{p \in \NN^*}$ such that $\beta^k_p \to  h^0_{k}+\lambda_k g^0_{k}$ and by contraction property of $V_{\e_n}$ and construction of $V'_{\e_n}$ :
\begin{align*}
& (V_{\e_n}(s, \bullet) \beta^k_p, V_{\e_n}(s, \bullet) h^0_{k}, \lambda_k V_{\e_n}(s, \bullet)  g^0_{k}) \stackrel{a.e. \,, p \to \infty}{\to} (V_{\e_n}(s, \bullet) (h^0_{k}+\lambda_k g^0_{k}), V_{\e_n}(s, \bullet) h^0_{k}, \lambda_k V_{\e_n}(s, \bullet)  g^0_{k}) \\
& (V'_{\e_n}(s, \bullet,\beta^k_p), V'_{\e_n}(s, \bullet, h^0_{k}), \lambda_k V'_{\e_n}(s, \bullet,g^0_{k})) \stackrel{a.e.  \,, p \to \infty}{\to} (V'_{\e_n}(s, \bullet,h^0_{k}+\lambda_k g^0_{k}), V'_{\e_n}(s, \bullet, h^0_{k}), \lambda_k V'_{\e_n}(s, \bullet, g^0_{k}))
\end{align*}

and that:
\begin{align*}
& (V_{\e_n}(s, \bullet) \beta^k_p, V_{\e_n}(s, \bullet) h^0_{k}, \lambda_k V_{\e_n}(s, \bullet)  g^0_{k}) \stackrel{d}{=} (V'_{\e_n}(s, \bullet,\beta^k_p), V'_{\e_n}(s, \bullet, h^0_{k}), \lambda_k V'_{\e_n}(s, \bullet,g^0_{k}))
\end{align*}

and we conclude by unicity of the limit in distribution. In particular we get that $V'_{\e_n}(s, \bullet,h^0_{k}+\lambda_k g^0_{k})-V'_{\e_n}(s, \bullet, h^0_{k})-\lambda_k V'_{\e_n}(s, \bullet, g^0_{k}) = 0$ a.e.. Let the subset $\Omega'^{**} \subseteq \Omega'^*$:
\begin{align*}
& \Omega'^{**} = \Omega'^* \cap \bigcap_{n,k_1,k_2 \in \NN, \, r \in \QQ} \{V'_{\e_n}(s, \bullet, g_{k_1}+r g_{k_2})-V'_{\e_n}(s, \bullet, g_{k_1})-r V'_{\e_n}(s, \bullet, g_{k_2})=0\}
\end{align*}

so that $\PP'(\Omega'^{**})=1$. Because the limit points $v'_0$ are continuous, we have on $\Omega'^{**}$ that:
\begin{align*}
& V'_{\e_n}(s, \bullet,h^0_{k}+\lambda_k g^0_{k})-V'_{\e_n}(s, \bullet, h^0_{k})-\lambda_k V'_{\e_n}(s, \bullet, g^0_{k}) \stackrel{n \to \infty}{\to}  v'_{0}(s, \bullet ,h^0_{k}+\lambda_k g^0_{k})-v'_{0}(s, \bullet,h^0_{k})-\lambda_k v'_{0}(s, \bullet,g^0_{k})
\end{align*}
and therefore that $v'_{0}(s, \bullet ,h^0_{k}+\lambda_k g^0_{k})-v'_{0}(s, \bullet,h^0_{k})-\lambda_k v'_{0}(s, \bullet,g^0_{k})=0$. Taking the limit as $k \to \infty$ (and using again the fact that the $v'_0$ are continuous) yields $v'_{0}(s, \bullet, h+\lambda g)=v'_{0}(s, \bullet, h)+\lambda v'_{0}(s, \bullet,  g)$.\\

Now to show that $||v'_{0}(s, t, g)|| \leq ||g||$ $\forall g \in Y, \forall t \in J(s)$ on some subset $\Omega'_0 \subseteq \Omega'^{**}$ such that $\PP'(\Omega'_0)=1$, we observe that because $V'_{\e_n}(s, \bullet,g^0_k) \to v'_{0}(s, \bullet, g^0_k)$, and $v'_{0}(s, \bullet ,g^0_k) \to v'_{0}(s, \bullet, g)$: by \cite{EK} (chapter 3, proposition 5.2) and using continuity of $v'_0$ we get that $\forall t$: $v'_{0}(s, t ,g^0_k) \to v'_{0}(s, t, g)$, $V'_{\e_n}(s, t ,g^0_k) \to v'_{0}(s, t, g^0_k)$. Since $V'_{\e_n}(s, \bullet ,g_p)$ and $V_{\e_n}(s, \bullet)g_p$ have the same distribution, denoting $S_p:=\{x \in D(J(s),Y): ||x(t)|| \leq ||g_p|| \, \forall t \in J(s)\}$ we get that $\PP'[V'_{\e_n}(s, \bullet ,g_p) \in S_p]=1$, say on $\Omega'_{n,p}$. Let $\Omega'_0:= \bigcap_{n,p} \Omega'_{n,p} \cap \Omega'^{**}$. Let $t \in J(s)$. We have on $\Omega'_0$ that: $||V'_{\e_n}(s, t ,g^0_k)|| \leq ||g^0_k||$. Taking the limit as $n \to \infty$, we get $|| v'_{0}(s, t, g^0_k)|| \leq ||g^0_k||$, and then as $k \to \infty$ we get $|| v'_{0}(s, t, g)|| \leq ||g||$.\\

 Now let's take some $g_p$ and show that on some $\Omega'_{00} \subseteq \Omega'_0$ such that $\PP(\Omega'_{00})=1$:
 \begin{align*}
 & \e_n \sum_{k=1}^{n_{\e_n}(s,\bullet)} V'_{\e_n}(s, t^{\e_n}_k(s,\bullet),g_p,\widehat{A}) \to \frac{1}{M_\rho} \int_s^\bullet V'_0(s,u) \widehat{A}(u)g_p du \\
 \end{align*}
 
 First, let's prove that a.e. (say on some $\Omega'_{00} \subseteq \Omega'_{0}$): $\forall p$, $V'_0(s,\bullet) \widehat{A}(\bullet)g_p=\alpha'(s,\bullet,g_p)$. By what we did before we know that $\forall h \in Y$: 
  \begin{align*}
 & (V'_{\e_n}(s, \bullet,g_p,\widehat{A}),V'_{\e_n}(s, \bullet, h)) \stackrel{d}{=} (V_{\e_n}(s, \bullet)\widehat{A}(\bullet)g_p, V_{\e_n}(s, \bullet) h) \\
  \end{align*}
  
   In particular taking $t \in J(s)$ and $h=\widehat{A}(t)g_p$:
   \begin{align*}
 & (V'_{\e_n}(s, t,g_p,\widehat{A}),V'_{\e_n}(s, t, \widehat{A}(t)g_p)) \stackrel{d}{=} (V_{\e_n}(s, t)\widehat{A}(t)g_p, V_{\e_n}(s, t) \widehat{A}(t)g_p) \\
  \end{align*} 
  
  And therefore:
  \begin{align*} 
 &  \PP'[V'_{\e_n}(s, t,g_p,\widehat{A})-V'_{\e_n}(s, t, \widehat{A}(t)g_p) =0]=1 \mbox{, say on } \Omega'_{n,t,p} \\
  \end{align*} 
 
 and on the other hand by continuity of the limits, we have on $ \Omega'_t:=\bigcap_{n,p} \Omega'_{n,t,p} \cap \Omega'_{0}$: 
\begin{align*} 
 & 0=V'_{\e_n}(s, t,g_p,\widehat{A})-V'_{\e_n}(s, t, \widehat{A}(t)g_p)  \to \al'(s, t,g_p)-v'_{0}(s, t, \widehat{A}(t)g_p) \\
  \end{align*}
  
 Let $\Omega'_{00}:=  \bigcap_{t \in \QQ} \Omega'_t $. Now let $t \in J(s)$ and a sequence of rationals $t_k \to t$. On $\Omega'_{00}$ we have that:
 \begin{align*} 
& 0=\al'(s, t_k,g_p)-v'_{0}(s, t_k, \widehat{A}(t_k)g_p)=\al'(s, t_k,g_p)-V'_{0}(s, t_k) \widehat{A}(t_k)g_p\\
  \end{align*}
  
  We have $\al'(s, t_k,g_p) \to \al'(s, t,g_p)$ by continuity of $\al'$. And:
  \begin{align*} 
  & ||V'_{0}(s, t_k) \widehat{A}(t_k)g_p-V'_{0}(s, t) \widehat{A}(t)g_p|| \\&\leq ||V'_{0}(s, t_k) (\widehat{A}(t_k)g_p-\widehat{A}(t)g_p)|| + ||V'_{0}(s, t_k) \widehat{A}(t)g_p-V'_{0}(s, t) \widehat{A}(t)g_p||\\
  & \leq  ||\widehat{A}(t_k)g_p-\widehat{A}(t)g_p|| + ||V'_{0}(s, t_k) \widehat{A}(t)g_p-V'_{0}(s, t) \widehat{A}(t)g_p|| \\
  \end{align*} 
  
  where the last inequality used linearity and contraction property of $V'_0$. The first term goes to $0$ by continuity of $\widehat{A}$ on $\widehat{D}$, and the second by continuity of $V'_{0}(s, \bullet) \widehat{A}(t)g_p$.\\
   
 Now back to the convergence of the Riemann sum, we have on $\Omega'_{00}$:
  \begin{align*}
 &d \left(\e_n \sum_{k=1}^{n_{\e_n}(s,\bullet)} V'_{\e_n}(s, t^{\e_n}_k(s,\bullet),g_p,\widehat{A}) , \frac{1}{M_\rho} \int_s^\bullet \al'(s, u,g_p)du \right)\\
 \leq &  \underbrace{d \left(\e_n \sum_{k=1}^{n_{\e_n}(s,\bullet)} V'_{\e_n}(s, t^{\e_n}_k(s,\bullet),g_p,\widehat{A}) , \e_n \sum_{k=1}^{n_{\e_n}(s,\bullet)} \al'(s, t^{\e_n}_k(s,\bullet),g_p) \right)}_{(i)} + \\
 & \underbrace{d \left( \e_n \sum_{k=1}^{n_{\e_n}(s,\bullet)} \al'(s, t^{\e_n}_k(s,\bullet),g_p) , \frac{1}{M_\rho} \int_s^\bullet \al'(s, u,g_p) du  \right)}_{(ii)}
 \end{align*}
 
 Let $T>0$. Because $V'_{\e_n}(s, \bullet ,g_p,\widehat{A}) \to  \al'(s, \bullet,g_p)$ and $\al'$ is continuous, the convergence in the Skorohod topology is equivalent to convergence in the uniform topology. In particular :
 \begin{align*}
 & \sup_{t \in [s,T]} ||V'_{\e_n}(s, t ,g_p,\widehat{A}) -  \al'(s, t,g_p) || \to 0
 \end{align*}
 
 Therefore we get on $\Omega'_{00}$:
  \begin{align*}
 & \sup_{t \in [s,T]} \left| \left| \e_n \sum_{k=1}^{n_{\e_n}(s,t)} V'_{\e_n}(s, t^{\e_n}_k(s,t),g_p,\widehat{A}) - \e_n \sum_{k=1}^{n_{\e_n}(s,t)} \al'(s, t^{\e_n}_k(s,t),g_p) \right| \right| \\
 & \leq \e_n n_{\e_n}(s,T) \sup_{t \in [s,T]} ||V'_{\e_n}(s, t ,g_p,\widehat{A}) -  \al'(s, t,g_p) || \to 0
 \end{align*}
 
 since $\e_n n_{\e_n}(s,T) \to \frac{T-s}{M_\rho}$. By remark \ref{22} we get $(i) \to 0$.\\
 
 For (ii), we proceed similarly as in \ref{25}. We let a sequence of positive numbers $\eta_m \to 0$. We have:
 \begin{align*} 
 &\sup_{t \in [s,T]} \left| \left| \e_n \sum_{k=1}^{n_{\e_n}(s,t)} \al'(s, t^{\e_n}_k(s,t),g_p) - \frac{1}{M_\rho} \int_s^t \al'(s, u,g_p) du \right| \right|\\
 & \leq \sup_{t \in [s,s+\eta_m]} \left| \left| \e_n \sum_{k=1}^{n_{\e_n}(s,t)} \al'(s, t^{\e_n}_k(s,t),g_p) - \frac{1}{M_\rho} \int_s^t \al'(s, u,g_p) du \right| \right| \\ &+ \sup_{t \in [s+\eta_m,T]}\left| \left| \e_n \sum_{k=1}^{n_{\e_n}(s,t)} \al'(s, t^{\e_n}_k(s,t),g_p) - \frac{1}{M_\rho} \int_s^t \al'(s, u,g_p) du \right| \right|
  \end{align*}
 
 Because on $\Omega'_{00}$, $\alpha'(s,\bullet,g_p)=V'_0(s,\bullet) \widehat{A}(\bullet)g_p$ and by contraction property of $V'_0$ we get that: 
 \begin{align*}
  & \sup_{t \in [s,s+\eta_m]}\left| \left| \e_n \sum_{k=1}^{n_{\e_n}(s,t)} \al'(s, t^{\e_n}_k(s,t),g_p) - \frac{1}{M_\rho} \int_s^t \al'(s, u,g_p) du \right| \right| \leq C_1  \e_n n_{\e_n}(s,s+\eta_m)+C_2 \eta_m \\
  & \Rightarrow \limsup_{n \to \infty} \sup_{t \in [s,s+\eta_m]}\left| \left| \e_n \sum_{k=1}^{n_{\e_n}(s,t)} \al'(s, t^{\e_n}_k(s,t),g_p) - \frac{1}{M_\rho} \int_s^t \al'(s, u,g_p) du \right| \right| \leq C_3 \eta_m
 \end{align*}
 
 For the other term, we proceed exactly as in the proof of  \ref{25}: because $t \geq \eta_m$, we can refine as we wish the partition $\{t^{\e_n}_k(s,t)\}$ uniformly in $t \in [s+\eta_m,T]$, namely:  $\forall \delta>0$, $\exists r(\delta,m)>0$ such that if $\e_n<r$: 
  \begin{align*} 
 &t^{\e_n}_{k+1}(s,t)- t^{\e_n}_k(s,t) < r \hspace{3mm} \forall t \in [s+\eta_m,T], \forall k \in [|1,n_{\e_n}(s,t)|]
  \end{align*}
 
 and get the convergence of the riemann sum to the Riemann integral uniformly on $[s+\eta_m,T]$, i.e. 
  \begin{align*}
  & \sup_{t \in [s+\eta_m,T]} \left| \left| \e_n \sum_{k=1}^{n_{\e_n}(s,t)} \al'(s, t^{\e_n}_k(s,t),g_p) - \frac{1}{M_\rho} \int_s^t \al'(s, u,g_p) du \right| \right|  \stackrel{n \to \infty}{\to} 0
 \end{align*}
 
 This yields: 
 \begin{align*} 
 \limsup_{n \to \infty} \sup_{t \in [s,T]} \left| \left| \e_n \sum_{k=1}^{n_{\e_n}(s,t)} \al'(s, t^{\e_n}_k(s,t),g_p) - \frac{1}{M_\rho} \int_s^t \al'(s, u,g_p) du \right| \right|  \leq C_3 \eta_m \,\forall m
  \end{align*}
 
 and therefore taking the limit as $m \to \infty$: 
 \begin{align*}  
 \lim_{n \to \infty} \sup_{t \in [s,T]} \left| \left| \e_n \sum_{k=1}^{n_{\e_n}(s,t)} \al'(s, t^{\e_n}_k(s,t),g_p) - \frac{1}{M_\rho} \int_s^t \al'(s, u,g_p) du \right| \right|=0 
 \end{align*} 
 which is what we want.\\
 
 Finally we get on $\Omega'_{00}$, by continuity of the limit points:
 \begin{align*}
 & V'_{\e_n}(s,\bullet)g_p-g_p- M_\rho \e_n \sum_{k=1}^{n_{\e_n}(s,\bullet)} V'_{\e_n}(s, t^{\e_n}_k(s,\bullet),g_p,\widehat{A}) \to V'_{0}(s,\bullet)g_p-g_p- \int_s^\bullet V'_0(s,u) \widehat{A}(u)g_pdu
 \end{align*}
 
 Since we have:
\begin{align*} 
 & (V'_{\e_n}(s,\bullet)g_p,V'_{\e_n}(s, \bullet,g_p,\widehat{A})) \stackrel{d}{=} (V_{\e_n}(s,\bullet)g_p,V_{\e_n}(s, \bullet)\widehat{A}(\bullet) g_p)
 \end{align*}
 
 Then we get, for some $M^0_\bullet(s,g_p)$:
 \begin{align*} 
 & V_{\e_n}(s,\bullet)g_p-g_p- M_\rho \e_n \sum_{k=1}^{n_{\e_n}(s,\bullet)} V_{\e_n}(s,t^{\e_n}_k(s,\bullet)) \widehat{A}  \left( t^{\e_n}_k(s,\bullet) \right) g_p \Rightarrow M^0_\bullet(s,g_p)
 \end{align*} 
 
 and therefore by \ref{25} that $\widetilde{M}_\bullet^{\e_n}(s)g_p  \Rightarrow M^0_\bullet(s,g_p)$, which yields by \ref{26} that $M^0_\bullet(s,g_p)=0$ a.e., and therefore: 
 \begin{align*}
 & V'_{0}(s,\bullet)g_p-g_p-\int_s^\bullet V'_0(s,u) \widehat{A}(u)g_pdu = 0 \mbox{ a.e., say on some } \Omega'_*:=\bigcap_p \Omega'_*(p) \subseteq \Omega'_{00}
 \end{align*}
  
 Let $f \in Y_1$ and $t \in J(s)$. Since $\{g_p\}$ is dense in $Y_1$, there exists a sequence $f_p \subseteq \{g_p\} \stackrel{Y_1}{\to} f$. We have on $\Omega'_*$:
 \begin{align*}
 & \left| \left| V'_{0}(s,t)f-f-\int_s^t V'_0(s,u) \widehat{A}(u)f  \right| \right| \\&\leq  ||V'_{0}(s,t)f-V'_{0}(s,t)f_p||+\left| \left| V'_{0}(s,t)f_p-f-\int_s^t V'_0(s,u) \widehat{A}(u)f \right| \right|\\
 & \leq 2||f_p-f|| +\int_s^t ||V'_0(s,u) \widehat{A}(u)f_p-V'_0(s,u) \widehat{A}(u)f||du\\
 & \leq 2 ||f_p-f|| +||f_p-f||_{Y_1} \int_s^t || \widehat{A}(u)||_{\B(Y_1,Y)} du \to 0\\
 \end{align*}
 
 And therefore on $\Omega'_*$ we have $\forall f \in Y_1$:
 \begin{align*}
 & V'_{0}(s,\bullet)f=f+\int_s^\bullet V'_0(s,u) \widehat{A}(u)fdu
 \end{align*}
 
 By continuity of $\widehat{A}$ on $Y_1$ (\textbf{(A0)}), and boundedness + continuity of $V'_0$: $t \to V'_0(s,t) \widehat{A}(t)f \in C(J(s),Y)$ on $\Omega'_*$ and therefore we have on $\Omega'_*$:
  \begin{align*}
 & \frac{\partial}{\partial t}V'_{0}(s,t)f=V'_0(s,t) \widehat{A}(t)f \hspace{3mm} \forall t \in J(s), \,\forall f \in Y_1 \\
 & V'_{0}(s,s)=I
 \end{align*}
 
By \textbf{(A3)} and \ref{10}, $V'_{0}(s,\bullet)(\oo')f=\widehat{\G}(s,\bullet)f$, $\forall \oo' \in \Omega'_*$, $\forall f \in Y_1$. By contraction property and density of $Y_1$ in $Y$, the previous equality is true in $Y$.
 
 \end{proof} 


\section{Applications}
In this section we give an application to inhomogeneous L\'{e}vy random evolutions. We first introduce inhomogeneous L\'{e}vy semigroups.\\

Let $J=[0,T_\infty]$, $Y:=C_0(\RR^d)$, $Y_1:=C_0^2(\RR^d)$. Let $(\Omega, \F, \PP)$ a probability space (possibly different than the probability space on which is defined the Semi-Markov process) and $(L_t)_{t \in J}$ a process with independent increments and absolutely continuous characteristics (PIIAC), or inhomogeneous L\'{e}vy processes in \cite{Kl}. It is a specific case of additive processes that are semimartingales, which is not the case of all additive processes (see \cite{Kl}) . For $A \in Bor(\RR^d)$ and $z \in \RR^d$ we let $p_{s,t}(z,A)=\PP(L_t-L_s \in A-z)$, $\mu_{s,t}$ the law of $L_t-L_s$ and the inhomogeneous L\'{e}vy semigroup: 
\begin{align*}
& \G(s,t)f(z):=\EE[f(L_t-L_s+z)]=\int_{\RR^d}p_{s,t}(z,dy)f(y)= \int_{\RR^d} \mu_{s,t}(dy) f(z+y)
\end{align*}

$\G$ is a regular $\B(Y)-$contraction semigroup and $\forall n \in \NN$ we have $\G(s,t) \in \B(C_0^n(\RR^d))$ and $||\G(s,t)||_{\B(C_0^n(\RR^d))} \leq 1$ (see proof in Appendix \ref{A2}).\\

The L\'{e}vy-Khintchine representation of such a process $(L_t)_{t \in J}$ (see \cite{CT}, 14.1) ensures that there exists unique $(B_t)_{t \in J} \subseteq \RR^d$, $(C_t)_{t \in J}$ a family of $d \times d$ symmetric nonnegative-definite matrices and $(\bar{\nu}_t)_{t \in J}$ a family of measures on $\RR^d$ such that:
\begin{align*}
& \EE[e^{i \left<u,L_t\right>}]=e^{\psi(u,t)} \mbox {, with }:\\
&\psi(u,t):=i \left<u,B_t\right>-\frac{1}{2} \left<u,C_t u\right>+\int_{\RR^d} (e^{i \left<u,y\right>}-1-i\left<u,y\right>1_{|y| \leq 1})\bar{\nu}_t(dy)
\end{align*}

$(B_t,C_t,\bar{\nu}_t)_{t \in J}$ is called the spot characteristics of $L$. They satisfy the following regularity conditions:
\begin{itemize}
\item $\forall t \in J$, $\bar{\nu}_t\{0\}=0$ and $\int_{\RR^d} (|y|^2 \wedge 1)\bar{\nu}_t(dy)< \infty$
\item $(B_0,C_0,\bar{\nu}_0)=(0,0,0)$ and $\forall (s,t) \in \Delta_J$: $C_t-C_s$ is symmetric nonnegative-definite and $\bar{\nu}_s(A) \leq \bar{\nu}_t(A)$ $\forall A \in Bor(\RR^d)$.
\item $\forall t \in J$, $B_s \to B_t$, $\left<u,C_s u\right> \to  \left<u,C_t u\right>$ $\forall u \in \RR^d$ and $\bar{\nu}_s(A) \to \bar{\nu}_t(A)$ $\forall A \in Bor(\RR^d)$ such that $A \subseteq \{z \in \RR^d :|z|>\e\}$ for some $\e>0$.
\end{itemize}

If $\forall t \in J$, $\int_{\RR^d} |y|1_{|y| \leq 1}\bar{\nu}_t(dy)<\infty$, we can replace $B_t$ by $B_t^0$ in the L\'{e}vy-Khintchine representation of $L$, where $B^0_t:=B_t-\int_{\RR^d} y1_{|y| \leq 1}\bar{\nu}_t(dy)$. We denote by $(B_t^0,C_t,\bar{\nu}_t)^0_{t \in J}$ this other version of the spot characteristics of $L$.\\
 
In the case of PIIAC, there exists $(b_t)_{t \in J} \subseteq \RR^d$, $(c_t)_{t \in J}$ a family of $d \times d$ symmetric nonnegative-definite matrices and $(\nu_t)_{t \in J}$ a family of measures on $\RR^d$ satisfying:
\begin{align*}
& \int_0^{T_\infty} \left(|b_s|+||c_s||+\int_{\RR^d} (|y|^2 \wedge 1)\nu_s(dy) \right)ds< \infty \\
& B_t=\int_0^t b_s ds\\
& C_t=\int_0^t c_s ds\\
& \bar{\nu}_t(A)=\int_0^t \nu_s(A) ds \hspace{3mm} \forall A \in Bor(\RR^d)
\end{align*}

where $||c_t||$ denotes any norm on the space of $d \times d$ matrices. $(b_t,c_t,\nu_t)_{t \in J}$ is called the local characteristics of $L$.\\

By \cite{Kl}, we have the following representation for $L$:
\begin{align*}
L_t=\int_0^tb_sds+\int_0^t \sqrt{c_s}dW_s+\int_0^t \int_{\RR^d} y1_{|y| \leq 1} (N-\bar{\nu}) (ds dy) + \sum_{s \leq t} \Delta L_s1_{|\Delta L_s| > 1}
\end{align*}

where for $A \in Bor(\RR^d)$: $\bar{\nu}([0,t] \times A):=\bar{\nu}_t (A)$ and $N$ the Poisson measure of $L$ ($N-\bar{\nu}$ is then called the compensated Poisson measure of $L$). $(W_t)_{t \in J}$ is a $d$-dimensional Brownian motion on $\RR^d$, independent from the jump process $\int_0^t \int_{\RR^d} y1_{|y| \leq 1} (N-\nu) (ds dy) + \sum_{s \leq t} \Delta L_s1_{|\Delta L_s| > 1}$. $\sqrt{c_t}$ here stands for the unique symmetric nonnegative-definite square root of $c_t$. Sometimes it is convenient to write a Cholesky decomposition $c_t=h_t h_t^T$ and replace $\sqrt{c_t}$ by $h_t$ in the previous representation.\\

It can be shown - see \cite{RW} - that the infinitesimal generator of the semigroup $\G$ is given by:
\begin{align*}
\begin{split}
& A_\G(t)f(z)=\sum_{j=1}^d b_t(j) \frac{\partial f}{\partial x_j}(z)+\frac{1}{2}\sum_{j,k=1}^d c_t(j,k) \frac{\partial^2 f}{\partial x_j \partial x_k}(z)\\
&+\int_{\RR^d} \left(f(z+y)-f(z)-\sum_{j=1}^d \frac{\partial f}{\partial x_j}(z) y(j)1_{|y| \leq 1}\right)\nu_t(dy)
\end{split}
\end{align*}

and that $Y_1=C^2_0(\RR^d) \subseteq \D(A_\G(t))=\D(A_\G)$. And if $b^0_t:=b_t-\int_{\RR^d} y1_{|y| \leq 1}\nu_t(dy)$ is well-defined:
\begin{align*}
& A_\G(t)f(z)=\sum_{j=1}^d b^0_t(j) \frac{\partial f}{\partial x_j}(z)+\frac{1}{2}\sum_{j,k=1}^d c_t(j,k) \frac{\partial^2 f}{\partial x_j \partial x_k}(z)+\int_{\RR^d} (f(z+y)-f(z))\nu_t(dy)
\end{align*}

Now to introduce inhomogeneous L\'{e}vy Random Evolutions, we consider $(L^x)_{x \in X}$ a collection of inhomogeneous L\'{e}vy processes on $(\Omega, \F, \PP)$  with local characteristics $(b_t(x),c_t(x),\nu_t(x))$. Define for $z \in \RR^d$, $(s,t) \in \Delta_J$, $x \in X$:
\begin{align*}
\G_x(s,t)f(z):=\EE[f(L^x_t-L^x_s+z)]=\int_{\RR^d}p^x_{s,t}(z,dy)f(y)
\end{align*} 
 
This inhomogeneous Random evolution is regular, $\B(Y)-$contraction because the corresponding semigroup is. In this case and under some technical conditions, we can actually prove the compact containment criterion in the case where $d=1$ (this result can probably be extended to any $d$). Indeed, define the jump operators:
 \begin{align*}
D^\e(x,y)f(z):=f(z+\e \al(x,y))
\end{align*}

where $\al \in B^b_\RR(X \times X, \X \otimes \X)$, so that $D_1^0(x,y)f=\al(x,y) f'$ and $ Y_1\subseteq \D(D_1)$. Let $\widetilde{L}$ be an inhomogeneous L\'{e}vy process with local characteristics $(0,0,\nu_t)$ and $\widetilde{\mu}_{s,t}$ the law of $\widetilde{L}_t-\widetilde{L}_s$. Assume that:
 \begin{enumerate}[i)]
 \item $\nu_t(x)=\nu_t$ $\forall x \in X$ (the L\'{e}vy measure is the same over all states)
 \item $\sup_{x,t} |b_t(x)| \leq r \in \RR^+$ and $\sup_{x,t} \sqrt{c_t(x)} \leq \sigma \in \RR^+$ (uniformly bounded drift and volatility)
 \item $\forall (s,T) \in \Delta_J$, the collection of measures $\{\widetilde{\mu}_{s,t}\}_{t \in [s,T]}$ is tight.
 \end{enumerate}

Then using \ref{18}, $\{V_\e\}$ satisfies \textbf{(A1)}, i.e. the compact containment criterion in $Y_1$ (see proof in Appendix \ref{A2}).\\

Let's make some comments about the other assumptions of \ref{27}. \textbf{(A0)} will be satisfied with $\widehat{\D}:=C_0^3(\RR^d)$ and provided the local characteristics $(b_t(x),c_t(x),\nu_t(x))$ are bounded and differentiable. By the discussion we had just before \ref{27}, \textbf{(A2)} will be satisfied because from the expression of $A_x(t)$ below and boundedness of the local characteristics, the following family of functions converge uniformly to $0$ at infinity, are equicontinuous and uniformly bounded:
\begin{align*} 
& \{\widehat{A}(t)f: t \in[s,T] \} 
\end{align*} 

About \textbf{(A3)}, the generator $A_x(t)$ has the following expression:

\begin{align*}
A_x(t)f(z)&=\sum_{j=1}^d b^x_t(j) \frac{\partial f}{\partial x_j}(z)+\frac{1}{2}\sum_{j,k=1}^d c^x_t(j,k) \frac{\partial^2 f}{\partial x_j \partial x_k}(z)\\&+\int_{\RR^d} \left(f(z+y)-f(z)-\sum_{j=1}^d \frac{\partial f}{\partial x_j}(z) y(j)1_{|y| \leq 1} \right)\nu^x_t(dy)
\end{align*}

Keep in mind that:
\begin{align*}
& \widehat{A}(t)=\frac{1}{M_\rho} \sum_{x \in X} \left(m_1(x) A_x(t)+PD_1^0(x,\bullet)(x) \right) \rho_x
\end{align*}

It is clear that $\widehat{A}$ is the generator of an inhomogeneous L\'{e}vy semigroup with local characteristcs $(\widehat{b}_t,\widehat{c}_t,\widehat{\nu}_t)$ given by:
\begin{align*}
& \widehat{b}_t=\frac{1}{M_\rho} \sum_{x \in X} \left(m_1(x) b_t^x+P\al(x,\bullet)(x) \right) \rho_x\\
& \widehat{c}_t=\frac{1}{M_\rho} \sum_{x \in X} m_1(x)  \rho_x c_t^x\\
& \widehat{\nu}_t(dy)=\frac{1}{M_\rho} \sum_{x \in X} m_1(x) \rho_x \nu_t^x(dy)
\end{align*}

In particular we check that $\widehat{\nu}_t$ is still a L\'{e}vy measure on $\RR^d$. \\


\appendix

\section{Proofs of the results in section 2 and 3}\label{A1}

\subsection{Proof of \ref{111}}
Let $(s,t) \in \Delta_J$, $f \in Y_1$.
\begin{align*}
&\frac{\partial^-}{\partial s} \G(s,t)f=\lim_{\substack{h \downarrow 0\\ (s-h,t) \in \Delta_J}}  \frac{\G(s,t)f- \G(s-h,t)f}{h}\\
&=-\lim_{\substack{h \downarrow 0\\ (s-h,t) \in \Delta_J}} \frac{ \G(s-h,s)- I}{h}\G(s,t)f=-A_\G(s) \G(s,t)f 
\end{align*}
since $ \G(s,t)f \in \D(A_\G)$.\\

For $s<t$:
\begin{align*}
&\frac{\partial^+}{\partial s} \G(s,t)f=\lim_{\substack{h \downarrow 0\\ (s+h,t) \in \Delta_J}}  \frac{\G(s+h,t)f- \G(s,t)f}{h}\\
&=-\lim_{\substack{h \downarrow 0\\ (s+h,t) \in \Delta_J}} \frac{ \G(s,s+h)- I}{h}\G(s+h,t)f\\
\end{align*}
Let $h \in (0,t-s]$:
\begin{align*}
\begin{split}
&\left| \left| \frac{(\G(s,s+h)- I)}{h}\G(s+h,t)f -A_\G(s)\G(s,t)f \right| \right| \\& \leq \left| \left| \frac{(\G(s,s+h)- I)}{h}\G(s,t)f-A_\G(s)\G(s,t)f \right| \right| \\&+ \left| \left| \frac{(\G(s,s+h)- I)}{h} \right| \right|_{\B(Y_1,Y)} ||\G(s+h,t)f-\G(s,t)f||_{Y_1}
\end{split}
\end{align*}
, the last inequality holding because $\forall (s,t) \in \Delta_J$: $\G(s,t)Y_1 \subseteq Y_1$.\\

We are going to apply the uniform boundedness principle to show that :
\begin{align*}
 \sup \limits_{h \in (0,t-s]}\left| \left| \frac{(\G(s,s+h)- I)}{h} \right| \right|_{\B(Y_1,Y)} < \infty
\end{align*}
$Y_1$ is Banach. We have to show that  $\forall g \in Y_1$: $\sup \limits_{h \in (0,t-s]} \left| \left| \frac{(\G(s,s+h)- I)}{h}g \right| \right|< \infty$. Let $g \in Y_1$. We have $\left| \left| \frac{(\G(s,s+h)- I)}{h}g \right| \right| \stackrel{h \downarrow 0}{\to} ||A_\G(s)g||$ since $Y_1 \subseteq  \D(A_\G)$. $\exists \delta(g) \in (0,t-s):$ $h \in (0,\delta) \Rightarrow \left| \left| \frac{(\G(s,s+h)- I)}{h}g \right| \right| <1 + ||A_\G(s)g||$. Then, by $Y_1$-strong $t-$continuity of $\G$, $h \to \left| \left| \frac{(\G(s,s+h)- I)}{h}g \right| \right| \in C([\delta,t-s])$. Let $M:=\max_{h \in [\delta,t-s]}\left| \left| \frac{(\G(s,s+h)- I)}{h}g \right| \right|$. Then we get $\left| \left| \frac{(\G(s,s+h)- I)}{h}g \right| \right| \leq \max(M,1 + ||A_\G(s)g||)$ $\forall h \in (0,t-s]$ and so $ \sup \limits_{h \in (0,t-s]} \left| \left| \frac{(\G(s,s+h)- I)}{h}g \right| \right|< \infty$.\\

Further, by $Y_1-$super strong $s-$continuity of $\G$, $||\G(s+h,t)f-\G(s,t)f ||_{Y_1} \stackrel{h \downarrow 0}{\to} 0$. Finally, since $\G(s,t)f \in \D(A_\G)$, $\left| \left| \frac{(\G(s,s+h)- I)}{h}\G(s,t)f -A_\G(s)\G(s,t)f  \right| \right|  \stackrel{h \downarrow 0}{\to} 0$.\\

Therefore we get $\frac{\partial^+}{\partial s} \G(s,t)f= -A_\G(s)\G(s,t)f$ for $s<t$, which shows that $\frac{\partial}{\partial s} \G(s,t)f= -A_\G(s)\G(s,t)f$ for $(s,t) \in \Delta_J$.\\

\subsection{Proof of \ref{2}}
Let $(s,t) \in \Delta_J$, $f \in Y_1$. We have:
\begin{align*}
&\frac{\partial^+}{\partial t} \G(s,t)f=\lim_{\substack{h \downarrow 0\\ (s,t+h) \in \Delta_J}} \frac{\G(s,t+h)f- \G(s,t)f}{h}\\
&=\lim_{\substack{h \downarrow 0\\ (s,t+h) \in \Delta_J}} \G(s,t)\frac{(\G(t,t+h)- I)f}{h}\\
\end{align*}
And for $h \in J$: $t+h \in J$:
\begin{align*}
&\left| \left| \G(s,t)\frac{(\G(t,t+h)- I)f}{h} -\G(s,t) A_\G(t)f \right| \right| \\&\leq ||\G(s,t)||_{\B(Y)} \left| \left| \frac{(\G(t,t+h)- I)f}{h}-A_\G(t)f \right| \right|  \stackrel{h \downarrow 0}{\to} 0
\end{align*}
since $f \in \D(A_\G)$. Therefore $\frac{\partial^+}{\partial t} \G(s,t)f=\G(s,t) A_\G(t)f$.\\

Now if $s<t$:
\begin{align*}
&\frac{\partial^-}{\partial t} \G(s,t)f=\lim_{\substack{h \downarrow 0 \\ (s,t-h) \in \Delta_J}} \frac{\G(s,t)f- \G(s,t-h)f}{h}\\
&=\lim_{\substack{h \downarrow 0 \\ (s,t-h) \in \Delta_J}} \G(s,t-h) \frac{(\G(t-h,t)- I)f}{h}\\
\end{align*}

For $h \in (0,t-s]$: 
\begin{align*}
\begin{split}
&\left| \left| \G(s,t-h) \frac{(\G(t-h,t)- I)f}{h} -\G(s,t) A_\G(t)f \right| \right|\\ &\leq  ||\G(s,t-h)||_{\B(Y)} \left| \left| \frac{(\G(t-h,t)- I)f}{h}-A_\G(t)f \right| \right| \\  & +  ||(\G(s,t-h)-\G(s,t))A_\G(t)f||
\end{split}
\end{align*}

Since $f \in \D(A_\G)$, $\left| \left| \frac{(\G(t-h,t)- I)f}{h}-A_\G(t)f \right| \right|  \stackrel{h \downarrow 0}{\to} 0$. By $Y-$strong $t$-continuity of $\G$: $||(\G(s,t-h)-\G(s,t))A_\G(t)f|| \stackrel{h \downarrow 0}{\to} 0$. By the principle of uniform boundedness together with the $Y-$strong $t$-continuity of $\G$, we have $\sup_{h \in (0,t-s]} ||\G(s,t-h)||_{\B(Y)} \leq \sup_{h \in [0,t-s]} ||\G(s,t-h)||_{\B(Y)} <\infty$.\\

Therefore we get $\frac{\partial^-}{\partial t} \G(s,t)f=\G(s,t) A_\G(t)f$ for $s<t$, which shows $\frac{\partial}{\partial t} \G(s,t)f=\G(s,t) A_\G(t)f$ for $(s,t) \in \Delta_J$.\\

\subsection{Proof of \ref{7}}
Let $f \in Y_1$, $(s,t) \in \Delta_J$. First $u \to \G(s,u) A_\G(u) f \in B_Y([s,t])$ as the derivative of $u \to \G(s,u)f$. By the principle of uniform boundedness together with the $Y-$strong $t-$continuity of $\G$, we have $M:=\sup_{u \in [s,t]} ||\G(s,u)||_{\B(Y)} <\infty$. We then observe that for $u \in [s,t]$:
\begin{align*}
& ||\G(s,u)A_\G(u)f|| \leq M ||A_\G(u)f|| \leq  M ||A_\G(u)||_{\B(Y_1,Y)}  ||f||_{Y_1} 
\end{align*}

\subsection{Proof of \ref{10}}
Let $(s,u), (u,t) \in \Delta_J$, $f \in Y_1$. Consider the function $\phi: u \to G(u)\G(u,t)f$. We are going to show that $\phi'(u)=0$ $\forall u \in [s,t]$ and therefore that $\phi(s)=\phi(t)$.
We have for $u<t$:
\begin{align*}
& \frac{d^+ \phi}{d u}(u)=\lim_{\substack{h \downarrow 0 \\ h \in (0,t-u]}} \frac{1}{h} [G(u+h) \G(u+h,t) f-G(u) \G(u,t) f]
\end{align*}

Let $h \in (0,t-u]$. We have:
\begin{align*}
\begin{split}
& \left| \left| \frac{1}{h} [G(u+h) \G(u+h,t) f-G(u) \G(u,t) f] \right| \right| \\ &\leq  \underbrace{\left| \left| \frac{1}{h}G(u+h) \G(u,t)f-\frac{1}{h}G(u) \G(u,t)f-G(u)A_\G(u) \G(u,t)f \right| \right|}_{(1)} \\
& + ||G(u+h)||_{\B(Y)} \underbrace{\left| \left| \frac{1}{h} \G(u+h,t) f - \frac{1}{h} \G(u,t) f+  A_\G(u) \G(u,t)f\right| \right|}_{(2)}\\
& + \underbrace{|| G(u+h) A_\G(u) \G(u,t)f-G(u) A_\G(u) \G(u,t)f||}_{(3)}
\end{split}
\end{align*}

And we have:
\begin{align*}
& (1) \to 0 \mbox{ as $G$ satisfies the initial value problem and } \G(u,t) Y_1 \subseteq Y_1 \\
& (2) \to 0 \mbox{ as } \frac{\partial}{\partial u} \G(u,t)f=-A_\G(u) \G(u,t)f\\
& (3) \to 0 \mbox{ by Y-strong continuity of } G
\end{align*}

Further, by the principle of uniform boundedness together with the $Y-$strong continuity of $G$, we have $\sup_{h \in (0,t-u]} ||G(u+h)||_{\B(Y)} \leq \sup_{h \in [0,t-u]} ||G(u+h)||_{\B(Y)} < \infty$. \\

We therefore get $\frac{d^+ \phi}{d u}(u)=0$. Now for $u>s$:
\begin{align*}
& \frac{d^- \phi}{d u}(u)=\lim_{\substack{h \downarrow 0 \\ h \in (0,u-s]}} \frac{1}{h} [G(u) \G(u,t) f-G(u-h) \G(u-h,t) f]
\end{align*}

Let $h \in (0,u-s]$:
\begin{align*}
\begin{split}
&\left| \left| \frac{1}{h} [G(u) \G(u,t) f-G(u-h) \G(u-h,t) f] \right| \right| \\&\leq   \underbrace{\left| \left| \frac{1}{h} G(u) \G(u,t)f- \frac{1}{h} G(u-h) \G(u,t)f -G(u) A_\G(u) \G(u,t)f\right| \right|}_{(4)}\\
&+ ||G(u-h)||_{\B(Y)} \underbrace{\left| \left| -\frac{1}{h} \G(u-h,u)\G(u,t) f + \frac{1}{h} \G(u,t) f+  A_\G(u) \G(u,t)f\right| \right|}_{(5)}\\
&+ \underbrace{||G(u) A_\G(u) \G(u,t)f-G(u-h) A_\G(u) \G(u,t)f||}_{(6)}
\end{split}
\end{align*}

By the principle of uniform boundedness together with the $Y-$strong $t$-continuity of $G$, we have $\sup_{h \in (0,u-s]} ||G(u-h)||_{\B(Y)} \leq \sup_{h \in [0,u-s]} ||G(u-h)||_{\B(Y)} < \infty$. And:
\begin{align*}
& (4) \to 0 \mbox{ as $G$ satisfies the initial value problem and } \G(u,t) Y_1 \subseteq Y_1 \\
& (5) \to 0 \mbox{ as } \G(u,t) Y_1 \subseteq Y_1 \\
& (6) \to 0 \mbox{ by Y-strong continuity of } G
\end{align*}

We therefore get $\frac{d^- \phi}{d u}(u)=0$.

\subsection{Proof of \ref{3}}
Since $\G$ is regular and $f$, $A_\G(u)f \in Y_1$ and $u \to ||A_\G(u)||_{\B(Y_1,Y)}$ is integrable on $[s,t]$ we have by \ref{7}:
\begin{align*}
 \G(s,t)f &=f+\int_s^t\G(s,u)A_\G(u)f du=f+\int_s^t \left[A_\G(u)f + \int_s^u \G(s,r) A_\G(r) A_\G(u)f dr \right]du\\
&=f+\int_s^tA_\G(u)f du+\int_s^t \int_s^u \G(s,r) A_\G(r) A_\G(u)f dr du\\
\end{align*}

If $g_1: [s,t] \to \RR$ and $g_2:[s,t] \to Y$ are differentiable functions, and $g_1g_2'$, $g_1'g_2$ are Bochner integrable on $[s,t]$, then we have by the Fundamental theorem of Calculus for the Bochner integral:
\begin{align*}
& g_1(t)g_2(t)-g_1(s)g_2(s)=\int_s^t (g_1g_2)'(u)du=\int_s^t g_1'(u)g_2(u)du+\int_s^t g_1(u)g_2'(u)du
\end{align*}

With $g_1(u):=(t-u)$ and $g_2(u):=\int_s^u \G(s,r) A_\G(r) A_\G(u)f dr$. We have shown that $g_1'g_2$ is integrable on $[s,t]$. Provided we show that:
\begin{align*}
&g'_2(u)=\frac{\partial}{\partial u} \int_s^u \G(s,r) A_\G(r) A_\G(u)f dr= \G(s,u) A_\G^2(u)f+\int_s^u  \G(s,r) A_\G(r) A'_\G(u)fdr
\end{align*}

the fact that $g_1g_2'$ is integrable on $[s,t]$ by assumption ends the proof. \\

We have:
\begin{align*}
& \frac{\partial}{\partial u} \int_s^u \G(s,r) A_\G(r) A_\G(u)f dr\\&= \lim_{\substack{h \to 0\\(s,u+h) \in \Delta_J}} \frac{1}{h} \left[ \int_s^{u+h} \G(s,r) A_\G(r) A_\G(u+h)f dr-\int_s^u \G(s,r) A_\G(r) A_\G(u)f dr \right]\\
&= \lim_{\substack{h \to 0\\(s,u+h) \in \Delta_J}}\frac{1}{h} \left[ \int_s^{u+h} \G(s,r) A_\G(r) A_\G(u+h)fdr - \int_s^u \G(s,r) A_\G(r) A_\G(u+h)f dr\right]\\
&+\frac{1}{h} \left[ \int_s^{u} \G(s,r) A_\G(r) A_\G(u+h)f - \G(s,r) A_\G(r) A_\G(u)f dr\right]
\end{align*}

By the principle of uniform boundedness together with the $Y-$strong $t$-continuity of $\G$, $Y_1-$strong continuity of $A_\G$ we have $M_1:=\sup_{r \in [s,t]} ||\G(s,r)||_{\B(Y)} < \infty$ and $M_2:=\sup_{r \in [s,t]} ||A_\G(r)||_{\B(Y_1,Y)} < \infty$. Let $h \in (0,t-u]$ (the proof is the same for $h \in [s-u,0)$):
\begin{align*}
& \left| \left| \int_s^{u} \frac{1}{h}\G(s,r) A_\G(r) A_\G(u+h)f - \frac{1}{h}\G(s,r) A_\G(r) A_\G(u)f -\G(s,r) A_\G(r) A'_\G(u)f dr \right| \right| \\
& \leq \int_s^{u} ||\G(s,r)||_{\B(Y)} ||A_\G(r)||_{\B(Y_1,Y)} \left| \left|\frac{1}{h}A_\G(u+h)f -\frac{1}{h} A_\G(u)f-A'_\G(u)f \right| \right|_{Y_1} dr\\
& < \e (u-s) M_1 M_2 \hspace{4mm} \mbox{ for } |h| < \delta_u \mbox{ since $f \in \D(A'_\G)$}
\end{align*}

We have also:
\small
\begin{align*}
& \left| \left|\frac{1}{h} \int_s^{u+h} \G(s,r) A_\G(r) A_\G(u+h)f - \frac{1}{h}\int_s^{u}\G(s,r) A_\G(r) A_\G(u+h)f dr - \G(s,u) A_\G^2(u)f\right| \right|\\
& = \left| \left|\frac{1}{h} \int_u^{u+h} \G(s,r) A_\G(r) A_\G(u+h)f dr -\G(s,u) A_\G^2(u)f\right| \right|\\
& \leq \underbrace{\frac{1}{h} \int_u^{u+h} ||\G(s,r)||_{\B(Y)} ||A_\G(r)||_{\B(Y_1,Y)} ||A_\G(u+h)f -A_\G(u)f ||_{Y_1}dr}_{(1)}\\&+\underbrace{\frac{1}{h} \int_u^{u+h} ||\G(s,r)||_{\B(Y)} ||A_\G(r)A_\G(u)f -A_\G(u)A_\G(u)f ||dr}_{(2)}\\&+ \underbrace{\frac{1}{h} \int_u^{u+h} ||\G(s,r)A^2_\G(u)f -\G(s,u)A^2_\G(u)f ||dr}_{(3)}
\end{align*}

Therefore:
\begin{align*}
& (1) \to 0 \mbox{ since } f \in  \D(A'_\G) \mbox{ and } M_1,M_2 < \infty\\
& (2) \to 0 \mbox{ by } Y_1-\mbox{strong continuity of } A_\G \mbox{ and } M_1< \infty  \\
& (3) \to 0 \mbox{ by } Y-\mbox{strong t-continuity of } \G
\end{align*}

\subsection{Proof of \ref{12}}
The fact that $V(s,t) \in \B(Y)$ is straightforward from the definition of $V$. The semigroup property comes from straightforward computations. Then, we will show that $u \to V(s,u)(\omega)$ is $Y-$strongly continuous on each $[T_{n}(s), T_{n+1}(s)) \cap J(s)$, $n \in \NN$ and $Y-$strongly RCLL at each $T_{n+1}(s) \in J(s)$, $n \in \NN$. Let $n \in \NN$ such that $T_n(s) \in J(s)$. $\forall t \in [T_n(s), T_{n+1}(s)) \cap J(s)$, we have:
\begin{align*}
&V(s,t)=\left[ \prod \limits_{k=1}^{n} \Gamma_{x_{k-1}(s)} \left(T_{k-1}(s), T_{k}(s) \right) D(x_{k-1}(s),x_k(s)) \right] \Gamma_{x_n(s)} \left(T_n(s),t \right)
\end{align*}

Therefore by $Y-$strong $t-$continuity of $\G$, we get that $u \to V(s,u)(\omega)$ is $Y-$strongly continuous on $[T_n(s), T_{n+1}(s)) \cap J(s)$. If $T_{n+1}(s) \in J(s)$, the fact that $V(s,\bullet)$ has a left limit at $T_{n+1}(s)$ also comes from the $Y-$strong $t-$continuity of $\G$: 
\begin{align*}
&V(s, T_{n+1}(s)^-)f= \lim_{h \downarrow 0} G^s_n \Gamma_{x_n(s)} (T_n(s),T_{n+1}(s)-h)f=G^s_n\Gamma_{x_n(s)}(T_n(s),T_{n+1}(s))f\\
&  G^s_n=\prod \limits_{k=1}^{n} \Gamma_{x_{k-1}(s)} \left(T_{k-1}(s), T_{k}(s) \right) D(x_{k-1}(s),x_k(s))
\end{align*}

Therefore we get the relationship:
\begin{align*}
&V(s, T_{n+1}(s))f= V(s, T_{n+1}(s)^-) D(x_{n}(s),x_{n+1}(s))f
\end{align*}

We notice therefore why we used the terminology "continuous inhomogeneous $Y-$random evolution" when $D=I$.

\subsection{Proof of \ref{14}}
Let $s \in J$, $\omega \in \Omega$, $f \in Y_1$. We are going to proceed by induction and show that $\forall n \in \NN$, we have $\forall t \in [T_n(s),T_{n+1}(s)) \cap J(s)$:
\begin{align*}
& V(s,t)f=f+\int_s^t V(s,u)A_{x(u)}(u)f du+\sum \limits_{k=1}^{n} V(s,T_k(s)^-) [D(x_{k-1}(s),x_k(s))-I]f
\end{align*}

For $n=0$, we have $\forall t \in [s,T_1(s)) \cap J(s)$: $V(s,t)f=\G_{x(s)}(s,t) f$, and therefore $V(s,t)f=f+\int_s^t V(s,u)A_{x(u)}(u)f du$ by regularity of $\G$.\\

Now assume that the property is true for $n-1$, namely: $\forall t \in [T_{n-1}(s),T_{n}(s)) \cap J(s)$, we have:
\begin{align*}
& V(s,t)f=f+\int_s^t V(s,u)A_{x(u)}(u)f du+\sum \limits_{k=1}^{n-1} V(s,T_k(s)^-) [D(x_{k-1}(s),x_k(s))-I]f
\end{align*}

Therefore it implies that (by continuity of the Bochner integral):
\begin{align*}
& V(s,T_n(s)^-)f=f+\int_s^{T_n(s)} V(s,u)A_{x(u)}(u)f du+\sum \limits_{k=1}^{n-1} V(s,T_k(s)^-) [D(x_{k-1}(s),x_k(s))-I]f
\end{align*}

Now, $\forall t \in [T_{n}(s),T_{n+1}(s)) \cap J(s)$ we have that:
\begin{align*}
& V(s,t)=G^s_n \G_{x_n(s)}(T_n(s),t)\\
& G^s_n:=\prod \limits_{k=1}^{n} \Gamma_{x_{k-1}(s)} \left(T_{k-1}(s), T_{k}(s) \right) D(x_{k-1}(s),x_k(s))
\end{align*}

and therefore $\forall t \in [T_{n}(s),T_{n+1}(s)) \cap J(s)$, by \ref{2} and regularity of $\G$:
\begin{align*}
& \frac{\partial}{\partial t}V(s,t)f=V(s,t)A_{x(t)}(t)f \Rightarrow V(s,t)f=V(s,T_n(s))f+\int_{T_n(s)}^t V(s,u)A_{x(u)}(u)f du
\end{align*}

Further, by the proof of \ref{12} we get $V(s, T_{n}(s))f= V(s, T_{n}(s)^-) D(x_{n-1}(s),x_{n}(s))f$. Therefore combining these results we have:
\begin{align*}
&V(s,t)f=V(s, T_{n}(s)^-) D(x_{n-1}(s),x_{n}(s))f+\int_{T_n(s)}^t V(s,u)A_{x(u)}(u)f du\\
&=V(s,T_n(s)^-)f+\int_{T_n(s)}^t V(s,u)A_{x(u)}(u)f du+V(s, T_{n}(s)^-) D(x_{n-1}(s),x_{n}(s))f-V(s,T_n(s)^-)f\\
&= f+\int_s^{T_n(s)} V(s,u)A_{x(u)}(u)f du+\sum \limits_{k=1}^{n-1} V(s,T_k(s)^-) [D(x_{k-1}(s),x_k(s))-I]f\\
&+ \int_{T_n(s)}^t V(s,u)A_{x(u)}(u)f du+V(s, T_{n}(s)^-) D(x_{n-1}(s),x_{n}(s))f-V(s,T_n(s)^-)f\\
&=f+\int_s^{t} V(s,u)A_{x(u)}(u)f du+ \sum \limits_{k=1}^{n} V(s,T_k(s)^-) [D(x_{k-1}(s),x_k(s))-I]f
\end{align*}

\section{Proofs of the results in section 5}\label{A2}

\subsection{Proof that $\G$ is a semigroup and that $\forall n \in \NN$ we have $\G(s,t) \in \B(C_0^n(\RR^d))$ and $||\G(s,t)||_{\B(C_0^n(\RR^d))} \leq 1$}

For $n,d \in \NN^*$, let $\NN_{d,n}:=\{\alpha \in \NN^d: \sum_{i=1}^d \alpha(i) \leq n\}$. The space $C^n_0(\RR^d)$ (resp. $C^n_b(\RR^d)$) denotes the space of functions $f: \RR^d \to \RR$ for which $\partial^\alpha f \in C_0(\RR^d)$ (resp. $C_b(\RR^d)$)  $\forall \alpha \in \NN_{d,n}$. This space is Banach for the norm: 
\begin{align*}
&||f||:= \max_{\alpha \in \NN_{d,n}} ||\partial^\alpha f||_{\infty}\\
&\partial^\alpha f:= \frac{\partial^{\sum_{i=1}^d \alpha(i)}f}{\partial_1^{\alpha(1)}...\partial_d^{\alpha(d)}}
\end{align*}

Let for $(s,t) \in \Delta_J$: $\G(s,t)f(x):=\EE[f(L_t-L_s+x)]=\int_{\RR^d}p_{s,t}(x,dy)f(y)$ for $f \in B^b_\RR(\RR^d)$. By linearity of the expectation, $\G(s,t)$ is linear. Further, $\G$ satisfies the semigroup equation because of the Chapman-Kolmogorov equation. Now let's show that $\forall n \in \NN$ we have $\G(s,t) \in \B(C_0^n(\RR^d))$ and $||\G(s,t)||_{\B(C_0^n(\RR^d))} \leq 1$. Let $Y:=C_0(\RR^d)$.\\

Let's first start with $C_0(\RR^d)$. Let $f \in C_0(\RR^d)$. By (\cite{Sa}), we get the representation $\G(s,t)f(x)=\int_{\RR^d} \mu_{s,t}(dy) f(x+y)$, where $ \mu_{s,t}$ is the distribution of $L_t-L_s$, i.e. $\mu_{s,t}(A)=\PP(L_t-L_s \in A)$ for $A \in Bor(\RR^d)$. \\
Let $x \in \RR^d$ and take any sequence $(x_n)_{n \in \NN} \subseteq \RR^d$: $x_n \to x$ and denote $g_n:=y \to f(x_n+y) \in C_0(\RR^d)$ and $g:=y \to f(x+y) \in C_0(\RR^d)$. By continuity of $f$, $g_n \to g$ pointwise. Further, $||g_n|| = ||f|| \in L^1_\RR(\RR^d, Bor(\RR^d),\mu_{s,t})$. Therefore by Lebesgue dominated convergence theorem, we get $\lim_{n \to \infty} \G(s,t)f(x_n)=\int_{\RR^d}\lim_{n \to \infty} \mu_{s,t}(dy) f(x_n+y)=\int_{\RR^d} \mu_{s,t}(dy) f(x+y)=\G(s,t)f(x)$. Therefore $\G(s,t)f \in C(\RR^d)$. By the same argument but now taking any sequence $(x_n)_{n \in \NN} \subseteq \RR^d$: $|x_n| \to \infty$, we get $\lim_{|x| \to \infty} \G(s,t)f(x)=0$ and therefore $\G(s,t)f \in C_0(\RR^d)$.\\
Further, we get:
\begin{align*}
& |\G(s,t)f(x)|=\left| \int_{\RR^d} \mu_{s,t}(dy) f(x+y) \right| \leq  \int_{\RR^d} \mu_{s,t}(dy) |f(x+y)| \leq \int_{\RR^d} \mu_{s,t}(dy) ||f||=||f||
\end{align*}
and therefore $||\G(s,t)||_{\B(C_0(\RR^d))} \leq 1$. \\

Let $x \in \RR^d$, $f \in C_0^q(\RR^d)$ ($q \geq 1$). Take any sequence $(h_n)_{n \in \NN} \subseteq \RR^*$: $h_n \to 0$, $j \in [|1,d|]$ and $f_{n,j}(x):=f(x_1,...,x_j+h_n,...,x_d)$. Then:
\begin{align*}
& \frac{1}{h_n}((\G(s,t)f)_{n,j}(x)-\G(s,t)f(x))=\int_{\RR^d} \mu_{s,t}(dy) \frac{1}{h_n}(f_{n,j}(x+y)-f(x+y))
\end{align*}
Let $g_n:=y \to \frac{1}{h_n}(f_{n,j}(x+y)-f(x+y)) \in C_0(\RR^d)$ and $g:=y \to \frac{\partial f}{\partial x_j}(x+y) \in C_0(\RR^d)$. We have $g_n \to g$ pointwise since $f \in C_0^1(\RR^d)$. By the Mean Value theorem, $\exists z_j(n,x,y) \in [-|h_n|, |h_n|]:$ $g_n(y)=\frac{\partial f}{\partial x_j}(x_1+y_1,...,x_j+y_j+z_j(n,x,y),...,x_d+y_d)$, and therefore $|g_n(y)| \leq \left| \left|\frac{\partial f}{\partial x_j} \right| \right| \in L^1_\RR(\RR^d, Bor(\RR^d),\mu_{s,t})$. Therefore by Lebesgue dominated convergence theorem, we get:
\begin{align*}
& \frac{\partial \G(s,t)f}{\partial x_j}(x)=\int_{\RR^d} \mu_{s,t}(dy) \frac{\partial f}{\partial x_j}(x+y)
\end{align*}

Using the same argument as for $C_0(\RR^d)$, we get that $\G(s,t)f \in C_0^1(\RR^d)$ since $\frac{\partial f}{\partial x_j} \in C_0^1(\RR^d)$ $\forall j \in [|1,d|]$. Repeating this argument by computing successively every partial derivative up to order $q$ by the relationship $\partial^\alpha\G(s,t)f(x)=\int_{\RR^d} \mu_{s,t}(dy)\partial^\alpha f(x+y)$ $\forall \alpha \in \NN_{d,q}$, we get $\G(s,t)f \in C_0^q(\RR^d)$.\\

Further, the same way we got $||\G(s,t)f|| \leq ||f||$ for $f \in C_0(\RR^d)$, we get for $f \in C_0^q(\RR^d)$: $||\partial^\alpha\G(s,t)f|| \leq ||\partial^\alpha f||$ $\forall \alpha \in \NN_{d,q}$. Therefore $\max_{\alpha \in \NN_{d,q}} ||\partial^\alpha \G(s,t)f|| \leq \max_{\alpha \in \NN_{d,q}} ||\partial^\alpha f||$ $\Rightarrow ||\G(s,t)f||_{C_0^q(\RR^d)} \leq ||f||_{C_0^q(\RR^d)} \Rightarrow ||\G(s,t)||_{\B(C_0^q(\RR^d))} \leq 1$.\\

\subsection{Proof that $\G$ is $Y_1-$super strongly $s$-continuous, $Y-$strongly $t$-continuous}

Let $(s,t) \in \Delta_J$, $f \in Y_1$ and $h \in [-s,t-s]$:
\begin{align*}
 ||\G(s+h,t)f-\G(s,t)f||_{Y_1}&=\max_{\alpha \in \NN_{d,m}}  ||\partial^\alpha \G(s+h,t)f - \partial^\alpha \G(t,s)f||\\
&=\max_{\alpha \in \NN_{d,m}}  || \G(s+h,t) \partial^\alpha f - \G(t,s) \partial^\alpha f||
\end{align*}

Let $\alpha \in \NN_{d,m}$ and $\{h_n\}_{n \in \NN} \subseteq [-s,t-s]$ any sequence such that $h_n \to 0$. Let $S_n:=L_t-L_{s+h_n}$ and $S:=L_t-L_s$. We have $S_n \stackrel{P}{\to} S$ by stochastic continuity of additive processes. By the Skorokhod's representation theorem, there exists a probability space $(\Omega', \F', \PP')$ and random variables $\{S_n'\}_{n \in \NN}$, $S'$ on it such that $S_n \stackrel{D}{=} S_n'$, $S \stackrel{D}{=} S'$ and $S_n' \stackrel{a.e.}{\to} S'$. Let $x \in \RR^d$, we therefore get:
\begin{align*}
 | \G(s+h_n,t) \partial^\alpha f(x) - \G(t,s) \partial^\alpha f(x)|&=|\EE'[\partial^\alpha f(x+S_n')-\partial^\alpha f(x+S')]|\\
 & \leq \EE'|\partial^\alpha f(x+S_n')- \partial^\alpha f(x+S')|\\
 \Rightarrow || \G(s+h_n,t) \partial^\alpha f - \G(t,s) \partial^\alpha f || &\leq \sup_{x \in \RR^d} \EE'|\partial^\alpha f(x+S_n')- \partial^\alpha f(x+S')|\\
 & \leq  \EE'[ \sup_{x \in \RR^d}  |\partial^\alpha f(x+S_n')- \partial^\alpha f(x+S')|]
\end{align*}

Further, since $\partial^\alpha f \in Y$, $\partial^\alpha f$ is uniformly continuous on $\RR^d$ and $\forall \e>0$, $\exists \delta>0$: $|x-y|<\delta \Rightarrow |\partial^\alpha f(x)-\partial^\alpha f(y)|<\e$. And because $S_n' \stackrel{a.e.}{\to} S'$, for a.e. $\omega' \in \Omega'$, $\exists N(\omega') \in \NN:$ $n \geq N(\omega') \Rightarrow |S_n'(\omega')-S'(\omega')|=|x+S_n'(\omega')-S'(\omega')-x|<\delta \hspace{4mm} \forall x \in \RR^d \Rightarrow |\partial^\alpha f(x+S_n'(\omega'))-\partial^\alpha f(x+S'(\omega'))|<\e \hspace{4mm} \forall x \in \RR^d$. Therefore we have that 
$g_n:=\sup_{x \in \RR^d}  |\partial^\alpha f(x+S_n')- \partial^\alpha f(x+S')| \stackrel{a.e.}{\to} 0$. Further $|g_n| \leq 2 ||\partial^\alpha f|| \in L^1_\RR(\Omega', \F', \PP')$. By Lebesgue dominated convergence theorem we get:
\begin{align*}
\lim_{n \to \infty} \EE'[ \sup_{x \in \RR^d}  |\partial^\alpha f(x+S_n')- \partial^\alpha f(x+S')|]=0
\end{align*}

We can notice that the proof strongly relies on the uniform continuity of $f$, and therefore on the topological properties of the space $C_0(\RR^d)$ (which $C_b(\RR^d)$ doesn't have). We prove that $\G$ is $Y-$strongly $t$-continuous exactly the same way, but now considering $Y_n:=L_{t+h_n}-L_{s}$ and any sequence $\{h_n\}_{n \in \NN} \subseteq K$ such that $h_n \to 0$, where $K:=[s-t,T_\infty-t]$ if $J=[0,T_\infty]$ and $K:=[s-t,1]$ if $J=\RR^+$.\\ 

\subsection{Proof that $\G$ is regular}

By Taylor's theorem we get $\forall f \in Y_1$, $t \in J$, $x,y \in \RR^d$:
\begin{align*}
&\left| f(x+y)-f(x)-\sum_{j=1}^d \frac{\partial f}{\partial x_j}(x) y(j) \right| \leq \frac{1}{2} |y|^2 \sum_{j,k=1}^d \left| \left| \frac{\partial^2 f}{\partial x_j \partial x_k}(x) \right| \right| \leq \frac{d^2}{2} |y|^2 ||f||_{Y_1}\\
\Rightarrow & || A_\G(t)f|| \leq ||f||_{Y_1} \left[\sum_{j=1}^d |b_t(j)| +\frac{1}{2}\sum_{j,k=1}^d |c_t(j,k)|+2 \nu_t\{|y|>1\}+\frac{d^2}{2} \int_{|y| \leq 1} |y|^2 \nu_t(dy)\right] \\
\Rightarrow & || A_\G(t) ||_{\B(Y_1,Y)} \leq \sum_{j=1}^d |b_t(j)| +\frac{1}{2}\sum_{j,k=1}^d |c_t(j,k)|+2 \nu_t\{|y|>1\}+\frac{d^2}{2} \int_{|y| \leq 1} |y|^2 \nu_t(dy)
\end{align*}

, observing that $\nu_t\{|y|>1\}+\int_{|y| \leq 1} |y|^2 \nu_t(dy)<\infty$ by assumption. Therefore by integrability assumption on the local characteristics and theorem \ref{7}, we get  the regularity of $\G$.\\

\subsection{Proof that $V_\e$ satisfies the compact containment criterion in $Y_1$}

We want to prove \ref{18}. Here we will assume - for sake of clarity - that the diffusion parts of the processes $L^x$ are driven by the same brownian motion $W$. The proof can be extended straightforwardly to the general case where the $|X|$ brownian motions $W^x$ are imperfectly correlated, expressing each one of them as a linear combination of $|X|$ independent brownian motions (Cholesky).\\

We showed $V_\e$ is a $\B(Y_1)-$contraction, so it remains to show the uniform convergence to 0 at infinity. We have the following representation, for $\oo' \in \Omega$ (to make clear that the expectation is wrt $\oo$ and not $\oo'$):
  \begin{align*}
    \begin{split}
  V_{\e}(s,t)(\oo')f(z)=\EE &\left[ f\left( z+  \int_s^t b_u(x(u^{\frac{1}{\e},s})(\oo')) du + \int_s^t \sqrt{c_u(x(u^{\frac{1}{\e},s})(\oo'))} dW_u+ \widetilde{L}_t -\widetilde{L}_s \right. \right.\\
 & \left. \left.  + \sum \limits_{k=1}^{N_s \left(t^{\frac{1}{\e},s} \right)(\oo')} \e \al(x_{k-1}(s)(\oo'),x_{k}(s)(\oo'))  \right) \right]
 \end{split}
  \end{align*}
 
 Let:
   \begin{align*}
       \begin{split}
  g_{\oo',t,\e}(z)=\EE &\left[ f\left( z+  \int_s^t b_u(x(u^{\frac{1}{\e},s})(\oo')) du + \int_s^t \sqrt{c_u(x(u^{\frac{1}{\e},s})(\oo'))} dW_u \right. \right. \\
  & \left. \left. + \sum \limits_{k=1}^{N_s \left(t^{\frac{1}{\e},s} \right)(\oo')} \e \al(x_{k-1}(s)(\oo'),x_{k}(s)(\oo'))  \right) \right]
      \end{split}
  \end{align*}
  
  and if $\NNN$ stands for the normal distribution, define:
  \begin{align*}
  & Z_1:=\int_s^t \sqrt{c_u(x(u^{\frac{1}{\e},s})(\oo'))} dW_u, \mbox{ so that } Z_1 \sim \NNN \left(0, \sigma_1^2:=\int_s^t c_u(x(u^{\frac{1}{\e},s})(\oo')) du\right)\\
  & Z_2:=\sigma (W_T-W_s)\mbox{ so that } Z_2 \sim \NNN \left(0, \sigma_2^2:= \sigma^2 (T-s)\right)
  \end{align*}
  
  Let $\delta>0$. There exists $k_\delta>0$ such that $\PP[|Z_2|>k_\delta]<\frac{\delta}{2||f||}$. Since $\sigma_1 \leq \sigma_2$, then $\PP[|Z_1|>k_\delta] \leq \PP[|Z_2|>k_\delta]<\frac{\delta}{2||f||}$. And letting $B:=\{|Z_1|>k_\delta\}$:
   \begin{align*}
  &\left| \EE \left[ 1_B f\left( z+  \int_s^t b_u(x(u^{\frac{1}{\e},s})(\oo')) du + \int_s^t \sqrt{c_u(x(u^{\frac{1}{\e},s})(\oo'))} dW_u \right. \right. \right. \\
  & \left. \left. \left. +\sum \limits_{k=1}^{N_s \left(t^{\frac{1}{\e},s} \right)(\oo')} \e \al(x_{k-1}(s)(\oo'),x_{k}(s)(\oo'))  \right) \right] \right| \leq ||f|| \PP(B)<\frac{\delta}{2}
  \end{align*}

 To define $A_\e$, we do the following: since $\lim_{t \to \infty }\frac{N(t)}{t} =\frac{1}{M_\rho}$ a.e., then in particular $\frac{N(t)}{t} \stackrel{P}{\to} \frac{1}{M_\rho}$ and therefore $\exists \e_0>0:$ 
\begin{align*}
|\e|<\e_0 \Rightarrow \PP \left[\e N_s \left(T^{\frac{1}{\e},s} \right) \leq 1+\frac{T-s}{M_\rho} \right]  \geq 1-\frac{\Delta}{2}
\end{align*}

In addition, since $\PP \left[ N_s \left(T^{\frac{1}{\e_0},s} \right) < \infty \right] =1$, and because every probability measure on a Polish space is tight (here $\RR$ is Polish), $\exists n_0 \in \NN:$ 
\begin{align*}
\PP \left[ N_s \left(T^{\frac{1}{\e_0},s} \right) \leq n_0 \right] \geq 1-\frac{\Delta}{2}
\end{align*}

so that $\forall \e \in (0,1]$, letting $n_{00} := \left(1+\frac{T-s}{M_\rho} \right) \vee n_0$:
\begin{align*}
 \PP \left[ \e N_s \left(T^{\frac{1}{\e},s} \right) \leq n_{00} \right]=\PP(A_\e) \geq1-\Delta
\end{align*}

 Note that $n_{00}$ only depends on $T$, $s$ and $ \Delta$. Now,  $\exists c_\delta>0$: $|z|>c_\delta \Rightarrow |f(z)|<\frac{\delta}{2}$.  We have for $\oo' \in A_\e$ and $|z|>C_\delta:=c_\delta+||\al|| n_{00}+k_\delta+r(T-s)$:
   \begin{align*}
  &\left| \EE \left[ 1_{B^c} f\left( z+  \int_s^t b_u(x(u^{\frac{1}{\e},s})(\oo')) du + \int_s^t \sqrt{c_u(x(u^{\frac{1}{\e},s})(\oo'))} dW_u \right. \right. \right. \\
  & \left. \left. \left. +\sum \limits_{k=1}^{N_s \left(t^{\frac{1}{\e},s} \right)(\oo')} \e \al(x_{k-1}(s)(\oo'),x_{k}(s)(\oo'))  \right) \right] \right| <  \frac{\delta}{2} \PP(B^c) \leq \frac{\delta}{2}
   \end{align*}
 
 so that $|g_{\oo',t,\e}(z)|< \delta$ for $|z|>C_\delta$ (uniform in $\oo',t,\e$). Now we have that:
 \begin{align*}
 V_{\e}(s,t)(\oo')f=\G_{\widetilde{L}}(s,t)g_{\oo',t,\e},
 \end{align*}
 
 where $\G_{\widetilde{L}}$ is the inhomogeneous semigroup corresponding to $\widetilde{L}$, so that:
  \begin{align*} 
  & V_{\e}(s,t)(\oo')f(z)=\int_{\RR} \widetilde{\mu}_{s,t}(dy) g_{\oo',t,\e}(z+y)
  \end{align*} 

By tightness of the family $\{\widetilde{\mu}_{s,t}\}_{t \in [s,T]}$, $\exists \overline{C}_\delta>0:  \widetilde{\mu}_{s,t}\{|y| > \overline{C}_\delta\}<\delta$ $\forall t \in [s,T]$. Let $c^*_\delta:=2(C_\delta \vee \overline{C}_\delta)$ and we have for $|z|>c^*_\delta$, observing that $|| g_{\oo',t,\e}|| \leq ||f||$:
 
  \begin{align*} 
   &V_{\e}(s,t)(\oo')f(z)=\int_{|y| \leq \overline{C}_\delta}  \widetilde{\mu}_{s,t}(dy) g_{\oo',t,\e}(z+y)+\int_{|y| > \overline{C}_\delta}  \widetilde{\mu}_{s,t}(dy) g_{\oo',t,\e}(z+y)\\
  &\Rightarrow | V_{\e}(s,t)(\oo')f(z)| \leq  \int_{|y| \leq \overline{C}_\delta} \widetilde{\mu}_{s,t}(dy) |g_{\oo',t,\e}(z+y)| + ||f||  \widetilde{\mu}_{s,t}\{|y| > \overline{C}_\delta\} 
   \end{align*}
 
 But $|y| \leq \overline{C}_\delta \Rightarrow |z+y| \geq |z|-|y| \geq c^*_\delta-\overline{C}_\delta \geq C_\delta \Rightarrow |g_{\oo',t,\e}(z+y)| < \delta$, and so $|V_{\e}(s,t)(\oo')f(z)| < \delta(||f|| +1)$.\\
 
 Because $c^*_\delta$ is uniform in $\omega'$, $t$ and $\e$, we have finished the proof.

\end{document}